\documentclass[a4paper, leqno]{article}
\usepackage[a4paper, total={6in, 8in}]{geometry}
\usepackage[utf8]{inputenc}
\usepackage{latexsym}
\usepackage{stmaryrd}
\usepackage{makeidx}
\usepackage{xr}
\usepackage{mathrsfs}
\usepackage{mathtools}
\usepackage{tikz-cd}
\usetikzlibrary{positioning}
\usepackage{amscd}
\usepackage{syntonly}
\usepackage{amssymb}
\usepackage{amsfonts}
\usepackage{amsmath}
\usepackage{amsthm}
\usepackage[english]{babel}
\usepackage{chngcntr}
\usepackage{yfonts}
\usepackage{xcolor}
\usepackage{natbib}
\usepackage{adjustbox}
\usepackage{quiver}
\usepackage{enumitem}
\usepackage[mathscr]{euscript}
\usepackage{hyperref}
\usepackage{comment}
\usepackage{graphicx}
\newtheorem{theorem}{Theorem}[section]

\newtheorem{definition}[theorem]{Definition}

\newtheorem{lemma}[theorem]{Lemma}
\newtheorem*{note*}{Note}
\newtheorem*{claim*}{Claim}
\newtheorem*{lemma*}{Lemma}
\newtheorem*{examples*}{example}
\newtheorem*{example*}{example}
\newtheorem*{corollary*}{Corollary}
\newtheorem*{theorem*}{Theorem}
\newtheorem*{Reminder*}{Reminder}
\newtheorem*{justification*}{Justification}
\newtheorem*{Notation Convention*}{Notation Convention}
\newcommand{\la}[1]{\mathrm{#1}}
\newcommand{\ca}[1]{\mathcal{#1}}
\newcommand{\bb}[1]{\mathbb{#1}}
\newcommand{\cat}[1]{\mathsf{#1}}
\newcommand{\go}[1]{\mathfrak{#1}}
\newcommand{\lu}[1]{\mathscr{#1}}
\newcommand{\Hom}{\mathrm{Hom}}
\newcommand{\forallw}{\forall_{\omega}}
\newcommand{\dash}{\text{-}}

\date{}
\author{Ali Hamad \footnote{Affiliation: University of Ottawa, Email: ahama099@uottawa.ca}}

\newcommand{\sqoint}{%
  \mathop{\mathchoice
    {\ooalign{\hfil$\scriptstyle\square$\hfil\cr$\displaystyle\int$\cr}}
    {\ooalign{\hfil$\scriptstyle\square$\hfil\cr$\textstyle\int$\cr}}
    {\ooalign{\hfil$\scriptscriptstyle\square$\hfil\cr$\scriptstyle\int$\cr}}
    {\ooalign{\hfil$\scriptscriptstyle\square$\hfil\cr$\scriptscriptstyle\int$\cr}}
  }\nolimits
}
\title{Generalised ultracategories and conceptual completeness of geometric logic}
\date{}

\begin{document}

\maketitle

\setcounter{secnumdepth}{2}
\setcounter{tocdepth}{2}

\begin{abstract}
    We introduce the theory of generalised ultracategories, these are relational extensions to ultracategories as defined by Lurie. Each topological space is a generalised ultracategory and these play a fundamental role in the theory of generalised ultracategories, since they form a dense subcategory. Another example of these generalised ultracategories is the category of points of a topos. In this paper, we show a conceptual completeness theorem for toposes with enough points, stating that any such topos can be reconstructed from its generalised ultracategory of points by considering the category of left ultrafunctors from its category of points to $\cat{Set}$ $\la{Lult}(\la{Points_E},\cat{Set})$ and showing it's equivalent to the topos $E$. This is done by considering left ultrafunctors from topological spaces to the category of points and paralleling this construction with another known fundamental result in topos theory, namely that any topos with enough points is a colimit of a topological groupoid. 

\end{abstract}

\tableofcontents
\section{Introduction}

As mathematical objects, toposes serve three main purposes: first, they generalise topological spaces, or more precisely they generalise the idea of topological spaces as studied in the theory of pointfree topology (the theory of frames and locales), second, they are ``nice'' universes to do mathematics. Lots of mathematical theory that is usually developed for the topos $\cat{Set}$, can be performed inside an arbitrary topos. The third point of view, is linked to the second, when it comes to interpreting a nice class of infinitary first order theories, called  geometric theories. Grothendieck toposes gain a special place through the idea of classifying topos. For every geometric theory $\bb{T}$, one can construct a topos $\ca{C}_{\bb{T}}$ satisfying the following universal property: for every other topos $E$, we have an equivalence  between $\mathrm{Geom}(E,\ca{C}_{\bb{T}})$ and $\la{Mod}_{\bb{T}}(E)$, and this equivalence is realised by sending a geometric morphism to the image of a ``universal model'' in  $\ca{C}_{\bb{T}}$ by the inverse image part of the geometric morphism to a model in $E$. This aligns in spirit with the definition of geometric logic being the fragment of infinitary first order logic preserved by the inverse image part of geometric morphisms. Following this a Grothendieck topos may be considered a geometric theory up to ``Morita equivalence''. For further discussion, the reader is invited to check  \cite{johnstone2002sketches} or \cite{caramello2018theories}.

A completeness result exists for geometric logic, it states that if a geometric sequent is true in every model of the geometric theory in every Grothendieck topos, then it's provable by a series of geometric sequents. Although such theorem is important, it presents the challenge of working in toposes other than $\cat{Set}$. Classical completeness result (completeness involving only $\cat{Set}$ models) exists for smaller fragments of geometric logic like regular logic or coherent logic (finitary geometric logic) \cite{johnstone2002sketches}.

Makkai's conceptual completeness is a result in the same spirit of the classical completeness result of coherent logic: Suppose that we have two coherent theories $\bb{T}_1$ and $\bb{T}_2$. In categorical logic, we would like to work with these using the technique of functorial semantics, this can be done by constructing pretoposes $\ca{P}(\bb{T}_1)$ and $\ca{P}(\bb{T}_2)$ (these can be regarded as ``completions'' of the syntactic category  of these theories \cite[Chapter 8]{reyesmichael}, characterised by the fact that for any coherent category $E$ and any pretopos $\bb{T}$, we have an equivalence between $\mathrm{Coherent \dash Functors}(\ca{P}(\bb{T}),E)$ and $\mathrm{Mod}_{E}(\bb{T})$), then conceptual completeness tells us that if  we have two pretoposes with a map $f$ of pretoposes between them, such that their categories of $\mathsf{Set}$ models, are equivalent via a $\dash\circ f$ , then the two pretoposes must be equivalent via $f$, and hence the coherent theories must be Morita equivalent. Makkai's conceptual completeness has another aspect, a reconstruction result stating that we may reconstruct the pretopos from its category of models, towards this we need an additional ingredient on the category of models which axiomatises the idea of category having ultraproducts. Towards this Makkai introduced his notion of ultracategories, and ultrafunctors. Makkai's ultracategories serve as a first axiomatisation of the categorification of the ultraproduct construction, so an ultracategory is a category with ultraproduct functor(s) satisfying coherence, and an ultrafunctor  is a functor that respects this additional structure on a category, and hence Makkai's conceptual completeness can be stated as:

\begin{theorem*}[Makkai]
    Let $T$ be a small pretopos, then the evaluation functor from $T$ to $\mathrm{Ult}(\la{Mod}_{\cat{Set}}(T), \cat{Set})$ is an equivalence of categories.
\end{theorem*}

\begin{note*} \normalfont
    \normalfont Small pretoposes are those constructed as pretopos completion of the syntactic category of a coherent theory.
\end{note*}

Lurie's version of conceptual completeness is similar, the difference is that functorial semantics is built using the topos theoretic results which we already introduced. Lurie showed that it's possible to reconstruct a coherent topos (a topos that classifies a coherent first-order theory), from its category of points (models of the coherent theory this topos classifies) with the additional structure that axiomatises ultraproducts. Towards showing his theorem, Lurie reintroduced ultracategories, and introduced a new class of functors between them (left ultrafunctors), and hence he obtained the following theorem:

\begin{theorem*}[Lurie]
    Let $E$ be a coherent topos. Then there is an equivalence of categories between $\mathrm{Lult}(\mathrm{Points}_E,\cat{Set})$ and $E$.
\end{theorem*}

In a paper \cite{hamad2025ultracategories}, we have shown that ultracategories are normal colax-algebras for a pseudo-monad $T$ on the category of categories, this monad was introduced independently by Rosolini and Garner in an unpublished work to this (for slides of a presentation by Rosolini, see \cite{ultracompletion}).

The goal of this paper is to extend Lurie's result to a larger class of toposes, namely toposes with enough points, these are toposes that have a small subcategory of $\cat{Set}$-models, which we can check completeness against. We want to use the ultraproduct construction to develop a similar result to Lurie's for such toposes. The problem is that unlike coherent toposes, the category of models does not possess a canonical notion of ultraproducts, and hence we need a generalisation.

In this paper, we introduce the notion of generalised ultracategories. These should be regarded as a relational version of Lurie's ultracategories. To be more precise, these are generalisations of $T$-pseudo algebras. Introducing them serves many purposes:

First, it allows the treatment of a larger class of mathematical objects. A result by Lurie states an equivalence between ultracategories with no non-identity morphisms (which he calls ultrasets) and compact Hausdorff spaces, this equivalence is established by noticing that in this case, the ultracategory axioms axiomatise convergence of ultrafilters. Building on this idea, we defined ultrapreorders \ref{ultapreorder}, and we have used the generalised ultrastructure to encode ultrafilter convergence, and this allows a fully faithful embedding of topological spaces inside the $2$-category of generalised ultracategories.

Second, this is the natural setting to show conceptual completeness for toposes with enough points. The category of points of any topos possesses a natural generalised ultrastructure, and hence we may state our main theorem \ref{Main theorem}:

\begin{theorem*}
       Let $E$ and $E^{'}$ be two toposes with enough points, and let $M_E$ and $M_{E^{'}}$ denote their respective generalised  ultracategories of points. Then there is an equivalence of categories between $\mathrm{Lult}(M_{E}, M_{E^{'}})$ and $\mathrm{Geom}(E,E^{'})$.
\end{theorem*}

By replacing the topos $E^{'}$ by the classifying topos of the theory of objects $\ca{S}[\bb{O}]$, one can deduce a result similar to Lurie's, namely 

\begin{theorem*}
    Any topos with enough points $E$ is equivalent to left ultrafunctors from its category of points $M_E$ to $\cat{Set}$.
\end{theorem*}

This provides an alternative proof of Lurie's result in the case where the topos is coherent.

\begin{note*} \normalfont A similar result to ours was independently obtained in \cite{saadia2025extendingconceptualcompletenessvirtual} and \cite{VANGOOL2026108269} . The author of the first paper introduced the notion of virtual ultracategories, while the authors of the second introduced ultraconvergence spaces. These are all equivalent to lax algebras of an extension of the pseudo-monad $T$ to $\cat{Prof}$ \cite{aristote2026profunctorialalgebras}.
    
\end{note*}

\subsection{Structure of paper and methodology}

In \ref{axioms-first section}, we introduce an axiomatic treatment of generalised ultracategories. These are categories for which ultraproducts may not be defined, but nonetheless we have ``representables'' at the ultraproduct, more formally for every object $A$ and every family of objects $(M_i)_{i \in I}$ and every ultrafilter $\mu$ on $I$ we have a generalised $\Hom$-set:

\[\Hom \left( A,\int_I M_i d\mu \right)\]

This is very similar in spirit to the axiomatic treatment of Lurie but the reader may notice that the generalised ultrastructure and its axioms that we defined, may be regarded also as a relational version of the $T$-colax algebra structure that Lurie ultracategories possess \cite{aristote2026profunctorialalgebras}. The axioms are self-contained and other than knowing basic constructions used in \cite{lurie2018ultracategories} and \cite{hamad2025ultracategories}, the reader is not required to know details of the construction of the pseudo-monad $T$ or the equivalence above, to understand and start working with generalised ultracategories (although we have used these to show some categorical results regarding generalised ultracategories, the reader is welcome to accept them on faith if they do not want to read in detail the paper \cite{hamad2025ultracategories}).

In sections \ref{1-morphism} and \ref{2-morphisms}, we introduce the necessary definitions needed to define a $2$-category of generalised ultracategories, moreover, we show that these are indeed functors, and natural transformations respectively, but we have left these to appendices \hyperref[not very crucial proof]{Appendix B} and \hyperref[Equaivalence of different notions of left ultrafunctors]{Appendix C}.

In section \ref{examples of generalised}, we started by introducing the first examples of generalised ultracategories. These are ultrapreorders which are generalised ultracategories for which every generalised $\la{Hom}$-set is a singleton or empty, and all ``change of base'' maps are isomorphisms. These turn out to be equivalent to the $2$-category of topological spaces where $2$-morphisms express specialisation preorder between continuous functions, and the generalised $\la{Hom}$-sets encode ultrafilter convergence, this is an adaptation of Barr's result stating that topological spaces are relational algebras for the ultrafilter monad \cite{barr2006relational}, to the setting of generalised ultracategories. The next and most important example of generalised ultracategories is that of points of toposes. The idea we used is simple but effective, which is the following theorem \ref{not very crucial theorem}:

\begin{theorem*}
Suppose $\ca{B}$ is   a full subcategory of an ultracategory which is a pseudo-algebra for the pseudo-monad $T$ (i.e. a strong ultracategory), then there is a canonical generalised ultrastructure on $\ca{A}$.
 \end{theorem*}

We have a similar theorem that establishes, for ultracategories which are $T$-pseudo algebras (which we will be calling strong ultracategories), an equivalence between left ultrafunctors between them seen as usual or generalised ultracategories. Together, this shows a fully faithful embedding of the category of $T$-pseudo algebras (strong ultracategories) inside the category of generalised ultracategories, this was done in \hyperref[not very crucial proof]{Appendix B} and \hyperref[Equaivalence of different notions of left ultrafunctors]{Appendix C} .

Via this theorem, we may define the generalised ultrastructure, of the category of points, by taking the category of points as a full subcategory of the category $\mathrm{Fun}(C,\mathsf{Set})$, where $C$ is any site of definition of this topos.

It turns out that topological spaces (viewed as ultrapreorders) play a fundamental role in the theory of generalised ultracategories, we want to understand how these topological spaces ``cover'' generalised ultracategories. The tool that we use, would be considering the lax and pseudo-slice categories $\cat{Top}//C$ and $\cat{Top}/C$, for a generalised ultracategory $C$, this serves multiple purposes: first we would like to see how can we reconstruct $C$ from the category $\cat{Top}//C$. Second, we would like to compare $\cat{Top}//C$ and $\cat{Top}//C^{'}$ for two different generalised ultracategories $C$ and $C^{'}$. It turns out that the right setting is to consider $\cat{Top}//E$ as categories $2$-fibred over the $2$-category of topological spaces. This setting is also convenient to consider for toposes, suppose that we have a topos $E$ it is also possible to consider the category $\cat{Top}//E$, by the inclusion of (Sober) topological spaces inside toposes. And hence, we would like also to compare $\cat{Top}//E$ and $\cat{Top}//M_E$, where $M_E$ is the category of points of $E$. But before this, we wanted to introduce the theory of $2$-fibred categories as done by \cite{buckley2014fibred}, and we do this in \ref{2-Fibrations}.

After that we start showing our main theorem, which states that there is a fully faithful embedding of the $2$-category of toposes with enough points, inside the $2$-category of generalised ultracategories via the pseudofunctor that sends a topos with enough point to its generalised ultracategory of points.

The proof is a corollary of the following three theorems:

Starting with the first theorem \ref{first equivalence theorem}
\begin{theorem*}
    Let $E$ and $E^{'}$ be toposes, and let $M_E$ and $M_{E^{'}}$ be their respective categories of points, then there is an equivalence of categories between $\la{Lult}(M_E, M_{E^{'}})$ and $\la{ClovenCart}(\cat{Top}//M_E,\cat{Top}//M_{E^{'}})$.
\end{theorem*}

Here $\la{ClovenCart}(\cat{Top}//M_E,\cat{Top}//M_{E^{'}})$ are the functors between these lax-slice categories that respect the fibration aspect over the category of topological spaces. The proof relies on constructing nice enough topological spaces that can realise generalised  morphisms in generalised ultracategories, as ultrafunctors from these topological spaces. This technique can be used to show that the category of topological spaces is dense inside the $2$-category of generalised ultracategories.

The second theorem we use to show is the following theorem \ref{second equivalence theorem}:

\begin{theorem*}
    Let $E$ be a topos with enough points, then there is an equivalence of discrete $2$-opfibrations over $\cat{Top}$ between $\cat{Top}//M_E$ and $\cat{Top}//E$ .
\end{theorem*}

This may be the easiest to show part of the equivalence. We show the fact that for any topological space $X$, there is an equivalence of categories between $Sh(X)$ and $\mathrm{Lult}(X,\cat{Set})$. This can be done by constructing an étale bundle whose fibres are the images of the points of $X$, by the left ultrafunctor. Then the entire argument can be done by exponentiation swap, the detailed argument is in \ref{second equivalence}.

The third theorem in our proof is the theorem \ref{third equivalence theorem}:

\begin{theorem*}
    Let $E$ and $E^{'}$ be two toposes with enough points, then there is an equivalence of categories between $\mathrm{Geom}(E,E^{'})$ and $\mathrm{ClovenCart}(\cat{Top}//E, \cat{Top}//E^{'})$.
\end{theorem*}

In order to show this theorem we use a representability theorem by Moerdijk and Butz stating that every topos with enough points is a colimit of a topological groupoid\cite{butz1997representation}. We use this to show an important intermediate result stating that for any topos with enough points the $2$-colimit of the ``forgetful'' functor from $\cat{Top}/E$ to $\cat{Topos}$ is $E$. The entire proof is done in subsection \ref{third equivalence}.

Combining these three equivalences together, one may deduce our main theorem, namely an equivalence of categories between $\la{Geom}(E,E')$  and $\la{Lult}(\la{Points}_E,\la{Points}_{E'})$.

\section{The axioms} \label{axioms-first section}
In this paper we refer to ultracategories as defined in \cite{lurie2018ultracategories}. In \cite{hamad2025ultracategories} we have shown that ultracategories are equivalent to normal colax algebras for pseudo-monad $T$ on the category of categories. An important construction of that paper is what we have called colax associators which we denote by $a_{(S,\eta,(T_s)_{s \in S}, (M_{(s,t)})_{(s,t) \in \coprod_{s \in S} T_s},(\nu_s)_{s \in S} )}$:

\[
\begin{tikzcd}
{\int_{\coprod_{s \in S}T_s}M_{(s,t)}d \int_S \iota_S \gamma_S d \mu } \arrow[rr, "{\Delta_{\iota \gamma_{\bullet},\mu}}"] &  & {\int_S\int_{\coprod_{s \in S}T_s} M_{(s,t)} d \iota_s \gamma_s d\mu} \arrow[rr, "{\int_T \Delta_{\mu,\iota_s}}"] &  & \int_S\int_{T_s} M_{(s,t)} d \gamma_s d\mu
\end{tikzcd}\]

Here, the first $\Delta$ is the so called categorical Fubini transform in \cite{lurie2018ultracategories}, and the second $\Delta$ is the so called ultraproduct diagonal map in \cite{lurie2018ultracategories}.

We will be referring to this morphism as simply $a$ when no confusion is to be made.

A generalised ultracategory $\ca{C}$ consists of the following data:
\begin{itemize}
    \item A ``set'' of objects $\mathrm{Obj}(\ca{C})$.
    \item For every object $A$ and every family of objects $(M_i)_{i \in I}$ and every ultrafilter $\mu$ on $I$ we have a set which we call $\la{Hom}(A, \int_I M_i d \mu)$, this set can be regarded as a generalised $\la{Hom}$ set.
    
    \item Given sets $I$, an ultrafilter $\mu$ on $I$, a family of sets $(X_i)_{i \in I}$ and for each $X_i$ an ultrafilter $\lambda_i$ on $X_i$, and a family of objects $(N_{(i,x)})_{x \in X_{i} }$ and finally a family of objects $(M_i)_{i \in I}$, then we have a function $$\beta_{A, (M_i) , (N_{(i,x)})_{x \in X_i}}: \la{Hom}\left(A , \int_I M_i d  \mu\right) \times \int_{I} \la{Hom}\left(M_{i}, \int_{X_i}N_{(i,x)}d \lambda_i\right) d \mu \longrightarrow  \la{Hom}\left(A,\int_{\coprod_{I}X_i}N_{(i,x)}d\int_I \iota_i \lambda_i d\mu\right)$$ Here $\iota_i$ is the inclusion map of $X_i$ inside $\coprod_{i \in I} X_i$.

    \item Given a set $J$ and an ultrafilter $\mu$ on $J$ and a map of sets $f$ from $J$ to $I$, we have a map of sets $\Xi_{f,\mu}$ from $\la{Hom}(A,\int_I M_id f\mu)$ to $\la{Hom}(A,\int_JM_{f(j)}d\mu)$ (think of this as the representable at the ultraproduct diagonal map in case of usual ultracategories).

    \item For every object $A$ a distinguished element $\kappa_A \in \la{Hom}(A, \int_{*} A d *)$ (which is equivalent to having a morphism from $\bf{1}$ to $\la{Hom}(A, \int_{*} A d *)$). Here $*$ simultaneously denotes the one-point set and at the same time the unique ultrafilter on it.
\end{itemize}

We require this data to satisfy the following axioms:

\begin{enumerate}

\item We require that $\Xi_{\la{id},\mu}= \la{id}$.

\item We require that $\Xi_{f,\mu}$ be an isomorphism when $f$ is injective.

\item We require $\Xi$ to be functorial which means that $\Xi_{f \circ g}=\Xi_{g}  \circ \Xi_{f}$.

\item  Left unit axiom, which can be expressed by the fact that the following composition is the inverse of the $\Xi_{f_{I},\mu}$, where $f_I$ is the natural isomorphism between $I$ and $\coprod_{*}I$:

\[
\begin{tikzcd}[row sep=3em, every label/.append style={font=\small}]
	{\mathrm{Hom}(A,\int_I M_i d\mu)} \\
	{\int_{*}\mathrm{Hom}(A,\int_I M_i d\mu)d*} \\
	{\mathbf{1} \times \int_{*}\mathrm{Hom}(A,\int_I M_i d\mu )d*} \\
	{\mathrm{Hom}(A, \int_\ast A d \ast ) \times \int_{\ast}\mathrm{Hom}(A,\int_I M_i d\mu )d*} \\
	{\mathrm{Hom}(A,\int_{\coprod_*I} M_i d\mu )}
	\arrow["{\epsilon^{-1}_{*,*}}", from=1-1, to=2-1]
	\arrow["\simeq", from=2-1, to=3-1]
	\arrow["{\kappa \times \mathrm{id}}", from=3-1, to=4-1]
	\arrow["\beta", from=4-1, to=5-1]
\end{tikzcd}
\]
\item  Right unit axiom which can be expressed as the fact that the following composition is the inverse of $\Xi_{g_{I},\mu}$ where $g_{I}$ is the natural isomorphism between $I$ and $\coprod_{I}*$.

\[
\begin{tikzcd}[row sep=3em, every label/.append style={font=\small}]
	{\mathrm{Hom}(A,\int_I M_i d\mu)} \\
	{\mathrm{Hom}(A,\int_I M_i d\mu) \times 1} \\
	{\mathrm{Hom}(A,\int_I M_i d\mu)\times \int_{I}\mathbf{1}d\mu} \\
	{\mathrm{Hom}(A,\int_I M_i d\mu )\times \int_I\mathrm{Hom}(M_{i} , \int_{*}M_id*)d\mu} \\
	{\mathrm{Hom}(A,\int_{\coprod_{I}*}M_id\mu)}
	\arrow["{\mathrm{id} \times \delta_{\mu}}", from=2-1, to=3-1]
	\arrow["\simeq", from=1-1, to=2-1]
	\arrow["{\mathrm{id} \times \int_I \kappa d\mu}", from=3-1, to=4-1]
	\arrow["\beta", from=4-1, to=5-1]
\end{tikzcd}
\]
Here $\delta_{\mu}$ is (in the case of sets) the isomorphism between $\mathbf{1}$ and $\int_{I}\mathbf{1}d\mu$.

%\item 

%\adjustbox{max width = \textwidth}{

% https://q.uiver.app/#q=WzAsNixbOCwwLCJcXG1hdGhybXtIb219KEEsIFxcaW50X0kgTV97aX0gZFxcbXUpIFxcdGltZXMgXFxpbnRfSlxcbWF0aHJte0hvbX0oTV97ZihqKX0sIFxcaW50XypNX3tmKGopfSBkKilkXFxtdSJdLFsxMSwwLCJcXG1hdGhybXtIb219KEEsIFxcaW50X3tcXGNvcHJvZF9KICp9IE1fe2Yoail9IGRcXG11KSJdLFsxNCwwLCJcXG1hdGhybXtIb219KEEsICBcXGludF9KTV97ZihqKX0gZFxcbXUpIl0sWzUsMCwiXFxtYXRocm17SG9tfShBLCBcXGludF9JIE1faSBkZlxcbXUpIFxcdGltZXMgXFxpbnRfSiBcXG1hdGhiZnsxfWRcXG11Il0sWzIsMCwiXFxtYXRocm17SG9tfShBLCBcXGludF9JIE1faSBkZlxcbXUpIFxcdGltZXMgXFxtYXRoYmZ7MX0iXSxbMCwwLCJcXG1hdGhybXtIb219KEEsXFxpbnRfSSBNX2kgZCBmIFxcbXUpIl0sWzAsMSwiXFxiZXRhIl0sWzEsMiwiXFxnYW1tYV97QSwoTV9pKV97aSBcXGluIEl9LFxcbXV9XnstMX0iXSxbMywwLCJcXG1hdGhybXtpZH0gXFx0aW1lcyBcXGludF9JIFxca2FwcGEgZFxcbXUiXSxbNCwzLCJcXG1hdGhybXtpZH0gXFx0aW1lcyBcXGRlbHRhX3tcXG11fSJdLFs1LDQsIlxcc2ltZXEgIl1d
%\begin{tikzcd}
	%{\mathrm{Hom}(A,\int_I M_i d f \mu)} && {\mathrm{Hom}(A, \int_I M_i df\mu) \times \mathbf{1}} &&& {\mathrm{Hom}(A, \int_I M_i df\mu) \times \int_J \mathbf{1}d\mu} &&& {\mathrm{Hom}(A, \int_I M_{i} d\mu) \times \int_J\mathrm{Hom}(M_{f(j)}, \int_*M_{f(j)} d*)d\mu} &&& {\mathrm{Hom}(A, \int_{\coprod_J *} M_{f(j)} d\mu)} &&& {\mathrm{Hom}(A,  \int_JM_{f(j)} d\mu)}
	%\arrow["\beta", from=1-9, to=1-12]f
	%\arrow["{\gamma_{A,(M_i)_{i \in I},\mu}^{-1}}", from=1-12, to=1-15]
	%\arrow["{\mathrm{id} \times \int_I \kappa d\mu}", from=1-6, to=1-9]
	%\arrow["{\mathrm{id} \times \delta_{\mu}}", from=1-3, to=1-6]
	%\arrow["{\simeq }", from=1-1, to=1-3]
%\end{tikzcd}
%}

\item We want  compatibility axioms between $\beta$ and $\Xi$ which can be stated as the commutativity of the following diagrams:

    \adjustbox{max width =\textwidth}{
% https://q.uiver.app/#q=WzAsNixbMCwwLCJcXEhvbShBLFxcaW50X0lNX2lkXFxtdSkgXFx0aW1lcyBcXGludF9JIFxcSG9tKE1fe2l9LFxcaW50X3tYX2l9Tl97KGkseCl9ZGZfaVxcbGFtYmRhX2kpIGRcXG11Il0sWzQsMCwiXFxIb20oQSxcXGludF97XFxjb3Byb2Rfe0l9IFhfaX1OX3soaSx4KX1kXFxpbnRcXGlvdGFfaWZfaVxcbGFtYmRhX2lkXFxtdSkiXSxbMCwzLCJcXEhvbShBLFxcaW50X0lNX2kgZFxcbXUpIFxcdGltZXMgXFxpbnRfSVxcSG9tKE1fe2l9LFxcaW50X3tLX2l9Tl97KGksZl9pKHgpKX1kXFxsYW1iZGFfaSlkXFxtdSJdLFs0LDMsIlxcSG9tKEEsXFxpbnRfe1xcY29wcm9kX3tJfSBLX2l9Tl97KGksZl9pKHgpKX1kXFxpbnRcXGlvdGFfaVxcbGFtYmRhX2lkXFxtdSkiXSxbMCwxXSxbMSwxXSxbMCwxLCJcXGJldGEiXSxbMiwzLCJcXGJldGEiLDJdLFsxLDMsIlxcWGlfe1xcYmFye2Z9fSJdLFswLDIsIlxcbGF7aWR9IFxcdGltZXMgXFxpbnRcXFhpX3tmX2ksXFxsYW1iZGFfaX0iLDJdXQ==
\begin{tikzcd}
	{\Hom(A,\int_IM_id\mu) \times \int_I \Hom(M_{i},\int_{X_i}N_{(i,x)}df_i\lambda_i) d\mu} &&&& {\Hom(A,\int_{\coprod_{I} X_i}N_{(i,x)}d\int\iota_if_i\lambda_id\mu)} \\
	{} & {} \\
	\\
	{\Hom(A,\int_IM_i d\mu) \times \int_I\Hom(M_{i},\int_{K_i}N_{(i,f_i(x))}d\lambda_i)d\mu} &&&& {\Hom(A,\int_{\coprod_{I} K_i}N_{(i,f_i(x))}d\int\iota_i\lambda_id\mu)}
	\arrow["\beta", from=1-1, to=1-5]
	\arrow["{\la{id} \times \int\Xi_{f_i,\lambda_i}}"', from=1-1, to=4-1]
	\arrow["{\Xi_{\bar{f}}}", from=1-5, to=4-5]
	\arrow["\beta"', from=4-1, to=4-5]
\end{tikzcd}
}

Here $\bar{f}$ is the map defined by  $\bar{f}_{K_i}=f_{i}$

 \adjustbox{scale=0.56}{
\begin{tikzcd}
	{\Hom(A, \int_IM_i df\mu) \times \int_I \Hom(M_{i},\int_{X_i}N_{(i,k)}d\lambda_{i})df\mu} &&&&&& {\Hom(A,\int_{\coprod_{I}X_i}N_{(i,x)}d\int_{I}\iota_i\lambda_idf\mu)} \\
	\\
	\\
	\\
	\\
	\\
	\\
	{\Hom(A, \int_IM_i df\mu) \times \int_J \Hom(M_{f(j)},\int_{X_{f(j)}}N_{(f(j),k)}d\lambda_{f(j)})d\mu} &&& {\Hom(A, \int_JM_{f(j)} d\mu) \times \int_J \Hom(M_{f(j)},\int_{X_{f(j)}}N_{(f(j),k)}d\lambda_{f(j)})d\mu} &&& {\Hom(A,\int_{\coprod_{J}X_{f(j)}}N_{(f(j),x)}d\int_{J}\iota_j\lambda_{f(j)}d\mu)}
	\arrow["\beta"', from=1-1, to=1-7]
	\arrow["{\la{id} \times \Delta_{f,\mu}}", from=1-1, to=8-1]
	\arrow["{\Xi_{\tilde{f},\int_J\iota_j \lambda_{f(j)}d\mu}}"', from=1-7, to=8-7]
	\arrow["{\Xi_{f,\mu} \times \la{id}}", from=8-1, to=8-4]
	\arrow["\beta", from=8-4, to=8-7]
\end{tikzcd}
}

Here $\tilde{f}$ is the map defined by $\tilde{f}(i,a) = (f(i),a) $.

Finally, we require a composition axiom, which can be stated as follows:

\item composition axiom: Suppose that we have a set $I$ and an ultrafilter $\mu$ on $I$ and a family of sets $(X_i)$, moreover we have an ultrafilter $\lambda_i$ on each $X_i$, moreover for every $(x,i) \in \coprod_{I} X_i$ we have family of sets $T_{(i,x)}$ and an ultrafilter $\omega_{(i,x)}$ on each such set, and for every $(i,x,t) \in \coprod T_{(i,x)} $ we have an object of $\ca{C}$ $L_{(i,x,t)}$, additionally we have for each $(i,x) \in \coprod_{I} X_{i} $  an object $N_{(i,x)}$, and a family of objects $(M_i)_{i \in I}$ and an object $A$.

Let $\theta$ be the natural isomorphism between $\coprod_{I}\coprod_{X_i}T_{(i,x)}$ and $\coprod_{\coprod_{I}X_i}T_{(i,x)}$.

Then the axiom can be expressed as the commutativity of the following diagram:
\[
\adjustbox{scale=0.65}{
\begin{tikzcd}
	& {} \\
	\\
	& {} \\
	{\mathrm{Hom}(A,\int_IM_id\mu) \times \int_{I}(\mathrm{Hom}(M_{i},\int_{X_{i}}N_{(i,x)}d\lambda_{i}) \times\int_{X_i}\mathrm{Hom}(N_{(i,x)}, \int_{T_{(i,x)}}L_{(i,x,t)}d\omega_{(i,x)})) d\mu} & {} & {\mathrm{Hom}(A,\int_IM_id\mu)\times \int_I\mathrm{Hom}(M_{i}, \int_{\coprod T_{(i,x)}}L_{(i,x,t)}d \int_{X_i}\iota_x \omega_{(i,x)}d\lambda_i)d\mu } \\
	\\
	\\
	{\mathrm{Hom}(A,\int_IM_id\mu) \times \int_I\mathrm{Hom}(M_{i},\int_{X_i}N_{(i,x)}d\lambda_{i}) \times\int_I \int_{X_i}\mathrm{Hom}(N_{(i,x)}, \int_{T_{(i,x)}}L_{(i,x,t)}d\omega_{i,x})d\lambda_{i} d\mu} \\
	&& {\mathrm{Hom}(A, \int_{\coprod_{I}\coprod_{X_i}T_{(i,x)}}L_{(i,x,t)}d\int_{}\omega_{(i,x)}d_{}\int_I\iota_i\lambda_id\mu)} \\
	\\
	{ \mathrm{Hom}(A,\int_{\coprod_{i}X_{i}}N_{i,x} d\int_{I}\iota_i\lambda_id\mu)\times \int_I \int_{X_i}\mathrm{Hom}(N_{(i,x)}, \int_{T_{(i,x)}}L_{(i,x,t)}d\omega_{(i,x)})d\lambda_i d\mu} \\
	\\
	{\mathrm{Hom}(A,\int_{\coprod_{I}X_{i}}N_{(i,x)}d\int_{I}\iota_i\lambda_id\mu)\times\int_{\coprod_{I}X_i}\mathrm{Hom}(N_{(i,x)},\int_{T_{(i,x)}}L_{(i,x,t)}d\omega_{(i,x)})d\int_{I}\iota_i\lambda_id\mu} & {} & {\mathrm{Hom}(A, \int_{\coprod_{\coprod_{I}X_i}T_{(i,x)}}L_{(i,x,t)}d\int_{}\omega_{(i,x)}d_{}\int_I\iota_i\lambda_id\mu)}
	\arrow["{\mathrm{id} \times \int\beta}"', from=4-1, to=4-3]
	\arrow["\beta"', from=4-3, to=8-3]
	\arrow["{\mathrm{commutativity \ of \  products \  with \ directed \  colimits \ in \ the \ category \ \mathsf{Set}}}"', from=7-1, to=4-1]
	\arrow["{\beta \times \mathrm{id}}", from=7-1, to=10-1]
	\arrow["{\Xi_{\theta}}"', from=8-3, to=12-3]
	\arrow["{\mathrm{id} \times  a^{-1}}", from=10-1, to=12-1]
	\arrow["\beta", from=12-1, to=12-3]
\end{tikzcd}
}
\]

$a$ here denotes the colax associator, which is invertible in $\cat{Set}$ ($\cat{Set}$ is a pseudo-algebra for the pseudo-monad that we introduced in \cite{hamad2025ultracategories}).
\end{enumerate}

\begin{note*}
    \normalfont 
    
    \begin{comment}{If the reader wants to identify a set $I$ with $\coprod){*}I$ and $\coprod_{I}*$, and the set $\coprod_{\copord_{z \in Z}T_z}I_{(z,t)}$ and $\coprod_{z\in Z} \coprod_{t \in T_z}I_{(z,t)}$ (this with some work can be made formal). 
\end{comment}

    The  composition and unit axioms can be written in classic category composition style axioms:

    Let us first denote by $ (g_i) \circ_{\mu,(\lambda_{i})} f$ the composition $\beta(f,(g_i))$ 

    \begin{itemize}
        \item  The unit axioms can be written as $ g \circ_{*,\mu} \kappa = \tilde{g}$ ( Here $\tilde{g}$ is the image of $g$ by the change of base from $I$ to $\coprod_{*}I$) and $(\kappa)_{i \in I} \circ_{\mu,*} g = \bar{g}$ ( $\bar{g}$ is the image of $g$ by the change of base between $I$ and $\coprod_{I}*)$.
        \item The composition axioms can be written as \[(f_{(i,t)}) \circ_{\lambda_i,(\nu_{(i,t)})} h_{i}) \circ_{\mu, (\lambda_i)} g = \underline{f_{(i,t)} \circ_{\lambda_i,(\nu_{(i,t)})} ((h_i) \circ_{\mu,(\lambda_i)} g)} \], here $\underline{---}$ is the change of base map from  $\coprod_{\coprod_{i \in I}T_i}Z_{(i,t)}$ to $\coprod_{i\in I} \coprod_{t \in T_i}Z_{i,t)}$.
    \end{itemize}
\end{note*}

\paragraph{Enriched generalised ultracategories} 

More generally, we can define an enriched version of generalised ultracategories. As we have noticed from the definition the category which we enrich over should be an ultracategory, moreover its colax algebra structure associator should be invertible. One should notice that the category of locally small ultracategories (or those which are pseudo-algebras for $T$) has products (product as categories + coordinate-wise ultraproduct). This motivates the following definition:

\begin{definition}[Monoidal ultracategory]
    A monoidal ultracategory is a pseudo-monoid in the category of locally small pseudo-algebras for the pseudo-monad $T$ (which we introduced in \cite{hamad2025ultracategories}), with left ultrafunctors, and natural transformations of left ultrafunctors.
\end{definition}

This builds on the definition of monoidal categories as pseudo-monoids in the category of categories, with products as monoidal functor. Unpacking this definition, this means that we have an ultracategory for which the colax associator is invertible, which is at the same time a monoidal category, for which both the monoidal structure bifunctor and the unit functor (from the terminal ultracategory) are left ultrafunctors, and for which all this structure satisfies nice coherence.

Many classical examples of monoidal categories such as abelian groups have this nice property (with the usual ultrastructure of abelian groups). We leave to the reader checking that this setting is convenient to enrich over. The reader should notice that one arrow we have, in the $\cat{Set}$ case is called ``commutativity of products with directed colimits in the category  $\cat{Set}$''. In a general enriched setting this arrow no longer is invertible.

In this paper, we restrict our attention to  $\cat{Set}$-enriched generalised ultracategories, and we are no longer going to refer to the general enriched version. 

\subsection{The underlying category of a generalised ultracategory}

We are going to denote by $\la{Hom}$ the formal Hom in the generalised ultracategory and by $\cat{Hom}$ (using sans serif font) what would be the $\la{Hom}$-set in the underlying category.

We define $\cat{Hom}(A,B)$ by $\la{Hom}(A, \int_* B d*) $, here $*$ denotes at the same time the one element set and the unique ultrafilter on it, and we define composition as follows :

\[
\adjustbox{scale =0.89}{
\begin{tikzcd}[row sep=3em, every label/.append style={font=\small}]
	{\Hom(A,\int_* B d*) \times \Hom(B,\int_*Cd*)} \\
	{\Hom(A,\int_* B d*) \times \int_*\Hom(B,\int_*Cd*)d*} \\
	{\Hom(A,\int_{\coprod_**}Cd*)} \\
	{\Hom(A,\int_{*}Cd*)}
	\arrow["{\mathrm{id} \times \epsilon^{-1}}", from=1-1, to=2-1]
	\arrow["\beta", from=2-1, to=3-1]
	\arrow["{\Xi_{f_*}}", from=3-1, to=4-1]
\end{tikzcd}}
\]

And we define units to be the elements $\kappa$ already defined.

Now, we need to verify that this construction really defines a category, we leave this to the appendix \hyperref[underlying]{Appendix A}.

\section{1-morphisms} \label{1-morphism}
Now, we define morphisms (left ultrafunctors) between generalised ultracategories  as follows: suppose $\ca{C}$ and $\ca{C}^{'}$ are generalised ultracategories, then the data of a morphism between them is:

\begin{itemize}

\item A function $F$ between $\la{Obj}(\ca{C})$ and $\la{Obj}(\ca{C}^{'})$.

\item for every object $A$ and every family of objects  $(M_i)_{i \in I}$ and every ultrafilter $\mu$ on $I$, we have a morphism $\zeta$ in $\cat{Set}$ between $\la{Hom}(A,\int_I M_i d\mu)$ and $\la{Hom}(F(A),\int_I F(M_i) d\mu)$  such that we have the  following compatibility axioms:

\end{itemize}
Suppose that we are in the same setting as the definition of the maps $\beta$ and $\kappa$, then the following diagrams commute:

\begin{enumerate}

    \item 
\[
\adjustbox{max width = \textwidth}{

\begin{tikzcd}
	{\mathrm{Hom}(A , \int_I M_i d \mu) \times \int_{I} \mathrm{Hom}(M_{i}, \int_{X_i}N_{(i,x)}d \lambda_i) d \mu  } &&&& {\mathrm{Hom}(A,\int_{\coprod_{I}X_i}N_{(i,x)}d\int_I \iota_i \lambda_i d\mu )} \\
	\\
	\\
	{\mathrm{Hom}(F(A) , \int_I F(M_i) d  \mu) \times \int_{I} \mathrm{Hom}(F(M_{i}), \int_{X_i}F(N_{(i,x)})d \lambda_i) d \mu  } &&&& {\mathrm{Hom}(F(A),\int_{\coprod_{I}X_i}F(N_{(i,x)})d\int_I \iota_i \lambda_i d\mu )}
	\arrow["\beta"', from=1-1, to=1-5]
	\arrow["{\zeta_{} \times \int\zeta d \mu}", from=1-1, to=4-1]
	\arrow["\zeta"', from=1-5, to=4-5]
	\arrow["\beta", from=4-1, to=4-5]
\end{tikzcd}
}
\]

\item
% https://q.uiver.app/#q=WzAsNCxbMCwwLCJcXGxhe0hvbX0oQSxcXGludF9JIE1faSBkZlxcbXUpIl0sWzAsMywiXFxsYXtIb219KEYoQSksXFxpbnRfSSBGKE1faSkgZGZcXG11KSJdLFs0LDAsIlxcbGF7SG9tfShBLFxcaW50X0ogTV9qIGRcXG11KSJdLFs0LDMsIlxcbGF7SG9tfShGKEEpLFxcaW50X0ogRihNX2opIGRcXG11KSJdLFswLDEsIlxcemV0YSJdLFswLDIsIlxcWGlfe2YsXFxtdX0iLDJdLFsxLDMsIlxcWGlfe2YsXFxtdX0iXSxbMiwzLCJcXHpldGEiLDJdXQ==
\[\begin{tikzcd}
	{\la{Hom}(A,\int_I M_i df\mu)} &&&& {\la{Hom}(A,\int_J M_j d\mu)} \\
	\\
	\\
	{\la{Hom}(F(A),\int_I F(M_i) df\mu)} &&&& {\la{Hom}(F(A),\int_J F(M_j) d\mu)}
	\arrow["{\Xi_{f,\mu}}"', from=1-1, to=1-5]
	\arrow["\zeta", from=1-1, to=4-1]
	\arrow["\zeta"', from=1-5, to=4-5]
	\arrow["{\Xi_{f,\mu}}", from=4-1, to=4-5]
\end{tikzcd}\]

Also it is required to satisfy the following unit axiom:

\item

\[
% https://q.uiver.app/#q=WzAsNSxbMCwwXSxbMSwyXSxbMSwwLCIxIl0sWzAsMSwiXFxtYXRocm17SG9tfShBLFxcaW50XypBIFxcIGQqKSJdLFsyLDEsIlxcbWF0aHJte0hvbX0oRihBKSxcXGludF8qRihBKSBcXCBkKikiXSxbMiwzLCJcXGthcHBhXjFfQSJdLFsyLDQsIlxca2FwcGFeMl97RihBKX0iLDJdLFszLDQsIlxcemV0YSJdXQ==
\begin{tikzcd}
	{} & \mathbf{1} \\
	{\mathrm{Hom}(A,\int_*A \ d*)} && {\mathrm{Hom}(F(A),\int_*F(A) \ d*)} \\
	& {}
	\arrow["{\kappa^1_A}", from=1-2, to=2-1]
	\arrow["{\kappa^2_{F(A)}}"', from=1-2, to=2-3]
	\arrow["\zeta", from=2-1, to=2-3]
\end{tikzcd}
\]
\end{enumerate}

Now we should show that every generalised left ultrafunctor is in fact a functor between the underlying categories of the generalised ultracategories, to do that we show that the outermost diagram commutes:

\[
\adjustbox{scale =0.56}{

\begin{tikzcd}
	{\Hom(A,\int_* B d*) \times \Hom(B,\int_*Cd*)} &&& {\Hom(A,\int_* B d*) \times \int_*\Hom(B,\int_*Cd*)d*} &&& {\Hom(A,\int_{\coprod_**}Cd*)} &&&& {\Hom(A,\int_{*}Cd*)} \\
	\\
	\\
	\\
	{\Hom(F(A),\int_* F(B) d*) \times \Hom(F(B),\int_*F(C)d*)} &&& {\Hom(F(A),\int_* F(B) d*) \times \int_*\Hom(F(B),\int_*F(C)d*)} &&& {\Hom(F(A),\int_{\coprod_**}F(C)d*)} &&&& {\Hom(F(A),\int_{*}F(C)d*)}
	\arrow[""{name=0, anchor=center, inner sep=0}, "{\mathrm{id} \times \epsilon^{-1}}", from=1-1, to=1-4]
	\arrow["{\zeta \times \zeta}", from=1-1, to=5-1]
	\arrow[""{name=1, anchor=center, inner sep=0}, "\beta", from=1-4, to=1-7]
	\arrow["{\zeta \times \int_* \zeta d*}"{description}, from=1-4, to=5-4]
	\arrow[""{name=2, anchor=center, inner sep=0}, "{\Xi_{f_*}}", from=1-7, to=1-11]
	\arrow["\zeta"{description}, from=1-7, to=5-7]
	\arrow["\zeta", from=1-11, to=5-11]
	\arrow[""{name=3, anchor=center, inner sep=0}, "{\mathrm{id} \times \epsilon^{-1}}"', from=5-1, to=5-4]
	\arrow[""{name=4, anchor=center, inner sep=0}, "\beta"', from=5-4, to=5-7]
	\arrow[""{name=5, anchor=center, inner sep=0}, "{\Xi_{f*}}"', from=5-7, to=5-11]
	\arrow["1", draw=none, from=0, to=3]
	\arrow["2", draw=none, from=1, to=4]
	\arrow["3", draw=none, from=2, to=5]
\end{tikzcd}
}
\]

Square $1$ is commutative by naturality of $\la{id} \times \epsilon^{-1}$, while squares $2$ and $3$ are parts of the definition.

\paragraph{Composition of left ultrafunctors} Suppose that $F_1$ is a left ultrafunctor from $\ca{A}$ to $\ca{B}$ with the family of maps $(\zeta_1)$ and $F_2$ a left ultrafunctor with family $(\zeta_2)$, then we define the composition $F_2 \circ F_1$ by the family $(\zeta_2 \circ \zeta_1)$. Clearly, this composition is compatible with the composition of underlying functors.

\begin{note*} \normalfont
If there is no risk of confusion, we will be denoting $\zeta(a)$ by $F(a)$ (using the same symbol for the function on objects and on generalised morphisms).
\end{note*}

\section{2-morphisms} \label{2-morphisms}

Suppose we have two generalised ultracategories $\ca{A}$ and $\ca{B}$ and two left ultrafunctors $F$ and $G$ between them. We define a $2$-morphism $\alpha$ to be two compatible families of morphisms $\alpha^1_{B,(M_i)}$, from $\la{Hom}(B, \int_I F(M_i) d\mu)$ to $\la{Hom}(B, \int_I G(M_i) d\mu)$,  where $B \in \ca{B}$, the family $(M_i)$ is of objects $\ca{A}$ and $\mu$ an ultrafilter on $I$; and a family $\alpha^2_{A, (M^{'}_i)_{i \in I} }$ from $\la{Hom}(G(A), \int_I M^{'}_i d\mu)$ to  $\la{Hom}(F(A), \int_I M^{'}_i d\mu)$, where $B \in \ca{B}$, $(M'_i)_{i \in I}$ is a family of objects of $\ca{B}$, and $\mu$ an ultrafilter on $I$, that  satisfies the following axioms, which can be expressed by the commutativity of these diagrams:

\begin{enumerate}
    \item 

% https://q.uiver.app/#q=WzAsNCxbMCwwLCJcXGxhe0hvbX0oRyhBKSxcXGludF9JIE1faSBkZlxcbXUpICJdLFswLDQsIlxcbGF7SG9tfShGKEEpLFxcaW50X0kgTV9pIGRmXFxtdSkgIl0sWzUsNCwiXFxsYXtIb219KEYoQSksXFxpbnRfSiBNX3tmKGopfSBkXFxtdSkiXSxbNSwwLCJcXGxhe0hvbX0oRyhBKSwgXFxpbnRfSiBNX3tmKGopfSBkXFxtdSJdLFswLDEsIlxcYWxwaGFeMl97QSwgKE1faSl9ICJdLFsxLDIsIlxcWGlfe2Z9Il0sWzAsMywiXFxYaV97Zn0iLDJdLFszLDIsIlxcYWxwaGFeMl97QSwoTV97KGYoail9KX0iLDJdXQ==
\[\begin{tikzcd}
	{\la{Hom}(G(A),\int_I M_i df\mu) } &&&&& {\la{Hom}(G(A), \int_J M_{f(j)} d\mu)} \\
	\\
	\\
	\\
	{\la{Hom}(F(A),\int_I M_i df\mu) } &&&&& {\la{Hom}(F(A),\int_J M_{f(j)} d\mu)}
	\arrow["{\Xi_{f}}"', from=1-1, to=1-6]
	\arrow["{\alpha^2_{A, (M_i)} }", from=1-1, to=5-1]
	\arrow["{\alpha^2_{A,(M_{f(j)})}}"', from=1-6, to=5-6]
	\arrow["{\Xi_{f}}", from=5-1, to=5-6]
\end{tikzcd}\]

\item
\[
% https://q.uiver.app/#q=WzAsNCxbMCwwLCJcXGxhe0hvbX0oQSxcXGludF9JIEYoTV9pKSBkZlxcbXUpICJdLFswLDQsIlxcbGF7SG9tfShBLFxcaW50X0kgRyhNX2kpIGRmXFxtdSkgIl0sWzUsNCwiXFxsYXtIb219KEEsXFxpbnRfSiBHKE1fe2Yoail9KSBkXFxtdSkiXSxbNSwwLCJcXGxhe0hvbX0oQSwgXFxpbnRfSiBGKE1fe2Yoail9KSBkXFxtdSkiXSxbMCwxLCJcXGFscGhhXjFfe0EsIChNX2kpfSAiXSxbMSwyLCJcXFhpX3tmfSJdLFswLDMsIlxcWGlfe2Z9IiwyXSxbMywyLCJcXGFscGhhXjFfe0EsKE1feyhmKGopfSl9IiwyXV0=
\begin{tikzcd}
	{\la{Hom}(A,\int_I F(M_i) df\mu) } &&&&& {\la{Hom}(A, \int_J F(M_{f(j)}) d\mu)} \\
	\\
	\\
	\\
	{\la{Hom}(A,\int_I G(M_i) df\mu) } &&&&& {\la{Hom}(A,\int_J G(M_{f(j)}) d\mu)}
	\arrow["{\Xi_{f}}"', from=1-1, to=1-6]
	\arrow["{\alpha^1_{A, (M_i)} }", from=1-1, to=5-1]
	\arrow["{\alpha^1_{A,(M_{f(j)})}}"', from=1-6, to=5-6]
	\arrow["{\Xi_{f}}", from=5-1, to=5-6]
\end{tikzcd}
\]

\item

\[
\adjustbox{max width =\textwidth}{
% https://q.uiver.app/#q=WzAsNCxbMCwwLCJcXGxhe0hvbX0oRyhBKSxcXGludF9JIE1faSBkXFxtdSkgXFx0aW1lcyBcXGludF9JIFxcbGF7SG9tfShNX2ksIFxcaW50X3tYX2l9Tl97KGkseCl9ZCBcXGxhbWJkYV9pKSBkXFxtdSJdLFswLDQsIlxcbGF7SG9tfShGKEEpLFxcaW50X0kgTV9pIGRcXG11KSBcXHRpbWVzIFxcaW50X0kgXFxsYXtIb219KE1faSwgXFxpbnRfe1hfaX1OX3soaSx4KX1kIFxcbGFtYmRhX2kpIGRcXG11Il0sWzUsNCwiXFxsYXtIb219KEYoQSksXFxpbnRfe1xcY29wcm9kX3tJfVhfaX1OX3soaSx4KX1kIFxcaW50X3tJfVxcaW90YV9pIFxcbGFtYmRhX2kgZFxcbXUiXSxbNSwwLCJcXGxhe0hvbX0oRyhBKSxcXGludF97XFxjb3Byb2Rfe0l9WF9pfU5feyhpLHgpfWQgXFxpbnRfe0l9XFxpb3RhX2kgXFxsYW1iZGFfaSBkXFxtdSJdLFswLDEsIlxcYWxwaGFeMl97QSwgKE1faSl9IFxcdGltZXMgXFxsYXtpZH0iXSxbMSwyLCJcXGJldGEiXSxbMCwzLCJcXGJldGEiLDJdLFszLDIsIlxcYWxwaGFeMl97QSwoTl97KGkseCl9KX0iLDJdXQ==
\begin{tikzcd}
	{\la{Hom}(G(A),\int_I M_i d\mu) \times \int_I \la{Hom}(M_i, \int_{X_i}N_{(i,x)}d \lambda_i) d\mu} &&&&& {\la{Hom}(G(A),\int_{\coprod_{I}X_i}N_{(i,x)}d \int_{I}\iota_i \lambda_i d\mu)} \\
	\\
	\\
	\\
	{\la{Hom}(F(A),\int_I M_i d\mu) \times \int_I \la{Hom}(M_i, \int_{X_i}N_{(i,x)}d \lambda_i) d\mu} &&&&& {\la{Hom}(F(A),\int_{\coprod_{I}X_i}N_{(i,x)}d \int_{I}\iota_i \lambda_i d\mu)}
	\arrow["\beta"', from=1-1, to=1-6]
	\arrow["{\alpha^2_{A, (M_i)} \times \la{id}}", from=1-1, to=5-1]
	\arrow["{\alpha^2_{A,(N_{(i,x)})}}"', from=1-6, to=5-6]
	\arrow["\beta", from=5-1, to=5-6]
\end{tikzcd}
}
\]
\item
% https://q.uiver.app/#q=WzAsNCxbMCwwLCJcXGxhe0hvbX0oQSxcXGludF9JIEJfaSBkXFxtdSlcXHRpbWVzIFxcaW50X0lcXGxhe0hvbX0oQl9pLCBcXGludF97WF9pfSBGKE5feyhpLHgpfSkgZFxcbGFtYmRhX2kpZFxcbXUiXSxbMCwzLCJcXGxhe0hvbX0oQSxcXGludF9JIEJfaSBkXFxtdSlcXHRpbWVzIFxcaW50X0lcXGxhe0hvbX0oQl9pLCBcXGludF97WF9pfSBHKE5feyhpLHgpfSkgZFxcbGFtYmRhX2kpZFxcbXUiXSxbNCwwLCJcXGxhe0hvbX0oQSxcXGludF97XFxjb3Byb2Rfe0l9WF9pfUYoTl97KGkseCl9KWRcXGludF9JIFxcaW90YV9pIFxcbGFtYmRhX2kgZFxcbXUpIl0sWzQsMywiXFxsYXtIb219KEEsXFxpbnRfe1xcY29wcm9kX3tJfVhfaX1HKE5feyhpLHgpfSlkXFxpbnRfSSBcXGlvdGFfaSBcXGxhbWJkYV9pIGRcXG11KSJdLFswLDEsIlxcbGF7aWR9IFxcdGltZXMgXFxpbnRfSSBcXGFscGhhXjFfe0JfaSwoTl97KGkseCkpfX1kXFxtdSJdLFswLDIsIlxcYmV0YSIsMl0sWzIsMywiXFxhbHBoYV4xX3tBLChGKE5feyhpLHgpfSl9IiwyXSxbMSwzLCJcXGJldGEiXV0=
\[\begin{tikzcd}
	{\la{Hom}(A,\int_I B_i d\mu)\times \int_I\la{Hom}(B_i, \int_{X_i} F(N_{(i,x)}) d\lambda_i)d\mu} &&&& {\la{Hom}(A,\int_{\coprod_{I}X_i}F(N_{(i,x)})d\int_I \iota_i \lambda_i d\mu)} \\
	\\
	\\
	{\la{Hom}(A,\int_I B_i d\mu)\times \int_I\la{Hom}(B_i, \int_{X_i} G(N_{(i,x)}) d\lambda_i)d\mu} &&&& {\la{Hom}(A,\int_{\coprod_{I}X_i}G(N_{(i,x)})d\int_I \iota_i \lambda_i d\mu)}
	\arrow["\beta"', from=1-1, to=1-5]
	\arrow["{\la{id} \times \int_I \alpha^1_{B_i,(N_{(i,x)})}d\mu}", from=1-1, to=4-1]
	\arrow["{\alpha^1_{A,(N_{(i,x)})}}"', from=1-5, to=4-5]
	\arrow["\beta", from=4-1, to=4-5]
\end{tikzcd}\]

\item

\[\begin{tikzcd}
	{\la{Hom}(A,\int_IM_i d\mu)} && {\la{Hom}(F(A),\int_IF(M_i) d\mu)} \\
	\\
	{\la{Hom}(G(A), \int_IG(M_i)d\mu)} && {\la{Hom}(F(A),\int_IG(M_i) d\mu)}
	\arrow["\zeta", from=1-1, to=1-3]
	\arrow["{\zeta^{'}}", from=1-1, to=3-1]
	\arrow["{\alpha^1_{F(A),(M_i)}}"{description}, from=1-3, to=3-3]
	\arrow["{\alpha^2_{A,(G(M_i))}}", from=3-1, to=3-3]
\end{tikzcd}\]

\end{enumerate}

\begin{note*}
    \normalfont  These $2$-morphisms are in fact natural transformations between the underlying functors, to do so suppose that we have a such $2$-morphism $\alpha= (\alpha^1,\alpha^2)$, then we define the components of this natural transformation to be the images of the identity element of the underlying category $\kappa \in \mathrm{Hom}(A,\int_*Ad*)$ by the composition $\alpha^2 \circ \zeta^{'} = \alpha^1 \circ \zeta $, in other words $\alpha^2(\kappa_{G(A)})=\alpha^1(\kappa_{F(A)})$. We leave to the reader filling the details of the proof (basically using either axioms $3$ or $4$, which with a little work, can be regarded as representable versions of the naturality square).

    This leads us to notice an additional thing: The way we have defined $2$-morphism  is not the most optimal way to define $2$-morphisms ``data-wise''. A much more concise way would be to define  $2$-morphisms  to be a family of  maps from $\mathbf{1}$ to $\cat{Hom}(F(A), G(A))= \la{Hom}(F(A), \int_* G(A) d*)$, or equivalently a family of morphisms $(\alpha_A)_{A \in \ca{A}}$ in $\ca{B}$ such that $\alpha_A \in \la{Hom}(F(A),\int_* G(A) d*)$, after that we define the family of maps $(\alpha^1)$ as follows:
\[
\adjustbox{scale =0.89}{
\begin{tikzcd}[row sep=3em, every label/.append style={font=\small}]
	{\la{Hom}(A, \int_I F(M_i)d\mu)} \\
	{\la{Hom}(A, \int_I F(M_i)d\mu) \times \int_*\mathbf{1}d*} \\
	{\la{Hom}(A, \int_I F(M_i)d\mu) \times \int_I \la{Hom}(F(M_i), \int_* G(M_i)d*)d\mu} \\
	{\la{Hom}(A, \int_{\coprod_I*} G(M_i)d\mu)} \\
	{\la{Hom}(A, \int_{I} G(M_i)d\mu)}
	\arrow["\simeq", from=1-1, to=2-1]
	\arrow["{\mathrm{id} \times  (\alpha_i)_{i \in I}}", from=2-1, to=3-1]
	\arrow["\beta", from=3-1, to=4-1]
	\arrow["\Xi", from=4-1, to=5-1]
\end{tikzcd}}
\]
    
And the family $(\alpha^2)$ as follows:

\[
\adjustbox{scale =0.89}{
\begin{tikzcd}[row sep=3em, every label/.append style={font=\small}]
	{\la{Hom}(G(A),\int_I M_i d\mu)} \\
	{\mathbf{1} \times \la{Hom}(G(A),\int_I M_i d\mu) } \\
	{\la{Hom}(F(A),\int_* G(A)d*)\times \la{Hom}(G(A),\int_I M_i d\mu) } \\
	{\la{Hom}(F(A),\int_{\coprod_*I}M_id\mu)} \\
	{\la{Hom}(F(A),\int_{I}M_id\mu)}
	\arrow["\simeq", from=1-1, to=2-1]
	\arrow["{\alpha_A \times \la{id}}", from=2-1, to=3-1]
	\arrow["\beta", from=3-1, to=4-1]
	\arrow["\Xi", from=4-1, to=5-1]
\end{tikzcd}}
\]
\end{note*}

\subsection{Vertical and Horizontal composition of natural transformations of left ultrafunctors} 

\paragraph{Vertical composition} Suppose that we are in the following setting:

\[\begin{tikzcd}
	{\ca{A}} &&&&& {\ca{B}}
	\arrow[""{name=0, anchor=center, inner sep=0}, "{F_1}", curve={height=-40pt}, from=1-1, to=1-6]
	\arrow[""{name=1, anchor=center, inner sep=0}, "{F_2}"{description,pos=0.6}, shift left, from=1-1, to=1-6]
	\arrow[""{name=2, anchor=center, inner sep=0}, "{F_3}"', curve={height=40pt}, from=1-1, to=1-6]
	\arrow["{\alpha=(\alpha^1,\alpha^2)}"{description}, Rightarrow, from=0, to=1]
	\arrow["{\alpha^{'}=({\alpha^{'}}^1, {\alpha^{'}}^2)}"{description}, Rightarrow, from=1, to=2]
\end{tikzcd}\]
We can define $\alpha^{'} * \alpha= ({\alpha^{'}}^1 * \alpha^1, {\alpha}^2 * {\alpha^{'}}^2)$.

\paragraph{Horizontal composition}

\[
% https://q.uiver.app/#q=WzAsMyxbMCwwLCJcXGNhe0F9Il0sWzMsMCwiXFxjYXtCfSJdLFs2LDAsIlxcY2F7Q30iXSxbMCwxLCJGXzEiLDAseyJjdXJ2ZSI6LTJ9XSxbMSwyLCJGXzIiLDAseyJjdXJ2ZSI6LTJ9XSxbMCwxLCJHXzEiLDIseyJjdXJ2ZSI6Mn1dLFsxLDIsIkdfMiIsMix7ImN1cnZlIjoyfV0sWzMsNSwiKFxcYWxwaGFfMSxcXGFscGhhXzIpIiwwLHsic2hvcnRlbiI6eyJzb3VyY2UiOjIwLCJ0YXJnZXQiOjIwfX1dLFs0LDYsIih7XFxhbHBoYV57J319XnsxfSx7XFxhbHBoYV57J319XnsyfSkiLDAseyJzaG9ydGVuIjp7InNvdXJjZSI6MjAsInRhcmdldCI6MjB9fV1d
\begin{tikzcd}
	{\ca{A}} &&& {\ca{B}} &&& {\ca{C}}
	\arrow[""{name=0, anchor=center, inner sep=0}, "{F_1}", curve={height=-12pt}, from=1-1, to=1-4]
	\arrow[""{name=1, anchor=center, inner sep=0}, "{G_1}"', curve={height=12pt}, from=1-1, to=1-4]
	\arrow[""{name=2, anchor=center, inner sep=0}, "{F_2}", curve={height=-12pt}, from=1-4, to=1-7]
	\arrow[""{name=3, anchor=center, inner sep=0}, "{G_2}"', curve={height=12pt}, from=1-4, to=1-7]
	\arrow["{(\alpha^1,\alpha^2)}", shorten <=3pt, shorten >=3pt, Rightarrow, from=0, to=1]
	\arrow["{({\alpha^{'}}^{1},{\alpha^{'}}^{2})}", shorten <=3pt, shorten >=3pt, Rightarrow, from=2, to=3]
\end{tikzcd}
\]
Let $C \in \ca{C}$ and $(M_i)$ a family of objects of $\ca{A}$, we define $(\alpha^{'} \circ \alpha)^2_{C,(M_i)}$ as the following composition:
\[
\begin{tikzcd}[row sep=3em, every label/.append style={font=\small}]
	{\la{Hom}(G_2\circ G_1(A),\int_I M_id\mu)} \\
	{\la{Hom}(F_2\circ G_1(A),\int_I M_id\mu)} \\
	{\la{Hom}(F_2\circ G_1 (A), \int_*F_2\circ G_1 (A)d*) \times \int_*\la{Hom}(F_2\circ G_1(A),\int_I M_id\mu)d*} \\
	{\la{Hom}(F_2\circ F_1 (A), \int_*F_2\circ G_1 (A)d*) \times \int_*\la{Hom}(F_2\circ G_1(A),\int_I M_id\mu)d*} \\
	{\la{Hom}(G_2\circ G_1 (A), \int_{\coprod_*I}M_i)} \\
	{\la{Hom}(G_2\circ G_1 (A), \int_{I}M_i)}
	\arrow["{{\alpha^2}^{'}_{G_1(A)}}", from=1-1, to=2-1]
	\arrow["{\kappa_{} \times \simeq}", from=2-1, to=3-1]
	\arrow["{\zeta_2(\alpha^2_{A,(M_i)_{i \in I}})\times \la{id}}", from=3-1, to=4-1]
	\arrow["\beta", from=4-1, to=5-1]
	\arrow["\Xi", from=5-1, to=6-1]
\end{tikzcd}
\]
Now let $A \in  \ca{A}$ and $(M_i)_{i \in I}$ be a family of objects of $\ca{C}$, we define  to be the following composition:
\[
\begin{tikzcd}[row sep=3em, every label/.append style={font=\small}]
	{\la{Hom}(C,\int_I F_2 \circ F_1(M_i)d\mu)} \\
	{\la{Hom}(C,\int_I G_2 \circ F_1(M_i)d\mu)} \\
	{\la{Hom}(C,\int_I G_2 \circ F_1(M_i)d\mu) \times \int_I\la{Hom}(G_2\circ F_1(M_i),\int_* G_2\circ F_1(M_i)d*)d\mu} \\
	{\la{Hom}(C,\int_I G_2 \circ F_1(M_i)d\mu) \times \int_I\la{Hom}(G_2\circ F_1(M_i),\int_* G_2\circ G_1(M_i)d*)d\mu} \\
	{\la{Hom}(C, \int_{\coprod_I*}G_2 \circ G_1 (M_i) d\mu)} \\
	{\la{Hom}(C, \int_{I}G_2 \circ G_1 (M_i) d\mu)}
	\arrow["{{\alpha^1}^{'}_{C,(F_1(M_i))}}", from=1-1, to=2-1]
	\arrow["{\la{id} \times (\kappa_i)}", from=2-1, to=3-1]
	\arrow["{ \la{id} \times  \int_I \zeta_2(\alpha^1_{A,M_i})d\mu}", from=3-1, to=4-1]
	\arrow["\beta", from=4-1, to=5-1]
	\arrow["\Xi", from=5-1, to=6-1]
\end{tikzcd}
\]

It's clear that this data makes generalised ultracategories a $2$-category with a forgetful functor to the category of categories, sending a generalised ultracategory to its underlying category.

\section{Examples of generalised ultracategories} \label{examples of generalised}

\subsection{Topological spaces}
\begin{theorem}
Let $X$ be a topological space,  then we define a generalised ultrastructure on it as follows, we define $\la{Hom}(A, \int_I M_i d\mu) = \{*\}$ a singleton iff the pushforward of the ultrafilter $\mu$ by the map $i \mapsto M_i$ converges to $A$, and otherwise we define $\la{Hom}(A, \int_I M_i d\mu) = \varnothing$ . This will make the topological space a generalised ultracategory, with underlying category the topological space viewed as a preorder category with the specialisation preorder, which we explain in the next definition.

\end{theorem}

\begin{definition}
 Let $X$ be a topological space, and let $\ca{N}(x)$ denote the set of all open neighbourhoods of $x$, then we define a preorder relation on $X$, by $x \leq y$ iff $\ca{N}(x) \subseteq \ca{N}(y)$ (in ultrafilter language iff the principal ultrafilter at $y$ converges to $x$), and we call this the specialisation preorder on $X$. Similarly, we can define the specialisation preorder on $\la{Hom}_{\cat{Top}}(X,Y)$, by setting $f \leq g$, iff $\forall x \in X \ f(x) \leq g(x)$.
\end{definition}

\begin{note*} \normalfont
    If the space is $T_0$ then the underlying  category is a poset, and if it is $T_1$ then the category is discrete.
\end{note*}

    We did not specify the maps $\beta$ and $\kappa$ and $\Xi$. We define the maps $\Xi$ to be isomorphisms. Now since all generalised $\la{Hom}$ sets are singleton or empty. The existence of the maps $\beta$ and $\kappa$  is equivalent to specifying $\mathbf{logical}$ statements on the ultrafilter convergence relation on $X$:

    \begin{itemize}
        \item[R] We  have defined all the maps $\Xi$ to be isomorphisms, and hence their existence with this additional property is equivalent to the fact that we can restrict our attention to ultrafilters on $X$.
    
        \item[UQ1'] Every principal ultrafilter $\delta_x$ converges to $x$ (follows from $\kappa$ the unique element in $\Hom(x , \int_*  x d*)$ + change of base).
        \item[UQ4'] 
\
Suppose that we have an ultrafilter $\mu$ on a set $I$, a map of sets $m$ from $I$ to $X$, such that $m\mu$ converges to $x$, and family of ultrafilters $(\nu_i)_{i \in I}$ on sets $X_i$, and map of sets $t_i$ from $X_i$ to $X$ such that there exists some subset $I^{'}$ of $I$ that belongs to $\mu$, and such that for every $i \in I^{'}$, $t_i\nu_i$ converges to $m(i)$, then if we denote $t$ the map from $\coprod_{I}X_i$ to $X$, whose restriction to each $X_i$ is $t_i$, then  $t \int_I \iota_i \nu_i m\mu$ converges to $x$.

    \end{itemize}

        \begin{claim*} \normalfont Given a set $X$, there is a bijection between topologies on $X$, and ultrarelations on $X$ (subsets of $\beta X  \times X$)  satisfying $UQ4^{'}$ and $UQ1^{'}$.
            
        \end{claim*}
        A corollary of this claim is that if $X$ is a topological space, then the maps $\kappa$ and $\beta$ are completely determined if we want the generalised $\la{Hom}$-sets to encode ultrafilter convergence.

        The proof of the claim  follows from a result by Wyler namely the following\cite{wyler1996convergence}:
        \begin{theorem} \label{topology theorem}
            let $X$ be a set then there is a  bijection between topologies on $X$, and ultrarelations on $X$ (subsets of $\beta X \times X$) satisfying the following properties:

            \begin{itemize}
                \item[UQ1] $\forall x \ (\delta_x,x) \in \ca{R}$

                \item[UQ4] Let $t$ be a map of sets from $I$ to $\beta X$, such that $(t(i), x_i) \in \ca{R} $, and suppose that $\mu$ is an ultrafilter on $I$ such that $(x\mu,y) \in \ca{R}$, then $(\int_I t(i)d\mu,y) \in \ca{R}$.
            \end{itemize}
            Here $x\mu$ is the pushforward of $\mu$ by $x$ regarded as the map $i \mapsto x_i$.
\end{theorem}
        It is not hard to show that the conditions $UQ1$ and $UQ4$ are equivalent to $UQ1^{'}$ and $UQ4^{'}$. And hence giving a topology on $X$ is equivalent to giving a subset of $\beta X \times X$ satisfying $UQ1^{'}$ and $UQ4^{'}$, or alternatively giving a generalised ultrastructure on a category with underlying set $X$, such that each generalised $\la{Hom}$-set contains at most one element, and such that change of base maps are isomorphisms.

        A preorder relation on a set is a relation (subset of $X \times X$), which is transitive and reflexive. Categorically , a preorder category is a category where there is at most one morphism between any two objects, a classic result in point set topology establishes an equivalence between finite topological spaces and finite preorders. Generalising this to subsets of $\beta X \times X$ motivates the following definition:
        
        \begin{definition}We define an ultrapreorder to be a small generalised ultracategory such that $\la{Hom}(A, \int_I M_i d \mu)$ contains at most one element, and such that all change of base maps $\Xi$ are isomorphisms. \label{ultapreorder}

        \end{definition}
      
    Thus a conclusion of this entire discussion:

        \begin{theorem} Let $X$ be a set, then there is a bijective correspondence between topologies on $X$, and ultrapreorder structures on the same set. \end{theorem}
 
But can we elevate the theorem above to an equivalence of $2$-categories? The answer is yes, as we can easily see that a left ultrafunctor between ultrapreorders is a continuous function, between underlying topologies, and a natural transformation of left ultrafunctors is just a natural transformation between continuous functions regarded as functors when we regard topological spaces as categories with specialisation preorder.

More formally:

\begin{theorem}Let $\cat{Top}$ denote the $2$-category whose objects are topological spaces, $1$-morphisms are continuous functions, and $2$-morphisms are $f \leq g$ (which means that $\forall x  \ f(x) \leq g(x)$ in the specialisation preorder) then the construction above provides an equivalence of $2$-categories between $\mathrm{Ultrapreorders}$ and $\cat{Top}$.

\end{theorem}

\begin{note*} \normalfont
    We should notice that two ultrapreorders being equivalent does not mean that they  are homeomorphic as topological spaces, but rather their $T_0$ (Kolmogorov) quotients are homeomorphic. This is not an obstruction, since sheaves do not distinguish between such spaces.
\end{note*}

Now we turn to our next important results:
\begin{theorem}
    Let $X$ be a topological space then there exists an equivalence of categories between left ultrafunctors from $X$ to $\cat{Set}$, and étale spaces  over $X$.
\end{theorem}

\begin{note*} \normalfont
    Recall that étale spaces over $X$, are equivalent to sheaves on $X$.
\end{note*}

\begin{note*} \normalfont
    It goes without saying that $\cat{Set}$ is trivially a generalised ultracategory with generalised $\la{Hom}$-sets being actual $\la{Hom}$-sets. For a rigorous construction of the ultrastructure, see the next subsection.
\end{note*}

This is a generalisation of the theorem stated by \cite{lurie2018ultracategories} regarding usual ultracategories, which says that sheaves over $X$ where $X$ is a compact Hausdorff space, and left ultrafunctors from $X$ to $\cat{Set}$ are equivalent.

\begin{proof} The proof stems from the fact that the alternative proof that we gave of the statement above in \cite{HAMAD2026103794} (the fact that $\mathrm{Lult}(X,\cat{Set})$ is equivalent to étale bundles over $X$, if $X$ is compact Hausdorff), would still work if the base space is no longer compact Hausdorff, if we generalise the definition of left ultrafunctors.

Suppose that we have a left ultrafunctor $\ca{F}$ from $X$ to $\cat{Set}$, then we equip the space $E= \coprod_{x \in X}  \ca{F}(x)$, with the topology such that an ultrafilter $\mu$ on $E$ converges to $a$ iff $\pi\mu$ converges to $\pi(a)$, and if $\sigma_{\mu}(a)= (b_x)_{x \in X}$, then $\{ b_x: x \in X \} \in \mu$ (obviously, this does not depend on the representative of the element in the $\int_{X}\ca{F}(x)d\mu$). Here $\sigma_{\mu} \in \la{Hom}(\ca{F}(\pi(a)),\int_{X} \ca{F}(x)d \mu)$ is the image of the unique element in $\la{Hom}(\pi (a) , \int_{X} x d\pi\mu)$ by $\ca{F}$, this is a topology by theorem \ref{topology theorem}.

Now we show that this space is étale over $X$. This is equivalent to saying that $\pi$ is open and that $\Delta: E \rightarrow E \times_{X} E$ is open (which is equivalent to saying that the diagonal $\Delta \subseteq E \times_{X} E$ is open).

To show that $\pi$ is open, let $V$ be an open set of $E$ and let $\mu$ be an ultrafilter on $X$ such that $\mu$ converges to some $y \in \pi(V)$. We want to show that $\pi(V) \in \mu$. Since $y \in \pi(V)$ we can find some $f \in V$ such that $\pi(f)=y$. Now let $\sigma_{\pi \mu}(f)= (a_x)_{x \in X}$, it follows from \cite[Lemma 3.3]{HAMAD2026103794} (in that paper we developed the theory of bundles of metric spaces, so to apply this lemma in the simple case of sets, we equip them with discrete metric structure) that there exists some $N \in \mu$ such that $\{ a_x : x \in N \} \subseteq V$, this should imply that $N \subseteq  \pi(V)$, and this implies that $\pi(V) \in \mu$  which implies that $\pi(V)$ is open.

Next thing, we need to show that the diagonal of $E \times_{X} E$ is open. Towards this, suppose that $\mu$ is an ultrafilter on $E \times_{X} E$ such that $\mu$ converges to $(f,f)$  such that $f \in E_x$. Suppose that $\sigma_{\pi \pi_1  \mu}(f)= \sigma_{\pi \pi_2 \mu}(f)= (a_x)_{x \in X}$. Hence $\{a_x \} \in \pi_1\mu$ and $\{a_x\} \in \pi_2 \mu$. This implies that $\{(a_x,a_x)\}\in \mu$ and hence the diagonal $\Delta \in \mu$. So $\Delta$ is open so the space $E$ is étale over $X$. 

Now suppose that we have an étale space $E$ over $X$. then we can simply define a left ultrafunctor from $X$ to $\cat{Set}$ as follows: we send $x$ to $E_x$. Now if we have an ultrafilter $\mu$ on $I$ and a function $f$ from $I$ to $X$ such that $f\mu$ converges to $x$ then we define $\sigma_{\mu}$ from $E_x$ to $\int_{I} E_{f(i)} d \mu$ by sending $g \in E_x$ to $(a_{f(i)})_{i \in f^{-1}(U)}$ where $x \mapsto a_x$ is the local section defined from a neighbourhood $U$ of $x$ (the reader should note that when passing to the ultraproduct the value of $\sigma_{\mu}$ is independent of $U$).

We leave to the reader showing that these two processes are functorial and are  inverses of each other (up to equivalence).

\end{proof}

\subsection{Points of a Topos}

\label{points}
Let $E$ be a topos. We claim that the category of points of this topos has an ultrastructure constructed as follows: let $(C,\ca{J})$ be a site of definition for this topos, then the category of points of this topos is equivalent to  the category of $J$-continuous flat functor from $C$ to the category of sets, then it is possible to define the pointwise ultraproduct of these points seen as functors from $C$ to $\cat{Set}$ (notice that this pointwise ultraproduct is not necessarily a point), and hence we may define $\la{Hom}(F, \int_I G_i d\mu) =\la{Hom}_{Fun(C,\mathsf{Set})}(F, \int_I G_i d \mu)$ , here $\int_I G_i d \mu$ is the pointwise ultraproduct of the family of functors $(G_i)$.

We claim that this gives the category of points of a topos the structure of a generalised ultracategory. To do this we show a more general result:

\begin{theorem}
    Suppose $\ca{A}$ is a strong ultracategory, i.e. a pseudo-algebra for the pseudo-monad $T$ (that means that the colax associator as defined in \cite{hamad2025ultracategories} is invertible), and let $\ca{B}$ be any full subcategory of $\ca{A}$, then there is a canonical  generalised ultrastructure on $\ca{B}$ that we describe in the following passage: 

    \label{not very crucial theorem}
\end{theorem}
We define $\la{Hom}(A, \int_I M_i d\mu)$ to be $\la{Hom}_{\ca{B}}(A,\int_I M_i d\mu)$.

In order to define the  maps $\beta$, we define another class of maps $\Omega_{(A_i), (B_i)}: \int_I \la{Hom}(A_i, B_i)d\mu \rightarrow \la{Hom} ( \int_I A_i d \mu , \int_I B_i d\mu)$ as follows. We know that $\int_I \mathrm{Hom}(A_i, B_i) d\mu =\varinjlim_{U \in \mu}  \prod_{i \in U} \mathrm{Hom}(A_i, B_i)$. Now take $U \in \mu$, looking at the diagram below: 

% https://q.uiver.app/#q=WzAsNCxbMCwwLCJcXHByb2Rfe2kgXFxpbiBVIH1cXG1hdGhybXtIb219KEFfaSwgQl9pKSAiXSxbMywwLCIgXFxtYXRocm17SG9tfShcXGludF9VQV9pIGRcXG11LCBcXGludF9VQl9pIGRcXG11KSJdLFs2LDAsIiBcXG1hdGhybXtIb219KFxcaW50X0lBX2lkXFxtdSwgXFxpbnRfSUJfaSBkXFxtdSkiXSxbMCwzLCJcXGludF9JIFxcbWF0aHJte0hvbX0oQV9pLCBCX2kpIGRcXG11Il0sWzAsMSwiXFxpbnRfVSJdLFsxLDIsIlxcc2ltZXEiXSxbMCwzLCJjX1UiXSxbMywyLCJcXE9tZWdhX3t9IiwwLHsic3R5bGUiOnsiYm9keSI6eyJuYW1lIjoiZGFzaGVkIn19fV1d
\[\begin{tikzcd}
	{\prod_{i \in U }\mathrm{Hom}(A_i, B_i) } &&& { \mathrm{Hom}(\int_UA_i d\mu, \int_UB_i d\mu)} &&& { \mathrm{Hom}(\int_IA_id\mu, \int_IB_i d\mu)} \\
	\\
	\\
	{\int_I \mathrm{Hom}(A_i, B_i) d\mu}
	\arrow["{\int_U}", from=1-1, to=1-4]
	\arrow["{c_U}", from=1-1, to=4-1]
	\arrow["\simeq", from=1-4, to=1-7]
	\arrow["{\Omega_{}}", dashed, from=4-1, to=1-7]
\end{tikzcd}\]

\noindent $\Omega$ is the unique map that exists by the universal property of the directed colimit $\varinjlim_{U \in \mu} \prod_{i \in U} \mathrm{Hom}(A_i, B_i)$. Here we called $\mu$, by abuse of language, the ultrafilter on $I$ and the same ultrafilter restricted to $U$.

\begin{note*} \normalfont
  It is perhaps easier to describe $\Omega$ in non-categorical language, this map is the map that sends $(g_i) \in \int_I\la{Hom}(A_i,B_i)d\mu$ to $\int_I g_i d\mu \in \la{Hom}(\int_I A_i d\mu, \int_I B_i d\mu)$.
\end{note*}

Now we define the family of maps $\beta$ as follows:
\[
\begin{tikzcd}[row sep=3em, every label/.append style={font=\small}]
	{\mathrm{Hom}(A, \int_I B_i d\mu) \times \int_I \Hom(B_i, \int_{X_i}N_{(i,x)}d\lambda_i) d\mu} \\
	{\mathrm{Hom}(A, \int_I B_i d\mu) \times  \Hom(\int_IB_id\mu, \int_I\int_{X_i}N_{(i,x)}d\lambda_id\mu) } \\
	{\mathrm{Hom}(A,\int_I\int_{X_i}N_{(i,x)}d\lambda_i d\mu) } \\
	{\mathrm{Hom}(A, \int_{\coprod_{I}X_i} N_{(i,x)} d\int_I \iota_i \lambda_i d\mu)}
	\arrow["{\mathrm{id}\times\Omega}", from=1-1, to=2-1]
	\arrow["\circ", from=2-1, to=3-1]
	\arrow["{a^{-1} }", from=3-1, to=4-1]
\end{tikzcd}
\]
Now we define the maps $\Xi$ to be just the representables at the ultraproduct diagonal maps $\Delta_{\mu,f}$, and the maps $\kappa$ are defined to be  the map that sends $* \in \{*\}$ to $\epsilon_{*,*}^{-1}$ .

We leave the proof that this supplies the data of a generalised ultracategory to the appendix for two reasons, we don't believe that the proof adds any mathematical value to this paper, rather it may distract the reader in long diagrammatic proofs. The second reason is that, in the case which interests us the most which is the category of point of a topos, the generalised  ultrastructure can be verified without the use of the proof of the theorem, and it turns out to be rather trivial verification (since the $\cat{Set}$ case is trivial and we're computing the ``ultraproduct'' of points pointwise).

Note that the underlying category structure associated to the generalised ultrastructure of an ultracategory, and the original category structure of such categories are equivalent (since $\la{Hom}(A,B) \simeq \la{Hom}(A, \int_* B d*)$) and we are not going to distinguish those two in the future, unless it's required to do so.

\begin{theorem}
    Suppose that $\ca{B}$ and $\ca{B}^{'}$ are generalised ultracategories, which are actual ultracategories such that their generalised ultrastructure comes from \ref{not very crucial theorem}, then there is an isomorphism of categories between the category of left ultrafunctors from $\ca{B}$ to $\ca{B}^{'}$ as regular ultracategories , and the category of left ultrafunctors from $\ca{B}$ to $\ca{B}^{'}$ as generalised ultracategories. 
\end{theorem}

\begin{note*}
    \normalfont This theorem justifies calling morphisms between generalised ultracategories ``left ultrafunctors''.
\end{note*}

    \begin{note*}
        \normalfont

        This theorem provides a fully faithful embedding of the category of $T$-pseudo-algebras with lax morphism of pseudo-algebras and adequate $2$-morphisms, inside the category of generalised ultracategories.
    \end{note*}

    The proof can be found in appendix \hyperref[Equaivalence of different notions of left ultrafunctors]{Appendix C}.

    Now back to the case which interests us, which is the points of a topos $E$. We have shown that such category has a generalised ultrastructure given by $\la{Hom}(A,\int_I M_i d \mu)$ inside the category $\mathrm{Fun}(C,\mathsf{Set})$, where $(C,\ca{J})$ is a site of definition. We claim that up to isomorphism this construction is independent of the choice of site of definition i.e. we want to show the following lemma:

    \begin{lemma}Let $C_1$ and $C_2$ be two sites of definition of $E$, and let $x$ be a point of $E$ and let $(M_i)_{i \in I}$ be a family of points of $E$, we may regard those as either functors from $C_1$ or $C_2$ to $\cat{Set}$ and we may hence construct the functor $\int_I M_i d\mu$ (which is not necessarily a point of $E$), and then we may construct the  $\Hom$-set $\mathrm{Hom}(x,\int_I M_id\mu)$, now we claim that this construction is independent (up to iso) of whether we consider $x$ and $M_i$ as functors from $C_1$ to $\cat{Set}$ or from $C_2$ to $\cat{Set}$.
    \end{lemma}
\begin{proof}

We start by showing some obvious but useful lemmas:

A morphism of sites from $(\ca{C},J_1)$ to $(\ca{D},J_2)$ is a functor $F$ , such that precomposing with $F$ is the direct image part of a geometric morphism from $Sh(\ca{D})$ to $Sh(\ca{C})$.

\begin{lemma} \label{maybe useful lemma}Suppose that $(C_1,\ca{J}_1) \xrightarrow{\iota} (C_2,\ca{J}_2) $ is a morphism of sites of definition of $E$, such that $Sh(C_1,\ca{J}_1) \simeq Sh(C_2,\ca{J}_2)$ via the geometric morphism induced by  $\iota$ , then any object in $C_2$ can be covered by objects in the image of $C_1$ by $\iota$.
\end{lemma}

\begin{proof}
 Let us have a look at the following diagram:
% https://q.uiver.app/#q=WzAsNCxbMCwwLCJDXzEiXSxbMywwLCJDXzIiXSxbMCwyLCJTaChDXzEpIl0sWzMsMiwiU2goQ18yKSJdLFswLDEsIlxcaW90YSJdLFswLDIsImFfMVxcY2lyYyB5XzEiLDJdLFsxLDMsImFfMiBcXGNpcmMgeV8yIl0sWzMsMiwiLVxcY2lyY1xcaW90YSIsMl0sWzIsMywiTCIsMCx7ImN1cnZlIjo1fV1d
\[\begin{tikzcd}
	{C_1} &&& {C_2} \\
	\\
	{Sh(C_1)} &&& {Sh(C_2)}
	\arrow["\iota", from=1-1, to=1-4]
	\arrow["{a_1\circ y_1}"', from=1-1, to=3-1]
	\arrow["{a_2 \circ y_2}", from=1-4, to=3-4]
	\arrow["L", curve={height=30pt}, from=3-1, to=3-4]
	\arrow["{-\circ\iota}"', from=3-4, to=3-1]
\end{tikzcd}
\]
Here $y$ is Yoneda, and $a$ is the associated sheaf functor (sheafification). From this diagram we get that any $A \in C_2$, we can write $ a_2 \circ y_2(A) =\varinjlim a_2 \circ y_2 \iota (A_i) $, where each $A_i \in C_1$ (using the fact that every sheaf is a colimit of sheafification of representables), hence the family  $(a_2 \circ  y_2( A_i))$  is jointly epimorphic, hence $(\iota(A_i))$ is covering.
\end{proof}

\begin{lemma}

Suppose that $(C_1,J_1) \xrightarrow{\iota} (C_2,J_2)$ is a morphism of sites, then the equivalence between the points of topos described as $J_1$-continuous flat functors from $C_1$ to set, and the points of topos described as $J_2$-continuous flat functors, is precomposing with $\iota$. \label{maybe useful lemma 2}
    
\end{lemma}

\begin{proof}

The proof follows from the following facts: the inverse image part of any point is the left Kan extension along the composition of sheafification and Yoneda. If we call such composition $\iota_1$ and $\iota_2$, the fact that we have a morphism of sites means that $\iota_1 \simeq \iota_2 \circ \iota$. Moreover, the fact that we have a morphism of sites implies that the precomposition with $\iota$ of any $J_2$-continuous flat functor is $J_1$-continuous flat.  Now let us have a look at the following diagram: 

\[
\begin{tikzcd}[row sep=3.2em, column sep=4.2em, every label/.append style={font=\small}]
	{\begin{array}{c}J_1 \dash \mathrm{continuous \ flat}\\ \mathrm{functors \ to \ } \cat{Set}\end{array}} & {\begin{array}{c}J_2 \dash \mathrm{continuous \ flat}\\ \mathrm{functors \ to \ } \cat{Set}\end{array}} & {\mathrm{Geom}(\mathsf{Set},E)}
	\arrow["{\mathrm{Lan}_{\iota}}", from=1-1, to=1-2]
	\arrow["{\mathrm{Lan}_{\iota_1}}"', curve={height=26pt}, from=1-1, to=1-3]
	\arrow["{\mathrm{Lan}_{\iota_2}}", from=1-2, to=1-3]
\end{tikzcd}
\]

This diagram commutes up to isomorphism from the fact that left Kan extension is adjoint to precomposing and from the fact that $\iota_1 \simeq \iota_2 \circ \iota$, hence since both $\mathrm{Lan}_{\iota_1}$ and $\mathrm{Lan}_{\iota_2}$ are equivalences then $\mathrm{Lan}_{\iota}$ is equivalence, hence its right adjoint which is precomposing with $\iota$ is an equivalence.
\end{proof}
Now, let $C_1$ and $C_2$ be two sites of definition, without loss of generality, we can replace $C_2$ by the site which is the full subcategory of $E$ generated by $a_1\circ y_1(C_1) \bigcup a_2 \circ y_2(C_2) $ equipped with the restriction of the canonical topology, this will induce a sites map $\iota$ (not necessarily fully faithful unless $C_1$ is subcanonical) from $(C_1, \ca{J}_1)$ to $(C_2, \ca{J}_2)$, also let us call $\iota'$ the map of sites from $(C_1, \ca{J}_1)$ to $((a\circ y)(C_1))$ equipped with the canonical topology. The fact that this map of sites induces an equivalence of toposes follows from the comparison lemma \cite[theorem 2.2.3]{johnstone2002sketches} applied to $\iota$, and the fact that $\iota^{'}$ is from $(C_1,J_1)$ to $((a\circ y)(C_1))$ equipped with the canonical topology and hence it also induces an equivalence of sheaf toposes.

Now let us regard the family of points $(M_i)$ and any point $x$  as functors from $C_1$ to $\cat{Set}$, we want to show that  $\Hom(x,\int_I M_i d\mu) \simeq \Hom(\iota(x), \int_I \iota(M_i) d \mu)$ and that this isomorphism is natural both in $x$ and the family $(M_i)_{i \in I}$.

Suppose that we have a natural transformation $\alpha_1$ from $\iota(x)$ to the functor $\int_{I}\iota(M_i) d\mu$ viewed as functors from $C_2$ to $\mathsf{Set}$. We wish to show that $\alpha_1$ is completely determined by the natural transformation $\alpha_1 \circ \iota$, we will be calling such natural transformation $\alpha_1|_{C_1}$ (although we are not claiming that $\iota$ is fully faithful).

Towards this, suppose that $\alpha_2$ is another natural transformation such that for $\alpha_2|_{C_1} =\alpha_1|_{C_1} $. Let $c$ be an element of $C_2$, let us cover it  by a family $\iota(c_j)$ of elements of $\iota(C_1)$, that means $x(\iota(c_j))$ with the corresponding maps $x(e_j)$ is a colimit diagram \cite[Chap VII, lemma 3]{maclane2012sheaves}, same thing for each $M_i$. Now let us have a look at the following diagram:

% https://q.uiver.app/#q=WzAsNCxbMywwLCJcXGludF97SX1NX2koXFxpb3RhKGNfaikpZFxcbXUiXSxbMywzLCJcXGludF97SX1NX2koYylkXFxtdSJdLFswLDMsIngoYykiXSxbMCwwLCJ4KFxcaW90YShjX2opKSJdLFsyLDEsIntcXGFscGhhXzF9X3tjfSJdLFsyLDEsIntcXGFscGhhXzJ9X3tjfSIsMCx7ImN1cnZlIjotMn1dLFszLDIsIngoZV9qKSIsMV0sWzAsMSwiXFxpbnRfe0l9IHgoZV9qXntpfSlkXFxtdSIsMV0sWzMsMCwie1xcYWxwaGFfMX1fe2Nfan09e1xcYWxwaGFfMn1fe2Nfan0iXV0=
\[\begin{tikzcd}
	{x(\iota(c_j))} &&& {\int_{I}M_i(\iota(c_j))d\mu} \\
	\\
	\\
	{x(c)} &&& {\int_{I}M_i(c)d\mu}
	\arrow["{{\alpha_1}_{c_j}={\alpha_2}_{c_j}}", from=1-1, to=1-4]
	\arrow["{x(e_j)}"{description}, from=1-1, to=4-1]
	\arrow["{\int_{I} x(e_j^{i})d\mu}"{description}, from=1-4, to=4-4]
	\arrow["{{\alpha_1}_{c}}", from=4-1, to=4-4]
	\arrow["{{\alpha_2}_{c}}", curve={height=-12pt}, from=4-1, to=4-4]
\end{tikzcd}\]

Since $x(c)$ is a colimit, then by the universal property of $x(c)$, ${\alpha_2}_{c}= {\alpha_1}_{c}$.

Now we use the same diagram to show that every natural transformation between $x|_{C_1}$ and $\int_I M_i|_{C_1} d\mu$ can be extended to a natural transformation from $x$ to  $\int_I M_i d\mu$. Towards that, suppose we have such natural transformation $\alpha$ from $x|_{C_1}$ to $\int_I M_i|_{C_1} d\mu$. Take any $c \in C_2$, and cover it by elements of the form $\iota(c_j)$, then define $\alpha_c$, to be the unique map from $x(c)$ to $\int_I M_i(c) d\mu$ that exists, by the fact that $x(c)$ is a colimit.
\end{proof}

Before continuing, we should note that this construction would allow us to recover the generalised ultrastructure of a sober topological space. Let $X$ be a sober topological space, and let $\ca{O}(X)$ be its category of open sets, then the category of points of $X$ is simply the category of points ($\ca{J}$-continuous flat functors from $\ca{O}(X)$ to $\cat{Set}$).

Now, let us explain why they should have the same ultrastructure. Let us regard the points of a sober space, as left exact $J$-continuous functors from $\ca{O}(X)$ to $\cat{Set}$, that means that up to isomorphism, each  point $x \in X$ is the same thing as the functor $U \mapsto *$ iff $U \ni x$. This implies that for any ultrafilter $\mu$ on $X$, the pointwise ultraproduct $\int_X y d\mu$ is nothing but the ultrafilter $\mu$ itself regarded as a functor sending $ U \mapsto *$ iff $U \in \mu$. Now we can see that $\la{Hom}_{\la{Fun}(\ca{O}(X), \cat{Set})}(x,\int_X y d\mu)$ is non-empty iff the ultrafilter converges to $x$. And hence the two ultrastructures should agree (even if the topological space is not sober the argument should also work, the only difference is that the elements $x \in X$ are now only a full subcategory of the category of the points of the topos $Sh(X)$).

\section{2-Fibrations over the category $\cat{Top}$} \label{2-Fibrations}

We use the theory of fibred $2$-categories as introduced in \cite{buckley2014fibred}: Let $E \xrightarrow{P} B$, be a $2$-functor between $2$-categories $E$ and $B$ (in our study we shall focus on the case where $B$ is the category $\cat{Top}$  the $2$-category of topological spaces), we make the following definitions:

 We say that a $1$-morphism $f: x \rightarrow y$ is Cartesian if it satisfies the following two conditions:

 \begin{enumerate}

     \item For any  $u$ such that $Pf \circ u =Pg$, then there exists a unique $ \hat{u}$ such that $P \hat{u} =u$ and such that $g =f \circ \hat{u}$, diagrammatically this can be expressed as:

% https://q.uiver.app/#q=WzAsOSxbMCwwLCJ6Il0sWzUsMCwiUHoiXSxbMCwyLCJ4Il0sWzMsMiwieSJdLFs1LDIsIlB4Il0sWzgsMiwiUHkiXSxbMywxXSxbMiwxXSxbNCwxXSxbMCwyLCJ7XFxleGlzdHMgeyF9ICBcXGhhdHt1fX0iXSxbMCwzLCJnIl0sWzEsNCwidSJdLFsxLDUsIlBnIl0sWzIsMywiZiJdLFs0LDUsIlBmIl0sWzYsOCwiUCIsMCx7InN0eWxlIjp7InRhaWwiOnsibmFtZSI6Im1hcHMgdG8ifX19XV0=
\[\begin{tikzcd}
	z &&&&& Pz \\
	&& {} & {} & {} \\
	x &&& y && Px &&& Py
	\arrow["{{\exists {!}  \hat{u}}}", from=1-1, to=3-1]
	\arrow["g", from=1-1, to=3-4]
	\arrow["u", from=1-6, to=3-6]
	\arrow["Pg", from=1-6, to=3-9]
	\arrow["P", maps to, from=2-4, to=2-5]
	\arrow["f", from=3-1, to=3-4]
	\arrow["Pf", from=3-6, to=3-9]
\end{tikzcd}\]

\item Suppose that we have a $2$-cell  $\alpha$ between $u$ and $v$, and a $2$-cell $\beta$ between $h$ and $g$ such that $P\beta= Pf \circ \alpha$, then there exists a unique $2$-cell $\hat{\alpha}$ such that $\beta= f \circ \hat{\alpha}$, diagrammatically this can be expressed as:

% https://q.uiver.app/#q=WzAsOCxbMCwwLCJ6Il0sWzUsMCwiUHoiXSxbMCwyLCJ4Il0sWzMsMiwieSJdLFs1LDIsIlB4Il0sWzgsMiwiUHkiXSxbMywxXSxbNCwxXSxbMCwyLCJ7XFx3aWRlaGF0e3Z9fSIsMSx7ImN1cnZlIjoyfV0sWzAsMiwie1xcd2lkZWhhdHt1fX0iLDEseyJjdXJ2ZSI6LTJ9XSxbMCwzLCJoIiwwLHsiY3VydmUiOi0zfV0sWzAsMywiZyIsMix7ImN1cnZlIjoyfV0sWzEsNCwidiIsMSx7ImN1cnZlIjoyfV0sWzEsNCwidSIsMSx7ImN1cnZlIjotMn1dLFsxLDUsIlBnIiwyLHsiY3VydmUiOjJ9XSxbMSw1LCJQaCIsMCx7ImN1cnZlIjotMn1dLFsyLDMsImYiLDIseyJsYWJlbF9wb3NpdGlvbiI6NDB9XSxbNCw1LCJQZiIsMl0sWzYsNywiUCIsMCx7InN0eWxlIjp7InRhaWwiOnsibmFtZSI6Im1hcHMgdG8ifX19XSxbOSw4LCJ7XFxleGlzdHMgeyF9IFxcIFxcaGF0e1xcYWxwaGF9fSIsMix7InNob3J0ZW4iOnsic291cmNlIjoyMCwidGFyZ2V0IjoyMH19XSxbMTAsMTEsIlxcYmV0YSIsMix7InNob3J0ZW4iOnsic291cmNlIjoyMCwidGFyZ2V0IjoyMH19XSxbMTMsMTIsIlxcYWxwaGEiLDIseyJzaG9ydGVuIjp7InNvdXJjZSI6MjAsInRhcmdldCI6MjB9fV0sWzE1LDE0LCJ7UFxcYmV0YX0iLDIseyJzaG9ydGVuIjp7InNvdXJjZSI6MTUsInRhcmdldCI6MTV9fV1d
\[\begin{tikzcd}
	z &&&&& Pz \\
	&&& {} & {} \\
	x &&& y && Px &&& Py
	\arrow[""{name=0, anchor=center, inner sep=0}, "{{\widehat{v}}}"{description}, curve={height=12pt}, from=1-1, to=3-1]
	\arrow[""{name=1, anchor=center, inner sep=0}, "{{\widehat{u}}}"{description}, curve={height=-12pt}, from=1-1, to=3-1]
	\arrow[""{name=2, anchor=center, inner sep=0}, "h", curve={height=-18pt}, from=1-1, to=3-4]
	\arrow[""{name=3, anchor=center, inner sep=0}, "g"', curve={height=12pt}, from=1-1, to=3-4]
	\arrow[""{name=4, anchor=center, inner sep=0}, "v"{description}, curve={height=12pt}, from=1-6, to=3-6]
	\arrow[""{name=5, anchor=center, inner sep=0}, "u"{description}, curve={height=-12pt}, from=1-6, to=3-6]
	\arrow[""{name=6, anchor=center, inner sep=0}, "Pg"', curve={height=12pt}, from=1-6, to=3-9]
	\arrow[""{name=7, anchor=center, inner sep=0}, "Ph", curve={height=-12pt}, from=1-6, to=3-9]
	\arrow["P", maps to, from=2-4, to=2-5]
	\arrow["f"'{pos=0.4}, from=3-1, to=3-4]
	\arrow["Pf"', from=3-6, to=3-9]
	\arrow["{{\exists {!} \ \hat{\alpha}}}"', Rightarrow, from=1, to=0]
	\arrow["\beta"', Rightarrow, from=2, to=3]
	\arrow["\alpha"', Rightarrow, from=5, to=4]
	\arrow["{{P\beta}}"', Rightarrow, from=7, to=6]
\end{tikzcd}\]

 \end{enumerate}

 We are going to say that a  $2$-cell $\alpha$ between $f,g: x \rightarrow y$ is Cartesian if it is Cartesian when regarded as $1$-cell for the $1$-functor $P_{xy}$ between $E(x,y)$ and $B(Px,Py)$.

 Now we say  that $E \xrightarrow{P} B$ is a $2$-fibration if it satisfies the following conditions:

 \begin{enumerate}
     \item For every $e \in E$ and every $f:a \rightarrow Pe$ in $B$, and for every $x$ such that $Px=a$, there exists a Cartesian $1$-cell $\bar{f}:x \rightarrow e$ such that  $ P \bar{f} =f$.
     \item For any $g \in E$, and every $\alpha: Pf \longrightarrow Pg$, there exists a Cartesian $2$-cell $\sigma: f \longrightarrow g$, such that $P \sigma= \alpha$.
     \item The  horizontal composition of  Cartesian $2$-cells  is Cartesian.
 \end{enumerate}

 We say that a $2$-Fibration is cloven, if it is equipped with a choice of Cartesian lift of $1$ and $2$-morphisms. Suppose that we have $E \xrightarrow{p} B$ and $E^{'} \xrightarrow{P} B$ two categories fibred over $B$, we define a fibred (also called Cartesian in \cite{buckley2014fibred}) functor to be a $2$-functor $\ca{L}$ such that $p^{'} \ca{L}=p$ and such that it maps Cartesian $1$ and $2$-cells to Cartesian $1$ and $2$-cells, we say that such functor is cloven (or split as in \cite{buckley2014fibred}), if moreover the fibration is equipped with cleavages and this functor preserves these cleavages. We define a vertical or Cartesian natural transformation between such to be a $2$-natural transformation (we are working in the context of strict $2$-categories)  between two Cartesian functors $\ca{L}$ and $\ca{L^{'}}$ if  it satisfies $ P^{'} \circ \alpha =P$ (here $P$ and $P^{'}$ are to be understood as the identity natural transformations on their corresponding functors). We are going to denote by $\mathrm{ClovenCart}(E,E^{'})$ the category whose objects are cloven (split) fibred (Cartesian) strict $2$-functors between $E$ and $E'$, and whose  morphisms are natural transformations between these functors which are vertical (or Cartesian). We call a $2$-Fibration with cleavage a split fibration if moreover, the choices of Cartesian lifts are closed under all forms of compositions,  the exact definition can be found in \cite{buckley2014fibred}.

 \begin{theorem*}[Buckley]
     There is an equivalence of $3$-categories between the category, of split  $2$-Fibrations over $B$, with cloven Cartesian functors  and Cartesian (vertical) natural transformation and vertical (Cartesian) modifications, and the category of contravariant $2-$functors from $B$ to the category $2$-$\cat{Cat}$ of strict $2$-categories with $2$-functors,$2$-natural transformations, modifications and perturbations. 

 \end{theorem*}

 \begin{note*} \normalfont
     The reader may notice that we defined $\mathrm{ClovenCart}(E,E^{'})$ as a $1$-category, although it is possible to define it as a $2$-category with $2$-morphisms being modifications, the reason is that the fibrations we care about in this paper turn out to be contravariant $2$-functors from the $2$-category of topological spaces to the $2$-category of Categories, so in particular we are interested in case where every fibre is a $1$-category rather than a $2$-category. So, these are in  fact discrete $2$-fibrations as in the sense of \cite{lambert2020discrete}.
     \end{note*}

     \begin{theorem}
         Let $F: \cat{Top} \rightarrow \ca{B}$ be a pseudofunctor and $E \in \ca{B}$, then one may consider the lax\dash slice categories $\cat{Top}//E$ and , then both this categories are discrete $2$\dash fibred over $\cat{Top}^{co}$, with the fibration functor being the forgetful functor. 

     \end{theorem}

Let us remind the reader about the definition of the pseudo and lax slice categories. Let $F$ be a functor from $\cat{Top}$ to $\ca{B}$.  In both the  cases of $\cat{Top}//E$ and  $\cat{Top}/E$, the objects are pairs $(X,g)$ where $X$ is a topological space, and $g$  is a map in $\ca{B}$ from $F(X)$  to $E$, a morphism between $(X,g_1)$ and $(Y,g_2)$ is a pair $(h, \alpha)$, where $h$ is a continuous function from $X$ to $Y$ and $\alpha$ is a natural transformation (natural isomorphism in the case of $\cat{Top}/E$) from $g_1$ to $g_2 \circ F(h)$. A $2$-morphism in the category $\cat{Top}//E$ between  $(h_1, \alpha_1)$ and $(h_2, \alpha_2)$ is a $2$-morphism  $\alpha' \in \cat{Top}$, such that $(g_2 \circ F(\alpha')) * \alpha_1 = \alpha_2$, while we require $\alpha'$ to be invertible in $\cat{Top}/E$ (which is equivalent to having no $2$-morphisms since we only care about spaces which are $T_0$).

\begin{note*} \normalfont
While the lax slice category $\cat{Top}//E$ is opfibred over $\cat{Top}$, where $\cat{Top}$ is the $2$-category of topological spaces (corresponding to the pseudo-functor $X \mapsto \la{Hom}(F(X), E)$), the pseudo-slice is fibred over the $1$-category of topological space (since we have to delete all non-invertible two cells)

\end{note*}
\begin{note*}

\normalfont

The theorem above means that the functor $F^{co}: \cat{Top}^{co} \rightarrow \ca{B}^{co}$ is a discrete $2$-fibration. We may call the functor $F: \cat{Top} \rightarrow \ca{B}$ a discrete opfibration (or discrete cofibration, whichever terminology the reader wants to use, but in this case the word cofibration has nothing to do with its use in homotopy theory).

\end{note*}

There are two important cases of the construction above  which are of interest:

\begin{itemize}
    \item Let  $E$ be a topos, and consider the pseudofunctor $Sh:\ \cat{Top} \rightarrow \cat{Topos}$, that sends a topological space to its sheaf topos, more details regarding this functor can be found in \cite{maclane2012sheaves}. The objects of the lax slice category are pairs $(X,g)$, where $X$ is a topological space, and $g$ a geometric morphism from $Sh(X)$  to $E$ (we will often be writing this as $X \xrightarrow{g} E$ by abuse of notation), morphism between $(X,g_1)$ and $(Y,g_2)$ is a pair $(h, \alpha)$, where $h$ is a continuous function from $X$ to $Y$ and $\alpha$ is a natural transformation from $g_1$ to $g_2 \circ h$. The category $\cat{Top}/E$ is defined with the same objects, but we require $\alpha$ to be a natural iso in the definition
    
    . 

    \item The second construction is similar. Let $E$ be a topos and let $M_E$ be its category of points, then by identification of ultrapreorders and topological spaces, we may hence consider, the pseudo (lax)-slice category whose objects are left ultrafunctors from $X$ to $M_E$, and a morphism between $X \xrightarrow{f} M_E$ and $Y \xrightarrow{g} M_E$ is a continuous function $h$ from $X$ to $Y$, and a natural isomorphism (transformation) $\alpha$ from $f$ to $g \circ h$. We will be denoting these categories respectively  $\cat{Top}/M_{E}$ and $\cat{Top}//M_{E}$.

\end{itemize}

    \begin{note*} \normalfont We can show that $\cat{Top}/E$ (or $\cat{Top}//E$) are strict $2$-categories, using the fact that the inclusion of topological spaces inside toposes is a pseudo-functor, hence the isomorphisms between $f_* \circ g_*$ (or $g^* \circ f^*$) for example and $(f \circ g)_*$ (or $(f \circ g)^*$) satisfy nice coherence conditions.
\end{note*} 

\begin{note*} 

\normalfont
    It is also possible to define  colax comma categories ($\cat{Top}///M_E$ or $\cat{Top}///E$). Although such category appeared in Lurie \cite{lurie2018ultracategories} (namely the category $\mathrm{Comp}_M$), such categories do not play a role in our discussion, although we expect that the proof we give, may be reformulated in terms of the colax slice, instead of the lax slice which we used.

\end{note*}

\section{Conceptual Completeness of Geometric Logic} \label{The proof of conceptual completeness}

Let $E$ and $E^{'}$ be  Grothendieck toposes, let $M_{E}$ and $M_{E^{'}}$ be their categories of points (which we identify with $\ca{J}$-continuous flat functors from $C$, where $C$ is a site of definition, to $\cat{Set}$). Let $\mathrm{Lult}(\ca{A},\ca{B})$ denote the category of left ultrafunctors between generalised ultracategories $\ca{A}$ and $\ca{B}$.

Our main theorem is hence the following statement:

\begin{theorem} 
\label{Main theorem}
    Let $E$ and $E^{'}$ be two toposes with enough points, then there is an equivalence of categories between $\mathrm{Lult}(M_{E}, M_{E^{'}})$ and $\mathrm{Geom}(E,E^{'})$.
\end{theorem}

We have already seen that  $\cat{Top}//E$ (by the inclusion of topological spaces inside the category of toposes)  $\cat{Top}//M_{E}$ (by the inclusion of ultrapreorders inside generalised ultracategories) are  discrete $2$-opfibrations over the $2$-category of topological spaces.

Our proof is divided into showing three equivalences:
\begin{itemize}
    \item $\mathrm{Lult}(M_{E}, M_{E^{'}})\simeq \mathrm{ClovenCart}(\cat{Top}//M_E, \cat{Top}//M_{E^{'}})$. 

\item   $\cat{Top}//M_{E} \simeq \cat{Top}//E$ as $2$-categories $2$-fibred over $\cat{Top}$.

    \item $\mathrm{ClovenCart}(\cat{Top}//E, \cat{Top}//E^{'}) \simeq \mathrm{Geom}(E,E^{'}).$
\end{itemize}

\begin{note*} \normalfont
    We work with the lax slice to get the correct $2$-morphisms in $\cat{Topos}$. If we had worked  with the pseudo-slice categories $\cat{Top}/E$ and $\cat{Top}/M_E$, we would get only invertible natural transformations between geometric morphisms.
\end{note*}

\subsection{Equivalence between $\la{Lult}(M_E,M_{E^{'}}) \simeq \la{ClovenCart}(\cat{Top}//M_{E}, \cat{Top}//M_{E^{'}})$} \label{first equivalence}
We claim the following:

\begin{theorem}

Let $E$ and $E^{'}$ be toposes, and let $M_E$ and $M_{E^{'}}$ be their categories of points, respectively, then there exists an equivalence of categories between the category of left ultrafunctors from $M_E$ to $M_{E^{'}}$ and $\la{ClovenCart}(\cat{Top}//{M_{E}}, \cat{Top}//{M_{E^{'}}})$. \label{first equivalence theorem}

\end{theorem}

We start by considering the  functor from $\la{Lult}(M_E,M_{E^{'}})$ to  $\la{ClovenCart}(\cat{Top}//{M_{E}}, \cat{Top}//{M_{E^{'}}})$, which we will call $ \Theta$, defined on objects by sending a left ultrafunctor $F$, to the Cartesian functor defined on objects by composing with $F$ and on morphism by whiskering with $F$. We want to show that $\Theta$ is essentially surjective.

Suppose that we have a discrete topological space $I$, then we are going to denote by $(M_i)_{i \in I}$ the left ultrafunctor from $I$ to an ultracategory determined by $(M_i)_{i \in I}$, we can easily see that there is a bijection that can be promoted to an equivalence of categories between $\la{Lult}(I,M_E)$ and the category $(M_E)^I$.

Let $I$ be a set, $\mu$ be an ultrafilter on $I$. We define the following topological space $I_{\mu}$ as follows, as a set $I_{\mu} =I \bigsqcup \{p_\mu\}$ (so we add an additional point to $I$). We are going to give this space a topology using \ref{topology theorem}, as follows: Notice that the set of ultrafilters on the new space is the set of ultrafilters on $I$ (with the bijection that sends an ultrafilter on $I$ to its pushforward by the inclusion map of $I$ in $I_{\mu}$) plus the principal ultrafilter on $p_{\mu}$, we define a convergence relation such that the only ultrafilters that converge are the principal ultrafilters and the pushforward of $\mu$ by the inclusion map of $I$ in $I_{\mu}$, and we require this ultrafilter to converge to $p_{\mu}$. The proof that this construction satisfies \ref{topology theorem} is left to the reader.

We can describe this topology using open sets, the set of open subsets is given by $\{A: \ p_{\mu} \in A \ \la{and} \ A-\{p_{\mu}\} \in \mu \} \bigcup \{ A : A \subseteq I\}$, so in particular the subspace topology of $I$ is discrete. We also notice that if $\mu=*$ the unique ultrafilter on the $1$ element set, then $I_{*}$ is nothing but the Sierpinski space.

The space obtained is Hausdorff if the ultrafilter is non-principal. If the ultrafilter is principal then it's no longer even $T_1$ (its specialisation order will have an arrow from $p_{\delta_{i_0}}$ to $i_0$) but it's still going to be sober anyway. These new topological spaces have a nice property, a function $g$ from $I_{\mu}$ to $X$ is continuous iff the pushforward by $g$ of $\mu^{'}$ (which is the pushforward of $\mu$ by the inclusion of $I$ in $I_{\mu}$) converges to $g(p_{\mu})$.

Now suppose that we have a cloven Cartesian functor $\lu{F}$ in $\la{ClovenCart}(\cat{Top}//{M_{E}}, \cat{Top}//{M_{E^{'}}})$, we wish to construct a left ultrafunctor $f$ from $M_{E}$ to $M_{E^{'}}$, such that $\lu{F} \simeq f \circ -$. For every object $a \in M_{E}$ we  take the left ultrafunctor from the one-point space to $M_{E}$ defined by sending this object to $a$, its image by $\lu{F}$ is a left ultrafunctor from the one-point space to $M_{E^{'}}$, which can be identified with an element $b \in M_{E^{'}}$, hence we define $f(a)=b$.

Now on ``generalised'' morphisms, let us do the  following construction: Suppose that we have $a \in \la{Hom}(A, \int_I M_i d\mu)$, then we can regard $a$ as a left ultrafunctor from $I_{\mu}$ to $M_{E}$ and then we can define $f(a)$ to be simply $\lu{F}(a)$ (regarded as a left ultrafunctor), we want to make sure that such definition respects the source and the generalised target. But this can be deduced from the following diagram: 

\[
% https://q.uiver.app/#q=WzAsMyxbMCwwLCJJIl0sWzAsMiwiSV97XFxtdX0iXSxbMiwyLCJNX3tFXnsnfX0iXSxbMCwxLCJcXGlvdGEiLDJdLFsxLDIsImYoYSkiXSxbMCwyLCJmKE1faSkiXV0=
\begin{tikzcd}
	I \\
	\\
	{I_{\mu}} && {M_{E^{'}}}
	\arrow["\iota"', from=1-1, to=3-1]
	\arrow["{f(M_i)}", from=1-1, to=3-3]
	\arrow["{f(a)}", from=3-1, to=3-3]
\end{tikzcd}
\]

\[
% https://q.uiver.app/#q=WzAsMyxbMCwwLCIqIl0sWzAsMiwiSV97XFxtdX0iXSxbMiwyLCJNX3tFXnsnfX0iXSxbMCwxLCJcXG11IiwyXSxbMSwyLCJmKGEpIl0sWzAsMiwiZihBKSJdXQ==
\begin{tikzcd}
	{*} \\
	\\
	{I_{\mu}} && {M_{E^{'}}}
	\arrow["p_{\mu}"', from=1-1, to=3-1]
	\arrow["{f(A)}", from=1-1, to=3-3]
	\arrow["{f(a)}", from=3-1, to=3-3]
\end{tikzcd}
\]

Let us start by showing that $f$ is a left ultrafunctor:

For the unit axiom, we wish to show that the following diagram is commutative:
% https://q.uiver.app/#q=WzAsMyxbMiwwLCIxIl0sWzAsMiwiXFxtYXRocm17SG9tfShBLCBcXGludF9JIE1faSBkXFxtdSkiXSxbMywyLCJcXG1hdGhybXtIb219KGYoQSksXFxpbnRfSSBmKE1faSkgZFxcbXUpIl0sWzAsMSwiXFxrYXBwYV9BIl0sWzAsMiwiXFxrYXBwYV97ZihBKX0iLDJdLFsxLDIsImYiXV0=
\[\begin{tikzcd}
	&& 1 \\
	\\
	{\mathrm{Hom}(A, \int_I M_i d\mu)} &&& {\mathrm{Hom}(f(A),\int_I f(M_i) d\mu)}
	\arrow["{\kappa_A}", from=1-3, to=3-1]
	\arrow["{\kappa_{f(A)}}"', from=1-3, to=3-4]
	\arrow["f", from=3-1, to=3-4]
\end{tikzcd}
\]

To do this, notice that the left ultrafunctor from the Sierpinski space to the category $M_E$ completely determined by $\kappa$, is the only left ultrafunctor that factors through the one-point space i.e:

 % https://q.uiver.app/#q=WzAsMyxbMCwwLCJcXGxvbmdyaWdodGFycm93Il0sWzAsMiwiXFxidWxsZXQiXSxbMywyLCJNX0UiXSxbMCwxXSxbMSwyLCJBIiwyXSxbMCwyLCJcXGthcHBhX0EiLDJdXQ==
\[\begin{tikzcd}
	S \\
	\\
	* &&& {M_E}
	\arrow[from=1-1, to=3-1]
	\arrow["{\kappa_A}"', from=1-1, to=3-4]
	\arrow["A"', from=3-1, to=3-4]
\end{tikzcd}\]

hence its image by $\lu{F}$ also factors through the one-point set:

% https://q.uiver.app/#q=WzAsMyxbMCwwLCJcXGxvbmdyaWdodGFycm93Il0sWzAsMiwiKiJdLFszLDIsIk1fRSJdLFswLDFdLFsxLDIsImYoQSkiLDJdLFswLDIsIlxcbHV7Rn0oXFxrYXBwYV9BKT0gZihcXGthcHBhX0EpIiwxXV0=
\[\begin{tikzcd}
	S \\
	\\
	{*} &&& {M_E}
	\arrow[from=1-1, to=3-1]
	\arrow["{\lu{F}(\kappa_A)= f(\kappa_A)}"{description}, from=1-1, to=3-4]
	\arrow["{f(A)}"', from=3-1, to=3-4]
\end{tikzcd}\]
Hence, by uniqueness $f(\kappa_A)= \kappa_{f(A)}$.

Now for the composition axiom, suppose that we have a point $A$, a family of points $(M_i)_{i \in I}$, an ultrafilter $\mu$ on $I$ for each  $i \in I$ we have a set $X_i$ and an ultrafilter $\lambda_i$ on $X_i$, and finally for each $X_i$  a family of points $(B_{(i,x)})$. 

Let us define the following topological space $I_{\mu, (\lambda_i)}$: As a set this is $\{p_{\mu}\} \bigsqcup \coprod_{i \in I} \{p_{\lambda_i}\} \bigsqcup \coprod_{i \in I}X_i$. As usual, the set of ultrafilters on the underlying set of the space is $\{\delta_{p_\mu}\} \bigsqcup \beta I \bigsqcup \beta(\coprod_{I} X_i)$ . Let us denote by $\mu'$ and $\lambda_i'$ the pushforward of $\mu$, and $\lambda_i$ by the inclusion map of $I$  and $X_i$ respectively in $I_{\mu, (\lambda_i)}$.

Now, we define a topology on this space as follows:

\begin{itemize}
    \item $\lambda_i^{'}$  to converge to $p_{\lambda_i}$, and the ultrafilter $\mu^{'}$ to converge to $p_{\mu}$.
    \item  If $\mu$ is principal, say $\mu = \delta_i$, we want $\lambda_i^{'}$ not only to converge to $p_{\lambda_i}$, but also to $p_{\mu}$.

    \item If $\lambda_i$ is principal say at $x$, we want the principal ultrafilter $\delta_{x}$ to converge to $p_{\lambda_i}$, moreover if $\mu$ is also principal at $i$, then we also want $\delta_{x}$ to converge to $p_{\mu}$.

    \item We require the ultrafilter $\int_{I} \lambda^{'}_i d \mu$ to converge to $p_\mu$.

    \item  We require principal ultrafilters converge (at least) to their respective points:
    
\end{itemize}

It would be useful to draw the specialisation order of this space:

% https://q.uiver.app/#q=WzAsOCxbMCwyLCJcXG11Il0sWzQsMiwiXFxsYW1iZGFfe2l9Il0sWzQsMSwiXFxsYW1iZGFfaiJdLFs0LDMsIlxcbGFtYmRhX2siXSxbNCwwLCJcXHZkb3RzIl0sWzQsNCwiXFx2ZG90cyJdLFs4LDEsInkgXFxpbiBYX2oiXSxbOCwyLCJ4IFxcaW4gWF9JIl0sWzAsMSwiXFxsYXt3aGVufSBcXCAgXFxtdSBcXCBcXGxheyAgaXMgXFwgIHByaW5jaXBhbCBcXCBhdH0gXFwgaSIsMCx7InN0eWxlIjp7ImJvZHkiOnsibmFtZSI6ImRhc2hlZCJ9fX1dLFsxLDcsIlxcbGF7d2hlbn0gXFwgIFxcbGFtYmRhX2ogXFwgIFxcbGF7aXMgXFwgIHByaW5jaXBhbCBcXCBhdH0geCIsMCx7InN0eWxlIjp7ImJvZHkiOnsibmFtZSI6ImRhc2hlZCJ9fX1dLFsyLDYsIlxcbGF7d2hlbn0gXFwgIFxcbGFtYmRhX2ogXFwgIFxcbGF7aXMgXFwgIHByaW5jaXBhbCBcXCBhdH0gXFwgeSIsMCx7InN0eWxlIjp7ImJvZHkiOnsibmFtZSI6ImRhc2hlZCJ9fX1dXQ==
\[\begin{tikzcd}
	&&&& \vdots \\
	&&&& {p_{\lambda_j}} &&&& {y \in X_j} \\
	p_{\mu} &&&& p_{\lambda_{i}} &&&& {x \in X_i} \\
	&&&& {p_{\lambda_k}} \\
	&&&& \vdots
	\arrow["{\la{when} \  \lambda_j \  \la{is \  principal \ at} \ y}", dashed, from=2-5, to=2-9]
	\arrow["{\la{when} \  \mu \ \la{  is \  principal \ at} \ i}", dashed, from=3-1, to=3-5]
	\arrow["{\la{when} \  \lambda_j \  \la{is \  principal \ at} \  x}", dashed, from=3-5, to=3-9]
\end{tikzcd}\]

To show that this is a topological space, we have of course to use \cite{wyler1996convergence} (by checking lots of cases, this is done in \ref{topology}). Also, we can check that this topological space is sober.

Now we should notice that the data of a left ultrafunctor from this space to an ultracategory is exactly the data of things to be composed i.e.  a left ultrafunctor from $I_{\mu, (\lambda_i)_{i \in I}}$ is given by an element $A$ which is the image of the ultrafilter $\mu$, a family of objects $(M_i)$ which are the images of the family of ultrafilters $(\lambda_i)_{i \in I}$, and finally for each $i \in I$ a family of objects $(B_{(i,x)})_{x \in X_i} $, plus a morphism $a \in \la{Hom}(A, \int_I M_i d \mu)$, plus a family of morphisms $(g_i) \in \int_I \la{Hom}(M_i, B_{(i,x)}) d \mu$.

Now notice that there exists a continuous map from the space $I_{\int_I \iota_i \lambda_i d\mu}$ to this topological space defined by sending the point $\int_I \iota_i \lambda_i d\mu$ to the point $p_{\mu}$ and defined by sending every other element to itself. This map is clearly continuous; now let us have a look at the following diagram:

\[
 % https://q.uiver.app/#q=WzAsMyxbMCwwLCJJX3tcXGludF97SX1cXGxhbWJkYV9pZCBcXG11fSJdLFswLDIsIklfe1xcbXUsIChcXGxhbWJkYV9pKV97aSBcXGluIEl9fSJdLFs0LDIsIk1fRSJdLFsxLDIsIihhICAsKGdfaSkpIl0sWzAsMiwiKChnX2kpXFxjaXJjX3tcXG11fWEiXSxbMCwxLCJcXGlvdGEiXV0=
\begin{tikzcd}
	{I_{\int_{I}\iota_i\lambda_id \mu}} \\
	\\
	{I_{\mu, (\lambda_i)_{i \in I}}} &&&& {M_E}
	\arrow["\iota", from=1-1, to=3-1]
	\arrow["{(g_i)\circ_{\mu, (\lambda_i)}a}", from=1-1, to=3-5]
	\arrow["{(a  ,(g_i))}", from=3-1, to=3-5]
\end{tikzcd}
\]

Here we used the symbol $\circ_{\mu, (\lambda_i)}$ to denote the generalised composition.

Then our goal reduces to showing that the following diagram is commutative: 

\[
% https://q.uiver.app/#q=WzAsMyxbMCwwLCJJX3tcXGludF97SX1cXGxhbWJkYV9pZCBcXG11fSJdLFswLDIsIklfe1xcbXUsIChcXGxhbWJkYV9pKV97aSBcXGluIEl9fSJdLFs0LDIsIk1fRSJdLFsxLDIsImYoYSAsKGdfaSkpIl0sWzAsMiwiZigoZ19pKVxcY2lyY197XFxtdX1hKSJdLFswLDEsIlxcaW90YSJdXQ==
\begin{tikzcd}
	{I_{\int_{I}\iota_i\lambda_id \mu}} \\
	\\
	{I_{\mu, (\lambda_i)_{i \in I}}} &&&& {M_{E^{'}}}
	\arrow["\iota", from=1-1, to=3-1]
	\arrow["{f((g_i)\circ_{\mu, (\lambda_i)}a)}", from=1-1, to=3-5]
	\arrow["{(f(a) ,(f(g_i)))}", from=3-1, to=3-5]
\end{tikzcd}
\]

But this follows from the fact that $\lu{F}$ is a cloven fibred functor. Note of course that we need to make sure that $(f(a), (f(g_i))) =\lu{F}(a,(g_i))$, but this can be done by taking inclusion of $I_{\mu}$ and each $I_{\lambda_i}$ inside $I_{\mu,(\lambda_i)}$ and then using the fact that we have a cloven fibred functor. Thus this shows the composition axiom, namely that $f(g_i) \circ_{\mu, (\lambda_i)} f(a) = f((g_i) \circ_{\mu, (\lambda_i)} a) $.

Now we need to check the axiom concerning the change of base maps $\Xi$. Towards this, let us suppose that we have a map of sets $g$ from $J$ to $I$ and an ultrafilter $\mu$ on $J$. Notice that such map induces a continuous function from $J_{\mu}$ to $I_{g\mu }$ by sending each $j$ to $g(j)$ and sending $\mu$ to $g \mu$, let us call such map $g_\mu$.

So we get the following commutative diagram:

% https://q.uiver.app/#q=WzAsMyxbMCwwLCJJX3tcXG11fSJdLFs0LDAsIkpfe2cgXFxtdX0iXSxbMiw0LCJNX3tFfSJdLFswLDEsImdfe1xcbXV9Il0sWzEsMiwiYSJdLFswLDIsIlxcWGkoYSkiLDJdXQ==
\[\begin{tikzcd}
	{J_{\mu}} &&&& {I_{g \mu}} \\
	\\
	\\
	\\
	&& {M_{E}}
	\arrow["{g_{\mu}}", from=1-1, to=1-5]
	\arrow["{\Xi(a)}"', from=1-1, to=5-3]
	\arrow["a", from=1-5, to=5-3]
\end{tikzcd}\]

Now again, diagrammatically, the fact that  $f$ respects $X_i$ reduces to the commutativity of the following diagram:

% https://q.uiver.app/#q=WzAsMyxbMCwwLCJJX3tcXG11fSJdLFs0LDAsIkpfe2cgXFxtdX0iXSxbMiw0LCJNX3tFXnsnfX0iXSxbMCwxLCJnX3tcXG11fSJdLFsxLDIsImYoYSkiXSxbMCwyLCJmKFxcWGkoYSkpIiwyXV0=
\[\begin{tikzcd}
	{J_{\mu}} &&&& {I_{g \mu}} \\
	\\
	\\
	\\
	&& {M_{E^{'}}}
	\arrow["{g_{\mu}}", from=1-1, to=1-5]
	\arrow["{f(\Xi(a))}"', from=1-1, to=5-3]
	\arrow["{f(a)}", from=1-5, to=5-3]
\end{tikzcd}\]

But this follows from the fact that $\lu{F}$ is a cloven fibred functor.

Now we claim that for any topological space $X$ and any left ultrafunctor $g$, we have a natural isomorphism of left ultrafunctors between $f \circ g$ and $\lu{F}(g)$. First, notice that as functors they are equal. To show this, let $x \in X$, and let us have a look at the following diagram:

% https://q.uiver.app/#q=WzAsNCxbMCwwLCIqIl0sWzAsMywiWCJdLFszLDMsIk1fe0V9Il0sWzYsMywiTV97RV57J319Il0sWzAsMSwieCJdLFsxLDIsImciXSxbMCwyLCJnKHgpIiwyXSxbMiwzLCJmIl0sWzAsMywiXFxsdXtGfShnKHgpKSJdLFsxLDMsIlxcbHV7Rn0oZykiLDIseyJjdXJ2ZSI6NH1dXQ==
\[\begin{tikzcd}
	{*} \\
	\\
	\\
	X &&& {M_{E}} &&& {M_{E^{'}}}
	\arrow["x", from=1-1, to=4-1]
	\arrow["{g(x)}"', from=1-1, to=4-4]
	\arrow["{\lu{F}(g(x))}", from=1-1, to=4-7]
	\arrow["g", from=4-1, to=4-4]
	\arrow["{\lu{F}(g)}"', curve={height=24pt}, from=4-1, to=4-7]
	\arrow["f", from=4-4, to=4-7]
\end{tikzcd}\]

By the fact that $\lu{F}$ is a cloven fibred functor, we do get that $\lu{F}(g(x))=f(g(x))$. Now, let us compare their respective ultrastructures of $\lu{F}(g)$ and $f \circ g$. Recall that the ultrastructure of $g$ is completely determined by giving a generalised element in $\la{Hom}(g(x), \int_X g(y) d\mu)$ for every ultrafilter $\mu$ on $X$ converging to $x$, such that this map is $\kappa_{g(x)}$ if $\mu$ is the principal ultrafilter at $x$. Let $\mu$ be a convergent ultrafilter on $X$ that converges to $x$, and let $a$ be the element in $\la{Hom}(g(x), \int_X g(y) d\mu)$,  Now let us consider the topological space $X_{\mu}$, there is a continuous map from $X_{\mu}$ to $X$ defined by sending each element to itself and sending $\mu$ to $x$. So we can look at the following diagram: 

\[\
% https://q.uiver.app/#q=WzAsNCxbMCwwLCJYX3tcXG11fSJdLFswLDIsIlgiXSxbMywyLCJNX0UiXSxbNiwyLCJNX3tFXnsnfX0iXSxbMCwxLCJwIl0sWzEsMiwiZyJdLFsyLDMsImYiXSxbMCwyLCJnIFxcY2lyYyBwIiwxXSxbMCwzLCJcXGx1e0Z9KGcpIiwxXV0=
\begin{tikzcd}
	{X_{\mu}} \\
	\\
	X &&& {M_E} &&& {M_{E^{'}}}
	\arrow["p", from=1-1, to=3-1]
	\arrow["{g \circ p}"{description}, from=1-1, to=3-4]
	\arrow["{\lu{F}(g)}"{description}, from=1-1, to=3-7]
	\arrow["g", from=3-1, to=3-4]
	\arrow["f", from=3-4, to=3-7]
\end{tikzcd}
\]

The argument in the diagram shows that the ultrastructures of both ultrafunctors agree. So any cloven fibred functor corresponds exactly to the composition by a left ultrafunctor from $M_E$ to $M_{E^{'}}$. This shows that $\Theta$ is essentially surjective. It remains to show full faithfulness.

Suppose that  $\psi$ is a Cartesian natural transformation between the functors $-\circ f$ and $-\circ g$, we immediately see by a classic Yoneda argument that such $\psi$ is induced by a natural transformation $\alpha$ between the underlying functors of $f$ and $g$, so it remains to show  that $\alpha$ is actually a natural transformation of left ultrafunctors, i.e. $\alpha=(\alpha^1,\alpha^2)$ (by abuse of language, so we don't write something like $\gamma =(\alpha_1,\alpha_2)$ and the underlying natural transformation of $\gamma$ is $\alpha$). Let us define $\alpha^1$ and $\alpha^2$ and show they satisfy the desired properties. The family $\alpha^2$ of the left ultrastructure of $\alpha$ is simply representables at the elements $(\alpha_A)_{A \in M_E}$, while the family $\alpha^1$ is constructed as the following compositions: $\alpha^1(t)= \beta( t, \alpha_{M_i})$, or using a less demanding nomenclature, $\alpha^1(t)= (\alpha_{M_i}) \circ_{\mu} t$.

Now it remains to show that such definition is really a natural transformation of left ultrafunctors as defined, towards that we are going to give a sketch of the proof that this construction satisfies the $5$ axioms:

\begin{itemize}

\item For axiom $1$ and $3$, let $A \in M_E$ and let us consider the following diagram:

% https://q.uiver.app/#q=WzAsMyxbMCwwLCIqIl0sWzIsMCwiTV97RX0iXSxbNCwwLCJNX3tFJ30iXSxbMCwxLCJBIl0sWzEsMiwiRyIsMix7Im9mZnNldCI6Mn1dLFsxLDIsIkYiLDAseyJvZmZzZXQiOi0zfV0sWzUsNCwiXFxhbHBoYSIsMCx7InNob3J0ZW4iOnsic291cmNlIjoyMCwidGFyZ2V0IjoyMH19XV0=
\[\begin{tikzcd}
	{*} && {M_{E}} && {M_{E'}}
	\arrow["A", from=1-1, to=1-3]
	\arrow[""{name=0, anchor=center, inner sep=0}, "G"', shift right=2, from=1-3, to=1-5]
	\arrow[""{name=1, anchor=center, inner sep=0}, "F", shift left=3, from=1-3, to=1-5]
	\arrow["\alpha", between={0.2}{0.8}, Rightarrow, nfold, from=1, to=0]
\end{tikzcd}\]

 Since the composition $\alpha \circ A$ is a natural transformation of left ultrafunctors we get the commutativity of the following diagrams:

 % https://q.uiver.app/#q=WzAsNCxbMCwwLCJcXGxhe0hvbX0oRyhBKSxcXGludF9JIE1faSBkZlxcbXUpICJdLFswLDQsIlxcbGF7SG9tfShGKEEpLFxcaW50X0kgTV9pIGRmXFxtdSkgIl0sWzUsNCwiXFxsYXtIb219KEYoQSksXFxpbnRfSiBNX3tmKGopfSBkXFxtdSkiXSxbNSwwLCJcXGxhe0hvbX0oRyhBKSwgXFxpbnRfSiBNX3tmKGopfSBkXFxtdSJdLFswLDEsIlxcYWxwaGFeMl97QSwgKE1faSl9ICJdLFsxLDIsIlxcWGlfe2Z9Il0sWzAsMywiXFxYaV97Zn0iLDJdLFszLDIsIlxcYWxwaGFeMl97QSwoTV97KGYoail9KX0iLDJdXQ==
\[\begin{tikzcd}
	{\la{Hom}(G(A),\int_I M_i df\mu) } &&&&& {\la{Hom}(G(A), \int_J M_{f(j)} d\mu)} \\
	\\
	\\
	\\
	{\la{Hom}(F(A),\int_I M_i df\mu) } &&&&& {\la{Hom}(F(A),\int_J M_{f(j)} d\mu)}
	\arrow["{\Xi_{f}}"', from=1-1, to=1-6]
	\arrow["{\alpha^2_{A, (M_i)} }", from=1-1, to=5-1]
	\arrow["{\alpha^2_{A,(M_{f(j)})}}"', from=1-6, to=5-6]
	\arrow["{\Xi_{f}}", from=5-1, to=5-6]
\end{tikzcd}\]

and  

\[
\adjustbox{max width =\textwidth}{
% https://q.uiver.app/#q=WzAsNCxbMCwwLCJcXGxhe0hvbX0oRyhBKSxcXGludF9JIE1faSBkXFxtdSkgXFx0aW1lcyBcXGludF9JIFxcbGF7SG9tfShNX2ksIFxcaW50X3tYX2l9Tl97KGkseCl9ZCBcXGxhbWJkYV9pKSBkXFxtdSJdLFswLDQsIlxcbGF7SG9tfShGKEEpLFxcaW50X0kgTV9pIGRcXG11KSBcXHRpbWVzIFxcaW50X0kgXFxsYXtIb219KE1faSwgXFxpbnRfe1hfaX1OX3soaSx4KX1kIFxcbGFtYmRhX2kpIGRcXG11Il0sWzUsNCwiXFxsYXtIb219KEYoQSksXFxpbnRfe1xcY29wcm9kX3tJfVhfaX1OX3soaSx4KX1kIFxcaW50X3tJfVxcaW90YV9pIFxcbGFtYmRhX2kgZFxcbXUiXSxbNSwwLCJcXGxhe0hvbX0oRyhBKSxcXGludF97XFxjb3Byb2Rfe0l9WF9pfU5feyhpLHgpfWQgXFxpbnRfe0l9XFxpb3RhX2kgXFxsYW1iZGFfaSBkXFxtdSJdLFswLDEsIlxcYWxwaGFeMl97QSwgKE1faSl9IFxcdGltZXMgXFxsYXtpZH0iXSxbMSwyLCJcXGJldGEiXSxbMCwzLCJcXGJldGEiLDJdLFszLDIsIlxcYWxwaGFeMl97QSwoTl97KGkseCl9KX0iLDJdXQ==
\begin{tikzcd}
	{\la{Hom}(G(A),\int_I M_i d\mu) \times \int_I \la{Hom}(M_i, \int_{X_i}N_{(i,x)}d \lambda_i) d\mu} &&&&& {\la{Hom}(G(A),\int_{\coprod_{I}X_i}N_{(i,x)}d \int_{I}\iota_i \lambda_i d\mu)} \\
	\\
	\\
	\\
	{\la{Hom}(F(A),\int_I M_i d\mu) \times \int_I \la{Hom}(M_i, \int_{X_i}N_{(i,x)}d \lambda_i) d\mu} &&&&& {\la{Hom}(F(A),\int_{\coprod_{I}X_i}N_{(i,x)}d \int_{I}\iota_i \lambda_i d\mu)}
	\arrow["\beta"', from=1-1, to=1-6]
	\arrow["{\alpha^2_{A, (M_i)} \times \la{id}}", from=1-1, to=5-1]
	\arrow["{\alpha^2_{A,(N_{(i,x)})}}"', from=1-6, to=5-6]
	\arrow["\beta", from=5-1, to=5-6]
\end{tikzcd}
}
\]
Which is what we wanted to show.
    \item For axiom $2$

Let $(M_i)$ be a family of objects, and let us consider the following diagram:

    % https://q.uiver.app/#q=WzAsNCxbMywwLCJJIl0sWzAsMCwiSiJdLFs1LDAsIk1fe0V9Il0sWzcsMCwiTV97RSd9Il0sWzEsMCwiZiJdLFswLDIsIihNX2kpIl0sWzIsMywiRyIsMix7Im9mZnNldCI6Mn1dLFsyLDMsIkYiLDAseyJvZmZzZXQiOi0zfV0sWzcsNiwiXFxhbHBoYSIsMCx7InNob3J0ZW4iOnsic291cmNlIjoyMCwidGFyZ2V0IjoyMH19XV0=
\[\begin{tikzcd}
	J &&& I && {M_{E}} && {M_{E'}}
	\arrow["f", from=1-1, to=1-4]
	\arrow["{(M_i)}", from=1-4, to=1-6]
	\arrow[""{name=0, anchor=center, inner sep=0}, "G"', shift right=2, from=1-6, to=1-8]
	\arrow[""{name=1, anchor=center, inner sep=0}, "F", shift left=3, from=1-6, to=1-8]
	\arrow["\alpha", between={0.2}{0.8}, Rightarrow, nfold, from=1, to=0]
\end{tikzcd}\]

Since the composition $\alpha \circ (M_i)$ is a left ultrafunctor we get that for any $A \in M_{E'}$, the commutativity of the following diagram: \[
% https://q.uiver.app/#q=WzAsNCxbMCwwLCJcXGxhe0hvbX0oQSxcXGludF9JIEYoTV9pKSBkZlxcbXUpICJdLFswLDQsIlxcbGF7SG9tfShBLFxcaW50X0kgRyhNX2kpIGRmXFxtdSkgIl0sWzUsNCwiXFxsYXtIb219KEEsXFxpbnRfSiBHKE1fe2Yoail9KSBkXFxtdSkiXSxbNSwwLCJcXGxhe0hvbX0oQSwgXFxpbnRfSiBGKE1fe2Yoail9KSBkXFxtdSkiXSxbMCwxLCJcXGFscGhhXjFfe0EsIChNX2kpfSAiXSxbMSwyLCJcXFhpX3tmfSJdLFswLDMsIlxcWGlfe2Z9IiwyXSxbMywyLCJcXGFscGhhXjFfe0EsKE1feyhmKGopfSl9IiwyXV0=
\begin{tikzcd}
	{\la{Hom}(A,\int_I F(M_i) df\mu) } &&&&& {\la{Hom}(A, \int_J F(M_{f(j)}) d\mu)} \\
	\\
	\\
	\\
	{\la{Hom}(A,\int_I G(M_i) df\mu) } &&&&& {\la{Hom}(A,\int_J G(M_{f(j)}) d\mu)}
	\arrow["{\Xi_{f}}"', from=1-1, to=1-6]
	\arrow["{\alpha^1_{A, (M_i)} }", from=1-1, to=5-1]
	\arrow["{\alpha^1_{A,(M_{f(j)})}}"', from=1-6, to=5-6]
	\arrow["{\Xi_{f}}", from=5-1, to=5-6]
\end{tikzcd}
\]

which is what we wanted to show.

  \item For axiom $4$, suppose that we are given a family of sets $X_i$, and for each $i$, a family of objects of $M_E$, $(N_{i,x})$ indexed by $X_i$. Let us consider the discrete topological space $\coprod_{i \in I} X_i$, and the following diagram:

\[\begin{tikzcd}
	{\coprod_{i \in I}X_i} && {M_{E}} && {M_{E'}}
	\arrow["{((N_{(i,x)})_{x \in X_i})_{i\in I}}", from=1-1, to=1-3]
	\arrow[""{name=0, anchor=center, inner sep=0}, "G"', shift right=2, from=1-3, to=1-5]
	\arrow[""{name=1, anchor=center, inner sep=0}, "F", shift left=3, from=1-3, to=1-5]
	\arrow["\alpha", between={0.2}{0.8}, Rightarrow, nfold, from=1, to=0]
\end{tikzcd}\]

Now by the fact that $\alpha \circ (N_{(i,x)})$ is a natural transformation of left ultrafunctors we get the commutativity of the following diagram:

\[
\adjustbox{scale =0.89}{\begin{tikzcd}
	{\la{Hom}(A,\int_I B_i d\mu)\times \int_I\la{Hom}(B_i, \int_{X_i} F(N_{(i,x)}) d\lambda_i)d\mu} &&&& {\la{Hom}(A,\int_{\coprod_{I}X_i}F(N_{(i,x)})d\int_I \iota_i \lambda_i d\mu)} \\
	\\
	\\
	{\la{Hom}(A,\int_I B_i d\mu)\times \int_I\la{Hom}(B_i, \int_{X_i} G(N_{(i,x)}) d\lambda_i)d\mu} &&&& {\la{Hom}(A,\int_{\coprod_{I}X_i}G(N_{(i,x)})d\int_I \iota_i \lambda_i d\mu)}
	\arrow["\beta"', from=1-1, to=1-5]
	\arrow["{\la{id} \times \int_I \alpha^1_{B_i,(N_{(i,x)})}d\mu}", from=1-1, to=4-1]
	\arrow["{\alpha^1_{A,(N_{(i,x)})}}"', from=1-5, to=4-5]
	\arrow["\beta", from=4-1, to=4-5]
\end{tikzcd}}\]

    \item It remains to check axiom $5$, which is the main property of natural transformations of left ultrafunctors. Our goal is to show that the following diagram is commutative:

\[\begin{tikzcd}
	{\la{Hom}(A,\int_IM_i d\mu)} && {\la{Hom}(F(A),\int_IF(M_i) d\mu)} \\
	\\
	{\la{Hom}(G(A), \int_IG(M_i)d\mu)} && {\la{Hom}(F(A),\int_IG(M_i) d\mu)}
	\arrow["\zeta", from=1-1, to=1-3]
	\arrow["{\zeta^{'}}", from=1-1, to=3-1]
	\arrow["{\alpha^1_{F(A),(M_i)}}"{description}, from=1-3, to=3-3]
	\arrow["{\alpha^2_{A,(G(M_i))}}", from=3-1, to=3-3]
\end{tikzcd}\]

Here $\zeta$  and $\zeta'$ are the left ultrastructure functions of $f$ and $g$ respectively.
 
In more simple terms the relation that we wish to show is that for any $x \in \la{Hom}(A, \int_I M_i d\mu)$ we have $ G(x) \circ_* \alpha_A = (\alpha_{M_i}) \circ_{\mu} F(x)$.

    Suppose that we have an element $x \in \la{Hom}(A,\int_IM_i d\mu)$, recall that this is the same thing as a left ultrafunctor $I_{\mu} \xrightarrow{x} M_{E}$. Let us take the following diagram:

  % https://q.uiver.app/#q=WzAsMyxbMCwwLCJJX3tcXG11fSJdLFszLDAsIk1fe0V9Il0sWzYsMCwiTV97RSd9Il0sWzAsMSwieCJdLFsxLDIsImciLDIseyJvZmZzZXQiOjN9XSxbMSwyLCJmIiwwLHsib2Zmc2V0IjotM31dLFs1LDQsIlxcYWxwaGEiLDAseyJzaG9ydGVuIjp7InNvdXJjZSI6MjAsInRhcmdldCI6MjB9fV1d
\[\begin{tikzcd}
	{I_{\mu}} &&& {M_{E}} &&& {M_{E'}}
	\arrow["x", from=1-1, to=1-4]
	\arrow[""{name=0, anchor=center, inner sep=0}, "G"', shift right=3, from=1-4, to=1-7]
	\arrow[""{name=1, anchor=center, inner sep=0}, "F", shift left=3, from=1-4, to=1-7]
	\arrow["\alpha", Rightarrow, nfold, from=1, to=0]
\end{tikzcd}\]

    By the fact that $\psi$ is a Cartesian natural transformation between  $\cat{Top}//M_{E}$ and $\cat{Top}//M_{E'}$, and since $\psi_{x} = (\mathrm{id}_{I_{\mu}},\alpha \circ x)$ (the fact that the first component is $\mathrm{id}_{I_{\mu}}$ follows from the fact that $\psi$ is Cartesian (vertical)), we get that $\alpha \circ x$ is a natural transformation of left ultrafunctors, this would imply that the following diagram is commutative:

\[\begin{tikzcd}
	{\la{Hom}(p_\mu,\int_Ii \ d\mu) \simeq *} &&&& {\la{Hom}(F(A),\int_IF(M_i)d\mu)} \\
	\\
	\\
	{\la{Hom}(G(A),\int_IG(M_i)d\mu)} &&&& {\la{Hom}(F(A),\int_IG(M_i)d\mu)}
	\arrow["{F(x)}"', from=1-1, to=1-5]
	\arrow["{G(x)}", from=1-1, to=4-1]
	\arrow["{(\alpha_{M_i}) \circ_{\mu} -}"', from=1-5, to=4-5]
	\arrow["{ -\circ_{*}\alpha_{A}}", from=4-1, to=4-5]
\end{tikzcd}\]
    
\begin{comment}
    This proof can be expanded if time allows
\end{comment}

This should imply that $G(x) \circ_* \alpha_A =(\alpha_{M_i}) \circ_{\mu} F(x)$, which is what we wanted to show.

\end{itemize}

This shows that the functor is full, seeing that it's faithful is trivial, so we have shown that $\Theta$ is fully faithful and essentially surjective, hence it's an equivalence of categories.

\begin{note*}
    \normalfont The reader should notice that this proof is independent from the fact that $M_E$ and $M_{E'}$ are the categories of points of toposes, i.e. for any two generalised ultracategories $\ca{A}$ and $\ca{B}$ we have an equivalence of categories between $\la{Lult}(\ca{A},\ca{B})$ and $\la{ClovenCart}(\cat{Top}//\ca{A},\cat{Top}//\ca{B})$.
\end{note*}

Using this we can conclude a further result:

\begin{theorem}
    The $2$-category of topological spaces is dense inside the $2$-category of generalised ultracategories.  
\end{theorem}

\begin{proof}
    Suppose that we have two generalised ultracategories $\ca{A}$ and $\ca{B}$ such that for any topological space $X$, we have an equivalence of categories between $\la{Lult}(X,\ca{A})$ and $\la{Lult}(X,\ca{B})$, which is natural in $X$, this implies that the two pseudo-functors from $\cat{Top}^{op}$ to $\cat{Cat}$, defined by $X \mapsto \la{Lult}(X,\ca{A})$ and $X \mapsto \la{Lult}(X,\ca{B})$ are equivalent, and by using the result \cite[2.2.11]{buckley2014fibred}, this implies that the two discrete opfibrations $\cat{Top}//\ca{A}$ and $\cat{Top}//\ca{B}$ are equivalent, and hence using the note above $\ca{A}$ and $\ca{B}$ are equivalent as generalised ultracategories.

\end{proof}

\subsection{Equivalence between $\cat{Top}//M_E$ and $ \cat{Top}//E$}\label{second equivalence}

We want to show the following theorem:

\begin{theorem}
Let $E$ be a topos with enough points, and let $M_{E}$ be its  category of points. There is an equivalence of $2$-categories fibred over $\cat{Top}$, between the category $\cat{Top}//E$ and $\cat{Top}//M_E$. \label{second equivalence theorem}
\end{theorem}

This allows us to deduce that, and in the same  setting as the previous theorem:

\begin{theorem}
     There is an equivalence of categories  between $ \la{ClovenCart}(\cat{Top}//M_E, \cat{Top}//M_{E^{'}})$ and $ \la{ClovenCart}(\cat{Top}//E, \cat{Top}//E^{'})$.
\end{theorem}

First, we show the equivalence of fibres:

Let $X$ be a  topological  space, then there is an equivalence of categories between the category $(\cat{Top}//{E})_X$ and the category $(\cat{Top}//{M_E})_X$. Let $C$ be a site of definition for the topos $E$. We may assume that $C \subseteq E$ and that $C$ has all finite limits, in this case, the condition of being flat is equivalent to being left exact i.e. preserving finite limits.

The proof is the following series of equivalences $$\la{Lult}(X, \ca{J}\dash \la{continuous \ lex}(C, \cat{Set})) \simeq \ca{J} \dash \la{continuous \ lex}(C , \la{Lult}(X,\cat{Set})) \simeq $$  $$\ca{J} \dash \la{continuous \ lex}(C , Sh(X)) \simeq \la{Geom}(X , E) $$ The only equivalence that needs verification is the first equivalence (the last equivalence is a classic result in topos theory, and the second follows from the fact that $Sh(X) \simeq \mathrm{Lult}(X,\cat{Set})$).

 Recall that for any topological space, the sheaf topos $Sh(X)$ always has enough points, and its category of points is the soberification of $X$. And as we stated before every category of points of a sheaf topos has a natural generalised ultrastructure, which coincides with the ultrastructure of the topological space (if it's sober, otherwise it's going to be the ultrastructure of its soberification) given by $\Hom(A , \int_I M_i d\mu)= \{*\}$ iff the ultrafilter $M\mu$ converges to $A$ ($M \mu$ is the pushforward of $\mu$ by the map $i \mapsto M_i$), and of course the underlying category of such category is the topological space with the specialisation order.

First, we show the equivalence $\la{Lult}(X, \la{Fun}(C,\cat{Set})) \simeq \la{Fun}(C, \la{Lult}(X,\cat{Set}))$. Suppose that we have a left ultrafunctor $F$ from $X$ to the category of  functors from $C$ to $\cat{Set}$.

We start by noticing that when we have a topological space $X$, a left ultrafunctor from $X$ to any ultracategory (not a generalised one) is given by defining for every ultrafilter $\mu$ on $X$ converging to $x$ a map $\sigma_{\mu}$ from $F(x)$ to $\int_X F(y) d \mu$ satisfying certain compatibility conditions, these elements $\sigma_{\mu}$ are the images by $F$ of the unique element in $\la{Hom}(x, \int_X y d\mu)$.

This is not the entire data of left ultrafunctors as we defined it,  but in the case of ultrafunctors from topological spaces, these maps are enough to uniquely determine the entire data. Let us explain further, suppose that $X$ is a topological space, and suppose that for every converging ultrafilter $\mu$ on $X$ converging to $x$ (not necessarily unique in this regard), we have a map $\sigma_{\mu}$ from $F(x)$ to $\int_X F(y)  d \mu$ satisfying the following compatibility conditions:

Let $I$ be a set and let $\mu$ be an ultrafilter on $I$ and $f$ a map of sets from $I$ to $X$ (this is the same as a family of points of $X$ indexed by $I$), such that $f\mu$ converges to $x$, we may define $\sigma_{\mu}$ from $F(x)$ to $\int_I F(f(i)) d \mu$ by $\sigma_{\mu}= \Delta_{f, \mu} \circ \sigma_{f\mu} = \Xi_{f,\mu}(\sigma_{f\mu})$.

The compatibility conditions can be expressed by the following diagrams:

    \[
% https://q.uiver.app/#q=WzAsNCxbMCwwLCJGKHgpIl0sWzQsMCwiXFxpbnRfSSBGKHpfaSkgZFxcbXUiXSxbNCwzLCJcXGludF9JIFxcaW50X3tYX2l9IEYoeSkgZCBcXGxhbWJkYV9pIGRcXG11Il0sWzAsMywiXFxpbnRfe1xcY29wcm9kX3tJfVhfaX1GKHkpIGQgXFxpbnRfe0l9IFxcaW90YV9pIFxcbGFtYmRhX2kgZCBcXG11Il0sWzAsMywiXFxzaWdtYV97XFxpbnRfSSBcXGlvdGFfaSBcXGxhbWJkYV9pIGQgXFxtdX0iLDJdLFszLDIsImEiLDJdLFswLDEsIlxcc2lnbWFfe1xcbXV9Il0sWzEsMiwiXFxpbnRfSSBcXHNpZ21hX3tcXGxhbWJkYV9pfWRcXG11Il1d
\begin{tikzcd}
	{F(x)} &&&& {\int_I F(z_i) d\mu} \\
	\\
	\\
	{\int_{\coprod_{I}X_i}F(y) d \int_{I} \iota_i \lambda_i d \mu} &&&& {\int_I \int_{X_i} F(y) d \lambda_i d\mu}
	\arrow["{\sigma_{\mu}}", from=1-1, to=1-5]
	\arrow["{\sigma_{\int_I \iota_i \lambda_i d \mu}}"', from=1-1, to=4-1]
	\arrow["{\int_I \sigma_{\lambda_i}d\mu}", from=1-5, to=4-5]
	\arrow["a"', from=4-1, to=4-5]
\end{tikzcd}
\]

 %https://q.uiver.app/#q=WzAsMixbMCwwLCJGKHgpIl0sWzIsMCwiXFxpbnRfKiBGKHgpIGQqIl0sWzAsMSwiXFxzaWdtYV8qIl0sWzEsMCwiXFxlcHNpbG9uX3sqLCp9IiwwLHsiY3VydmUiOi00fV1d
\[\begin{tikzcd}
	{F(x)} && {\int_* F(x) d*}
	\arrow["{\sigma_*}", from=1-1, to=1-3]
	\arrow["{\epsilon_{*,*}}", curve={height=-24pt}, from=1-3, to=1-1]
\end{tikzcd}\]

Here $a$ is the colax associator which is the composite of the categorical Fubini transform with the ultraproduct diagonal map, and $z_i$ are the limits of the pushforwards of the ultrafilters $\lambda_i$. This entire discussion says that left ultrafunctors from a topological space $X$ to arbitrary ultracategories are completely determined by the images of $\la{Hom}(x,\int_Xyd\mu)$ (it's enough to restrict our attention to ultrafilters on $X$).

Now suppose that $\ca{A}$ is a generalised ultracategory for which the generalised ultrastructure is given by taking $\mathrm{Hom}(A,\int_I M_i d\mu)$ to be $\mathrm{Hom}_{\ca{B}}(A,\int_I M_i d\mu)$ i.e. the usual Hom set inside another ultracategory $\ca{B}$ such that $ \ca{A} $ is a full subcategory of $\ca{B}$ (i.e. the construction \ref{not very crucial theorem}). Then it is easy to notice that given a  generalised  ultracategory $\ca{C}$, there is a one-to-one correspondence (which extends to isomorphism of categories) between left ultrafunctors from $\ca{C}$ to $\ca{A}$ and between left ultrafunctors from $\ca{C}$ to $\ca{B}$ that factor through the inclusion, so the description that we gave for left ultrafunctors from a topological spaces to a regular ultracategory, still holds when we have a left ultrafunctor from a topological space $X$ to a generalised ultracategory constructed using \ref{not very crucial theorem}.

Now we notice that the ultrastructure of the category $\la{Fun}(C,\cat{Set})$ is defined by taking pointwise ultraproducts. That means that given a left ultrafunctor $F$ from $X$ to $\la{Fun}(C,\cat{Set})$, for every  $c \in C$  the functor $F_c$ given by $x \mapsto  F(x)(c)$ has a natural left ultrastructure (so it is a left ultrafunctor from $X$ to $\cat{Set}$), given by defining for every ultrafilter $\mu$ on $X$ converging to  $x$, the maps $\sigma^{'}_{\mu}={\sigma_{\mu}}_c$ (this makes sense since $\sigma_{\mu}$ is a natural transformation). It is easily verifiable that these maps $\sigma^{'}_{\mu}$ indeed satisfy the compatibility axioms necessary in the definition of left ultrafunctors (this reduces to the fact that ultrastructure in the category of functors to $\cat{Set}$ is computed pointwise). On the other hand, suppose that we have a functor $\ca{F}$ from $C$ to the category of left ultrafunctors from $X$ to $\cat{Set}$. Then we can give the functor $F$ defined from $X$ to the category of functors from $C$ to $\cat{Set}$, by $F(x)(c)=\ca{F}(c)(x)$ an ultrastructure by defining the map $\sigma_{\mu}$ from $F(x)$ to $\int_X F(y) d\mu$ by setting $(\sigma_{\mu})_{c} ={\sigma_c}_{\mu}$ here $\sigma_{c}$ is the left ultrafunctor map assigned to the functor $\ca{F}(c)$, and one can easily verify that this gives $F$ a left ultrastructure. And one can see that these two processes are part of an equivalence between $\la{Lult}(X , \la{Fun}(C, \cat{Set}))$ and $\la{Fun}(C , \la{Lult}(X, \cat{Set}))$.

Now we claim that this last equivalence restricts to an equivalence between $\ca{J}\dash\la{continuous \ lex}$ functors from $C$ to the category $\la{Lult}(X,\cat{Set})$ (this makes sense since the last category is equivalent to the category of sheaves of sets on $X$), and the category of left ultrafunctors from $X$ to $\ca{J}\dash\la{continuous \ lex}$ functors from $C$ to $\cat{Set}$. If we drop the requirement of being $\ca{J}$-continuous, then the equivalence is evident (finite limits are computed pointwise). Now we remind that the property of being $\ca{J}$-continuous, is the requirement to send covering sieves to jointly epimorphic families \cite{maclane2012sheaves}. Suppose that we have a  left ultrafunctor $F$ from $X$ to the category of $\ca{J}\dash\la{continuous \ lex}$ functors from $C$  to $\cat{Set}$, then  the corresponding  functor $\ca{F}$ from $C$ to the category $\la{Lult}(X,\cat{Set}) \simeq Sh(X)$ is $\ca{J}\dash\la{continuous \ lex}$. As we have seen, $\ca{F}$ is clearly lex. To see why it must be also $J$-continuous. Let $c$ be an object of $C$ and suppose that $S$ is a covering sieve for $c$, for every $x \in X$, $\ca{F}(S)(x)$ is jointly epimorphic and hence since the sheaf  topos $Sh(X)$ has enough points, epimorphisms can be tested stalkwise, and hence we can deduce that $\ca{F}(S)$  is jointly epimorphic, and hence the equivalence that we already had restricts.

Now, since we have shown that all fibres are isomorphic, we need to show that this extends to a morphism of discrete $2$-fibrations (contravariant pseudofunctor from $\cat{Top}^{op}$ to the $2$-category of categories) the proof can be deduced by inspecting the following diagram:

% https://q.uiver.app/#q=WzAsOCxbMCwwLCJcXG1hdGhybXtMZWZ0dWx0cmFmdW5jdG9yc30oWCwgSlxcdGV4dHstfVxcbWF0aHJte2NvbnRpbnVvdXMgXFwgbGV4fShDLFxcbWF0aHNme1NldH0pICJdLFszLDAsIlxcbWF0aHJte0xlZnR1bHRyYWZ1bmN0b3JzfShZLCBKXFx0ZXh0ey19XFxtYXRocm17Y29udGludW91cyBcXCBsZXh9KEMsXFxtYXRoc2Z7U2V0fSkgIl0sWzAsMywiSiBcXHRleHR7LX0gXFxtYXRocm17Y29udGludW91cyBcXCBsZXh9KEMsIFxcbWF0aHJte0xlZnR1bHRyYWZ1bmN0b3JzfShYLFxcbWF0aHNme1NldH0pKSAiXSxbMywzLCJKIFxcdGV4dHstfSBcXG1hdGhybXtjb250aW51b3VzIFxcIGxleH0oQywgXFxtYXRocm17TGVmdHVsdHJhZnVuY3RvcnN9KFksXFxtYXRoc2Z7U2V0fSkpICJdLFswLDYsIkogXFx0ZXh0ey19IFxcbWF0aHJte2NvbnRpbnVvdXMgXFwgbGV4fShDLCBTaChYKSkgIl0sWzMsNiwiSiBcXHRleHR7LX0gXFxtYXRocm17Y29udGludW91cyBcXCBsZXh9KEMsIFNoKFkpKSAiXSxbMCw5LCJcXG1hdGhybXtHZW9tfShYLFNoKEMpKSJdLFszLDksIlxcbWF0aHJte0dlb219KFksU2goQykpIl0sWzAsMSwiLSBcXGNpcmMgZiJdLFswLDIsIlxcc2ltZXEiLDJdLFsyLDMsIiggLSBcXGNpcmMgZilfe2MgXFxpbiBDfSIsMl0sWzEsMywiXFxzaW1lcSJdLFsyLDQsIlxcc2ltZXEiXSxbMyw1LCJcXHNpbWVxICJdLFs0LDUsImZeKiBcXGNpcmMgLSIsMV0sWzQsNiwiXFxzaW1lcSIsMV0sWzUsNywiXFxzaW1lcSIsMV0sWzYsNywiLVxcY2lyYyBmIiwxXV0=
\[\begin{tikzcd}
	{\mathrm{Lult}(X, J\text{-}\mathrm{continuous \ lex}(C,\mathsf{Set})) } &&& {\mathrm{Lult}(Y, J\text{-}\mathrm{continuous \ lex}(C,\mathsf{Set})) } \\
	\\
	\\
	{J \text{-} \mathrm{continuous \ lex}(C, \mathrm{Lult}(X,\mathsf{Set})) } &&& {J \text{-} \mathrm{continuous \ lex}(C, \mathrm{Lult}(Y,\mathsf{Set})) } \\
	\\
	\\
	{J \text{-} \mathrm{continuous \ lex}(C, Sh(X)) } &&& {J \text{-} \mathrm{continuous \ lex}(C, Sh(Y)) } \\
	\\
	\\
	{\mathrm{Geom}(X,Sh(C))} &&& {\mathrm{Geom}(Y,Sh(C))}
	\arrow["{- \circ f}", from=1-1, to=1-4]
	\arrow["\simeq"', from=1-1, to=4-1]
	\arrow["\simeq", from=1-4, to=4-4]
	\arrow["{( - \circ f)_{c \in C}}"', from=4-1, to=4-4]
	\arrow["\simeq", from=4-1, to=7-1]
	\arrow["{\simeq }", from=4-4, to=7-4]
	\arrow["{f^* \circ -}"{description}, from=7-1, to=7-4]
	\arrow["\simeq"{description}, from=7-1, to=10-1]
	\arrow["\simeq"{description}, from=7-4, to=10-4]
	\arrow["{-\circ f}"{description}, from=10-1, to=10-4]
\end{tikzcd}\]

\noindent Here, by abuse of language, we are using the same name $f$ for the continuous map $f$ from $Y$ to $X$, and the corresponding geometric morphism from $Sh(Y)$ to $Sh(X)$.

Using a similar diagram as the above we show also that if we have $f \leq g$ for two continuous functions $f$  and $g$, from $Y$ to $X$, then we have the commutativity of the following diagram:

\[\begin{tikzcd}
	{\mathrm{Lult}(X, J\text{-}\mathrm{continuous \ lex}(C,\mathsf{Set}))} &&& {\mathrm{Lult}(Y, J\text{-}\mathrm{continuous \ lex}(C,\mathsf{Set})) } \\
	\\
	\\
	{\mathrm{Geom}(X,Sh(C))} &&& {\mathrm{Geom}(Y,Sh(C))}
	\arrow[""{name=0, anchor=center, inner sep=0}, "{- \circ f}", curve={height=-18pt}, from=1-1, to=1-4]
	\arrow[""{name=1, anchor=center, inner sep=0}, "{- \ \circ g}"', shift right=4, curve={height=12pt}, from=1-1, to=1-4]
	\arrow["\simeq"', from=1-1, to=4-1]
	\arrow["\simeq", from=1-4, to=4-4]
	\arrow[""{name=2, anchor=center, inner sep=0}, "{-\circ f}", curve={height=-18pt}, from=4-1, to=4-4]
	\arrow[""{name=3, anchor=center, inner sep=0}, "{- \ \circ g}"{description}, shift right=4, curve={height=18pt}, from=4-1, to=4-4]
	\arrow[Rightarrow, from=0, to=1]
	\arrow[Rightarrow, from=2, to=3]
\end{tikzcd}\]

\subsection{Equivalence between $\mathrm{ClovenCart}(\cat{Top}//E, \cat{Top}//{E^{'}})$ and $\mathrm{Geom}(E,E^{'})$}
\label{third equivalence}
The main theorem we want to show in this subsection is:

\begin{theorem}
    Let $E$ and $E^{'}$ be two toposes with enough points, then there is an equivalence of categories between $\la{Geom}(E,E^{'})$ and $\la{ClovenCart}(\cat{Top}//E, \cat{Top}//E^{'})$.\label{third equivalence theorem}
\end{theorem}

But before this, we will be showing the following lemma:

\begin{lemma}
    
Let $E$ be a topos with enough points, take the forgetful functor $\ca{F}$ from the category $\cat{Top}/E$ (the pseudo-slice) to the category of Grothendieck toposes with enough points, then $\varinjlim\ca{F} =E$.
\end{lemma}

We know from \cite{butz1997representation} that there exists a simplicial topos (or more simply a topological groupoid) such that $E$ is the colimit of this functor (a functor from $\Delta^{op}$ or from $\mathbf{2}^{op}$  to the category of toposes with enough points).

\paragraph{Construction of the category $\ca{P}$}
Next thing we are going to do is to construct a nice category, with a functor to the category of topological groupoids. Let us choose an object $A$ such that the topos $E$ is generated by subobjects of powers of $A$. The choice of the object $A$ will be the same for all the groupoids in the category $\ca{P}$. Logically speaking, the choice of $A$ should be regarded as the choice of a syntax for the theory. Now, the objects of the category $\ca{P}$ are pairs $(P,I)$ where $P$ is a jointly epimorphic family of points and $I$ is an infinite cardinal such that $I \geq |A_p| $ for every $p \in P$. Now for morphisms we want $\ca{P}$ to be a large poset, and we do this by requiring the existence of a unique morphism from $(P,I)$ to $(P^{'},I^{'})$ iff $I \leq I^{'}$ and $P \subseteq P^{'}$.

Suppose that we have a cardinal $I$ and a set $B$, we define an enumeration by $I$ of elements of $B$ to be a partial function $e$ from $I$ to $B$, which is a surjection, such that for every $b \in B$, $e^{-1}(b)$ is infinite.

 \paragraph{Construction of the functor $G$}
Now we want to define a functor $G$ from $\ca{P}$ to the category of topological groupoids as follows: For every pair $(P,I)$, we define a groupoid $G(P,I)$: the object space of the groupoid $G(P,I)$ consists of pairs $(p,e)$ where $p \in P$  equipped with an enumeration $e$ of $A_p$ by $I$, subject to the equivalence relation that identifies  $(p_1,e_1)$  and $(p_2,e_2)$ if there exists an isomorphism $\omega$ between $p_1$ and $p_2$ such that $\omega_A \circ e_{p_1} =e_{p_2}$. The morphism space consists  of any natural isomorphism between points (not necessarily enumeration preserving). Now the topology of the object space of $G(P,I)$ is defined by basic open sets $U_{i_1,\dots,i_n,C}= \{(p,e) \in G_0: \  (e(i_1), \dots,e(i_n)) \in C_p   \}$, here $C \subseteq A^n$. The morphism space of $G(P,I)$ has basic open sets $V_{\vec{i},B,\vec{j},C}= \{(p_1,e_1,p_2,e_2,\omega) \in G_1: e_1(\vec{i}) \in B_{p_1},e_2(\vec{j}) \in C_{p_2}, \ \text{and} \ \omega(e_1(\vec{i}))= e_2(\vec{j}) \}$ Here $\vec{i}=(i_1,\dots,i_n)$ and $\vec{j}=(j_1,\dots,j_n)$ and $B,C \subseteq A^n$, and $e(\vec{i})$ is short for $(e(i_1),\dots,e(i_n))$. We will be calling such a topological groupoid a \textit{logical topological groupoid} since it's built out of models (points of the topos $E$). See \cite{butz1997representation, butz1999topological} for more details on this construction.

Now for the action of $G$ on morphisms, if there is a morphism from $(P,I)$ to $(P^{'},I^{'})$ (i.e. $P \subseteq P^{'}$ and $I \leq I^{'}$) then the image by $G$ of this  morphism is the following map of topological groupoids: we define a map from $I'$ to $I$ by sending every element  of $I$ to itself and every element in $I^{'}$ which is not in $I$ to the empty set, this allows us to send an enumeration of $A_p$ by $I$ to an enumeration by $I^{'}$. On morphisms this is just sending an isomorphism between $p$ and $p^{'}$ to itself (of course with different enumerations). It is easy to show that such maps are well defined (do not depend on the equivalence class of elements of $G_0$ and $G_1$). This construction is due to a remark in Moerdijk-Butz \cite{butz1999topological}.

To show that this construction is  a functor, notice that if we have cardinals $I\geq I^{'} \geq I^{''}$, and let us call $r_1$ the map from $I$ to $I^{'}$ that sends $x$ to itself if $x \in I^{'}$ and otherwise send it to the empty set and $r_2$ the similar map from $I^{'}$ to $I^{''}$, then their composition is exactly the map from $I$ to $I^{''}$ that sends $x$ to itself if $x \in I^{''}$ and sends it to the empty set otherwise (recall that the empty set is an element of every non-empty cardinal number).

It was shown in Moerdijk and Butz \cite{butz1997representation}, that the colimit of any such groupoid is the topos $E$, this isomorphism realises elements of $E$ as sheaves (étale spaces) over $G_0$, with a continuous action of $G_1$ (see \cite{moerdijk1988classifying, butz1999topological, butz1997representation}).

\begin{lemma}
The category $\ca{P}$  is filtered.   
\end{lemma}

\begin{proof}
   Suppose that $(P_1,I_1)$ and $(P_2,I_2)$ are two elements of the category, then take the element $(P_1 \cup P_2,\la{max}(I_1,I_2))$. Then clearly there are morphisms from each of these elements, to this new element. Now by the fact that $\ca{P}$ is a poset, this is enough to show that $\ca{P}$ is filtered.

\end{proof}

\begin{lemma} Suppose that we have two topological groupoids $G(P,I)$ and $G(P^{'},I')$, and a morphism $f$ between them which is in the image of the category $\ca{P}$ by $G$, then the following diagram commutes up to natural isomorphism:

\[
% https://q.uiver.app/#q=WzAsNSxbMCwyLCJHXzAiXSxbMCwwLCJHXzEiXSxbMiw0LCJFIl0sWzQsMiwiR157J31fMCJdLFs0LDAsIkdeeyd9XzEiXSxbMSwwLCIiLDAseyJsYWJlbF9wb3NpdGlvbiI6NzAsIm9mZnNldCI6M31dLFsxLDAsIiIsMix7Im9mZnNldCI6LTN9XSxbMCwyLCJwIiwxXSxbMCwzLCJmXzAiLDFdLFs0LDMsIiIsMix7Im9mZnNldCI6LTN9XSxbNCwzLCIiLDAseyJvZmZzZXQiOjN9XSxbMSw0LCJmXzEiLDFdLFszLDIsInBeeyd9IiwxXSxbNywxMiwiXFxzaW1lcSIsMSx7InNob3J0ZW4iOnsic291cmNlIjoyMCwidGFyZ2V0IjoyMH19XV0=
\begin{tikzcd}
	{G_1} &&&& {G^{'}_1} \\
	\\
	{G_0} &&&& {G^{'}_0} \\
	\\
	&& E
	\arrow["{f_1}"{description}, from=1-1, to=1-5]
	\arrow[shift right=3, from=1-1, to=3-1]
	\arrow[shift left=3, from=1-1, to=3-1]
	\arrow[shift left=3, from=1-5, to=3-5]
	\arrow[shift right=3, from=1-5, to=3-5]
	\arrow["{f_0}"{description}, from=3-1, to=3-5]
	\arrow[""{name=0, anchor=center, inner sep=0}, "p"{description}, from=3-1, to=5-3]
	\arrow[""{name=1, anchor=center, inner sep=0}, "{p^{'}}"{description}, from=3-5, to=5-3]
	\arrow["\simeq"{description}, shorten <=13pt, shorten >=13pt, Rightarrow, from=0, to=1]
\end{tikzcd}
\]

\end{lemma}

\begin{proof}

To see this, let $B$ be a sheaf in $E$, and let $p$ and $p^{'}$ be the respective geometric morphisms from  $G(P,I)$ and $G^{'}(I^{'},P^{'})$, then $B$ can be regarded as a sheaf on $G_0$ (image by $p^*$) with an action of $G_1$, or as a sheaf on $G^{'}_0$ (image by ${p^{'}}^*$) with an action of $G^{'}_1$, now let us regard the following diagram: 

\[\begin{tikzcd}
	{{p^{*}(B)} } && {{{p'}^{*}(B)}} \\
	\\
	{G_0} && {G'_0}
	\arrow["{\pi_1}"', from=1-1, to=3-1]
	\arrow["{\pi_2}", from=1-3, to=3-3]
	\arrow["{f_0}"', from=3-1, to=3-3]
\end{tikzcd}\]
Our goal is to show that the topology of the étale space $p^*(B)$ is the pullback of that of ${p'}^*(B)$, first it's easy to check that as sets, the set  of $p^*(B)$  is the set of the pullback (up to iso). It remains to check that it has the pullback topology, but one can easily notice that the basic open sets in both coincide.

\end{proof}

\paragraph{Digression: $2$-coend calculus $2$-left Kan extension along $2$-Yoneda}

Suppose that we have  bicategories $\ca{A}$ and $\ca{B}$, furthermore suppose that we have a pseudo-functor $F$  $\ca{A}^{op} \times \ca{A} \rightarrow \ca{B}$, we define the pseudocoend of $F$ to be the colimit (in the bicategorical sense) of the functor $F$ weighted by the weight  $\la{Hom}_A$ bifunctor: we write  $$\int^AF(A,A)=\varinjlim^{\la{Hom}}F$$

Unpacking this means that the pseudo coend satisfies the following property: For every object $X$ there is an equivalence of categories:

$$\ca{B}(\int ^A F(A,A),X) \simeq \mathrm{pseudo}\dash\la{Nat}(A^{op} \times A ,\cat{Cat})(\mathrm{Hom},\ca{B}(F(-,-),X))$$.

which is pseudonatural in $X$.

This gives exactly the definition in terms of pseudo cowedges. We are going to give the definition since it's the classical way to do such construction, although it plays no role in our discussion:

\begin{definition}
Let $S : \ca{A}^{op} \times \ca{A} \to \ca{B}$ be a pseudofunctor between bicategories $\ca{A}, \ca{B}$. A \emph{pseudo-cowedge} $\omega$ for $S$ consists of:
\begin{itemize}
    \item a triple $(B, (\omega_{A})_{A \in \ca{A}}, (\omega_f)_{f\in \la{Hom}(A,A'),A,A'\in \ca{A}})$, where $B \in \ca{B}$ is an object (called the \emph{tip} of the cowedge);
    \item collections of 1-cells $\omega_{A} : S(A,A) \to B$, for each $A \in \ca{A}$;
    \item invertible 2-cells $\omega_f : \omega_A \circ S(f,\la{id}_A) \Rightarrow \omega_{A'} \circ S(\la{id}_{A'},f)$, diagrammatically:

 % https://q.uiver.app/#q=WzAsNCxbMCwwLCJTKEEnLEEpIl0sWzAsMiwiUyhBJyxBJykiXSxbMywyLCJCIl0sWzMsMCwiUyhBLEEpIl0sWzAsMSwiUyhcXGxhe2lkfV97QSd9LGYpIl0sWzEsMiwiXFxvbWVnYV97QSd9Il0sWzAsMywiUyhmLFxcbGF7aWR9X3tBfSkiLDJdLFszLDIsIlxcb21lZ2Ffe0F9IiwyXSxbMywxLCJcXG9tZWdhX2YiLDAseyJsZXZlbCI6Mn1dXQ==
\[\begin{tikzcd}
	{S(A',A)} &&& {S(A,A)} \\
	\\
	{S(A',A')} &&& B
	\arrow["{S(f,\la{id}_{A})}"', from=1-1, to=1-4]
	\arrow["{S(\la{id}_{A'},f)}", from=1-1, to=3-1]
	\arrow["{\omega_f}", Rightarrow, from=1-4, to=3-1]
	\arrow["{\omega_{A}}"', from=1-4, to=3-4]
	\arrow["{\omega_{A'}}", from=3-1, to=3-4]
\end{tikzcd}\]
    
\end{itemize}
These data must satisfy the following coherence axioms:
\begin{enumerate}[label=\textup{cw\arabic*)}]
    \item For any 2-cell $\lambda : f \Rightarrow f'$ in $\ca{A}$, we have
    $$\omega_{f'} * (\omega_A \circ S(\lambda,\la{id}_A)) = (\omega_{A'} \circ S(\la{id}_{A'},\lambda)) * \omega_f$$ Diagrammatically:
\[
\begin{tikzcd}
	{S(A',A)} &&& {S(A,A)} &&& {S(A',A)} &&& {S(A,A)} \\
	&&&& {=} \\
	{S(A',A')} &&& B &&& {S(A',A')} &&& B
	\arrow["{S(f,\la{id}_A)}"', from=1-1, to=1-4]
	\arrow[""{name=0, anchor=center, inner sep=0}, "{S(\la{id}_{A'},f)}", curve={height=-12pt}, from=1-1, to=3-1]
	\arrow[""{name=1, anchor=center, inner sep=0}, "{S(\la{id}_{A'},f')}"', shift right=4, curve={height=12pt}, from=1-1, to=3-1]
	\arrow["{\omega_f}", Rightarrow, nfold, from=1-4, to=3-1]
	\arrow["{\omega_{A}}"', from=1-4, to=3-4]
	\arrow[""{name=2, anchor=center, inner sep=0}, "{S(f',\la{id}_A)}"', curve={height=12pt}, from=1-7, to=1-10]
	\arrow[""{name=3, anchor=center, inner sep=0}, "{S(f,\la{id}_A)}"{description}, curve={height=-18pt}, from=1-7, to=1-10]
	\arrow["{S(\la{id}_{A'},f')}", from=1-7, to=3-7]
	\arrow["{\omega_{f'}}", Rightarrow, nfold, from=1-10, to=3-7]
	\arrow["{\omega_{A}}"', from=1-10, to=3-10]
	\arrow["{\omega_{A'}}", from=3-1, to=3-4]
	\arrow["{\omega_{A'}}", from=3-7, to=3-10]
	\arrow["{S(\la{id}_{A'},\lambda)}", Rightarrow, nfold, from=0, to=1]
	\arrow["{S(\lambda,\la{id}_A)}"{description}, Rightarrow, nfold, from=3, to=2]
\end{tikzcd}\]
    
    \item For each pair $A \xrightarrow{f} A' \xrightarrow{f'} A''$ of composable 1-cells in $\ca{A}$, the following diagram is commutative
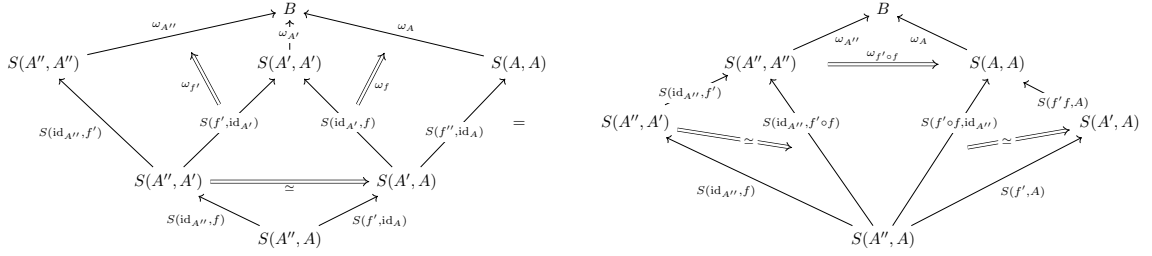
\[
\adjustbox{max width =\textwidth}{
\begin{tikzcd}
	&& B &&&&& B && \\
	{S(A'',A'')} && {S(A',A')} && {S(A,A)} && {S(A'',A'')} && {S(A,A)} \\
	&&&& {=} & {S(A'',A')} &&&& {S(A',A)} \\
	& {S(A'',A')} && {S(A',A)} \\
	&& {S(A'',A)} &&&&& {S(A'',A)}
	\arrow[""{name=0, anchor=center, inner sep=0}, "{{\omega_{A''}}}", from=2-1, to=1-3]
	\arrow["{{\omega_{A'}}}"{description}, from=2-3, to=1-3]
	\arrow[""{name=1, anchor=center, inner sep=0}, "{{\omega_A}}"', from=2-5, to=1-3]
	\arrow["{{\omega_{A''}}}"', from=2-7, to=1-8]
	\arrow["{{\omega_{f'\circ f}}}", Rightarrow, from=2-7, to=2-9]
	\arrow["{{\omega_A}}", from=2-9, to=1-8]
	\arrow["{{S(\la{id}_{A''},f')}}"{description, pos=0.4}, shift left, from=3-6, to=2-7]
	\arrow["{{S(f,\la{id}_A)}}"{description, pos=0.3}, shift left, from=3-10, to=2-9]
	\arrow["{{S(\la{id}_{A''},f')}}", from=4-2, to=2-1]
	\arrow[""{name=2, anchor=center, inner sep=0}, "{{S(f',\la{id}_{A'})}}"{description}, from=4-2, to=2-3]
	\arrow["\simeq"', Rightarrow, from=4-2, to=4-4]
	\arrow[""{name=3, anchor=center, inner sep=0}, "{{S(\la{id}_{A'},f)}}"{description}, from=4-4, to=2-3]
	\arrow["{{S(f,\la{id}_{A})}}"{description, pos=0.4}, from=4-4, to=2-5]
	\arrow["{{S(\la{id}_{A''},f)}}", from=5-3, to=4-2]
	\arrow["{{S(f',\la{id}_A)}}"', from=5-3, to=4-4]
	\arrow[""{name=4, anchor=center, inner sep=0}, "{S(\la{id}_{A''},f'\circ f)}"{description, pos=0.7}, from=5-8, to=2-7]
	\arrow[""{name=5, anchor=center, inner sep=0}, "{S(f'\circ f,\la{id}_{A})}"{description, pos=0.7}, from=5-8, to=2-9]
	\arrow["{{S(\la{id}_{A''},f)}}", from=5-8, to=3-6]
	\arrow["{{S(f',\la{id}_A)}}"', from=5-8, to=3-10]
	\arrow["\simeq"{description}, Rightarrow, from=3-6, to=4]
	\arrow["{{\omega_{f'}}}", Rightarrow, from=2, to=0]
	\arrow["{{\omega_f}}"', Rightarrow, from=3, to=1]
	\arrow["\simeq"{description}, Rightarrow, from=5, to=3-10]
\end{tikzcd}}
\]
    
Here the equivalences displayed are the natural equivalences resulting from the pseudofunctor structure of $S$.
    
    \item For each object $A \in \ca{A}$, the following diagram commutes:
    
    $\omega_{\mathrm{id}_A} = \mathrm{id}_{\omega_A}$.
\end{enumerate}
\end{definition}

\begin{definition}[Modification between pseudo cowedges]
A \emph{modification} $\Theta : \omega \Rrightarrow \sigma$ between two pseudo cowedges $\omega, \sigma$ for $S$ with tip $B$ consists of a collection of 2-cells $\Theta_A : \omega_A \Rightarrow \sigma_A$ indexed by objects $A \in \ca{A}$, such that for every 1-cell $f : A \to A'$ we have
$$(\Theta_{A'} \circ S(f,\la{id}_{A'}))* \omega_f = \sigma_f * (\Theta_A \circ S(\la{id}_A,f))$$

Here $\circ$ and $*$ denote the horizontal and vertical composition of $2$-cells respectively.

Diagrammatically:
% https://q.uiver.app/#q=WzAsOSxbMCwwLCJTKEEnLEEpIl0sWzMsMCwiUyhBJyxBJykiXSxbNiwwLCJTKEEnLEEpIl0sWzksMCwiUyhBJyxBJykiXSxbNSwxLCI9Il0sWzAsMiwiUyhBLEEpIl0sWzMsMiwiQiJdLFs2LDIsIlMoQSxBKSJdLFs5LDIsIkIiXSxbMCwxLCJTKFxcbGF7aWR9X3tBJ30sZikiLDFdLFswLDUsIlMoZixcXGxhe2lkfV97QX0pIiwxXSxbMSw2LCJcXG9tZWdhX3tBJ30iLDIseyJjdXJ2ZSI6M31dLFsxLDYsIlxcc2lnbWFfe0EnfSIsMCx7Im9mZnNldCI6LTMsImN1cnZlIjotM31dLFsyLDMsIlMoXFxsYXtpZH1fe0EnfSxmKSIsMV0sWzIsNywiUyhmLFxcbGF7aWR9X3tBfSkiLDFdLFszLDgsIlxcc2lnbWFfe0EnfSIsMV0sWzUsNiwiXFxvbWVnYV97QX0iLDFdLFs3LDMsIlxcc2lnbWFfZiIsMSx7ImxldmVsIjoyfV0sWzcsOCwiXFxzaWdtYV97QX0iLDEseyJjdXJ2ZSI6LTN9XSxbNyw4LCJcXG9tZWdhX3tBfSIsMSx7Im9mZnNldCI6LTMsImN1cnZlIjo0fV0sWzUsMSwiXFxvbWVnYV9mIiwxLHsibGV2ZWwiOjJ9XSxbMTEsMTIsIlxcVGhldGFfe0EnfSIsMl0sWzE5LDE4LCJcXFRoZXRhX0EiLDFdXQ==
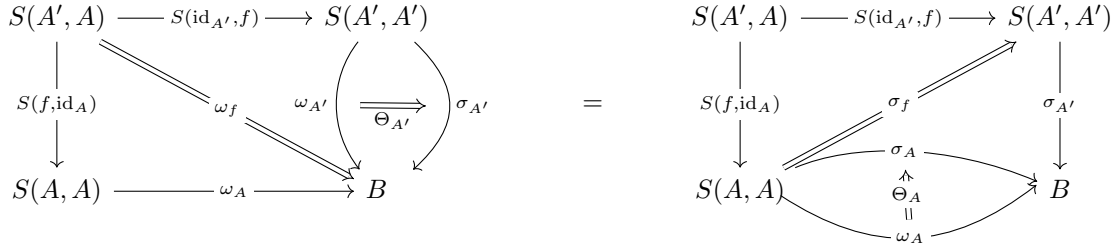
\[\begin{tikzcd}
	{S(A',A)} &&& {S(A',A')} &&& {S(A',A)} &&& {S(A',A')} \\
	&&&&& {=} \\
	{S(A,A)} &&& B &&& {S(A,A)} &&& B
	\arrow["{S(\la{id}_{A'},f)}"{description}, from=1-1, to=1-4]
	\arrow["{S(f,\la{id}_{A})}"{description}, from=1-1, to=3-1]
	\arrow[""{name=0, anchor=center, inner sep=0}, "{\omega_{A'}}"', curve={height=18pt}, from=1-4, to=3-4]
	\arrow[""{name=1, anchor=center, inner sep=0}, "{\sigma_{A'}}", shift left=3, curve={height=-18pt}, from=1-4, to=3-4]
	\arrow["{S(\la{id}_{A'},f)}"{description}, from=1-7, to=1-10]
	\arrow["{S(f,\la{id}_{A})}"{description}, from=1-7, to=3-7]
	\arrow["{\sigma_{A'}}"{description}, from=1-10, to=3-10]
	\arrow["{\omega_f}"{description}, Rightarrow, nfold, from=3-1, to=1-4]
	\arrow["{\omega_{A}}"{description}, from=3-1, to=3-4]
	\arrow["{\sigma_f}"{description}, Rightarrow, nfold, from=3-7, to=1-10]
	\arrow[""{name=2, anchor=center, inner sep=0}, "{\sigma_{A}}"{description}, curve={height=-18pt}, from=3-7, to=3-10]
	\arrow[""{name=3, anchor=center, inner sep=0}, "{\omega_{A}}"{description}, shift left=3, curve={height=24pt}, from=3-7, to=3-10]
	\arrow["{\Theta_{A'}}"', Rightarrow, nfold, from=0, to=1]
	\arrow["{\Theta_A}"{description}, Rightarrow, nfold, from=3, to=2]
\end{tikzcd}\]
\end{definition}

The collection of pseudo-cowedges with tip $B$ and modifications between them forms a category $\cat{CW}^{\mathsf{ps}}(B,S)$, more explicitly:
\begin{definition}[Category of pseudo-cowedges]
Let $S : \ca{A}^{op} \times \ca{A} \to \ca{B}$ be a pseudofunctor and let $B \in \ca{B}$ be an object. The \emph{category of pseudo-cowedges with tip $B$}, denoted $\mathrm{CW}^{\mathsf{ps}}(B, S)$, is defined as follows:
\begin{itemize}
    \item Objects are pseudo-cowedges to $S$ with tip $B$.
    \item Morphisms are modifications $\Theta : \omega \Rrightarrow \sigma$ between pseudo-cowedges with the same tip $B$.
    \item Composition  is defined as composition of the corresponding 2-cells: $(\Theta' \circ \Theta)_A := \Theta'_A * \Theta_A$.
\end{itemize}

\end{definition}

The reader can easily verify that the category of pseudo cowedges to $S$ with tip $B$, $\cat{CW}^{\mathsf{ps}}(B,S)$, is equivalent to the category $\la{Pseudo\dash Nat}(\mathrm{Hom_A},\ca{B}(S(-,-),B))$, so the pseudo coend exists iff the pseudo-functor $B \mapsto \cat{CW}^{\mathsf{ps}}(B,S)$ is representable.

This is similar to the way  lax ends and coends were defined in Bozapalides\cite{bozapalides1976theorie}, Loregian \cite{loregian2021co}, and Hirata \cite{hirata2022notes}; however there are three main differences:

\begin{enumerate}
    \item We have accounted for the pseudofunctoriality of $S$ in the definition of pseudo-cowedges.

    \item The universal property no longer requires the uniqueness of the map coming out of the coend to any cowedge,  since for any object $X$ of $B$, we require an equivalence of categories pseudo-natural in $X$, between $\cat{CW}^{\mathsf{ps}}(X, S)$ and $\la{Hom}(\int^A S(C,C),X)$, while the definition in \cite{bozapalides1976theorie,loregian2021co,hirata2022notes} requires an isomorphism of categories , since  the authors interpreted the defining colimit (limit) of such coends (ends) in the strict $2$-categorical sense.

    \item Since we are dealing with pseudo-coends rather than lax coends, we require the $2$-morphism in the definition to be invertible (actually our $2$ cells  go in the opposite direction), of course by dropping invertibility we can get oplax coends, and by changing the direction we can get lax coends.
\end{enumerate}

Similarly, we define pseudo ends.

Let $\cat{Cat}$ denote the 3-category of categories, pseudo-functors, pseudo-natural transformations and modifications.

Our goal is to show the following assertions:
\begin{lemma*}
    $$\mathrm{pseudo \dash Nat}(F,G)\simeq \int_{\ca{A}} \ca{B}(F(C),G(C))$$ 
\end{lemma*}

\begin{proof}
    
 For any category $X$, we have the following $$\cat{Cat}(X,\int_{\ca{A}} \ca{B}(F(C),G(C))) \simeq \cat{Cat}(\mathrm{Hom},\cat{Cat}(X,(\ca{B}(F-,G-))))$$.

In particular this applies to the terminal category, so we get:

$$ \int_{\ca{A}}\ca{B}(F(C),G(C))\simeq\cat{Cat}(\mathrm{Hom},\ca{B}(F-,G-))\simeq \mathrm{pseudo \dash Nat}(F,G)  $$

\end{proof}

{\bf Note.} If we wanted a version of the theorem above that would work for lax ends then the argument should be the following: 

 $$\cat{Cat}(X,\sqoint_{\ca{A}} \ca{B}(F(C),G(C))) \simeq \cat{Cat^{Lax}}(\mathrm{Hom},\cat{Cat}(X,(\ca{B}(F-,G-))))$$
 
 Again applying this in the case where $X$ is the terminal category we get that:

  $$\sqoint_{\ca{A}} \ca{B}(F(C),G(C))) \simeq \cat{Cat^{Lax}}(\mathrm{Hom},(\ca{B}(F-,G-)))) \simeq \cat{Lax} (F,G)$$

  Same idea should work for oplax coends.

The second assertion that we want to prove is the following:

\begin{lemma*}[Fubini transformation for pseudo coends]
Suppose we have a pseudo functor $F:\ca{A}^{op}\times \ca{A} \times \ca{B}^{op} \times \ca{B} \rightarrow \ca{C}$ we have $\int^{\ca{A} \times \ca{B}}F(a,a,b,b) \simeq \int^{\ca{A}} \int^{\ca{B}} F(a,a,b,b) \simeq \int^{\ca{B}} \int^{\ca{A}} F(a,a,b,b)$. 

\end{lemma*}

\begin{proof}
    
Using classical facts about weighted colimits: We have $$\varinjlim^{\Hom A}\varinjlim^{\Hom_B}F \simeq \varinjlim^{\Hom_{A} \times \la{Hom}_B}F \simeq  \varinjlim^{\Hom_{A \times B}}F \simeq  \varinjlim^{\Hom B}\varinjlim^{\Hom_A}F$$
\end{proof}

In the same setting as the previous subsection, we want to construct the left Kan extension along the Yoneda embedding of the forgetful functor from  $\cat{Top}/E$ to $\cat{Topos}$, moreover we want to show that it preserves $2$-colimits. There is not a single version of left Kan extension in $2$-categorical settings, so we have many possible constructions, we are going to employ the notion of pseudo-Kan extension, the pseudo-Kan extension satisfies the following universal property:

Let $\ca{A},\ca{B},\ca{C}$ be bicategories, and suppose we have pseudofunctors $\ca{F}: \  \ca{A} \rightarrow \ca{B}$ and a $\ca{T}: \ca{A} \rightarrow \ca{C}$ as in the following figure:

% https://q.uiver.app/#q=WzAsMyxbMCwwLCJcXG1hdGhjYWx7QX0iXSxbMywwLCJcXG1hdGhjYWx7Qn0iXSxbMCwzLCJcXG1hdGhjYWx7Q30iXSxbMCwxLCJcXG1hdGhjYWx7Rn0iXSxbMCwyLCJcXG1hdGhjYWx7VH0iLDJdXQ==
\[\begin{tikzcd}[ampersand replacement=\&]
	{\mathcal{A}} \&\&\& {\mathcal{B}} \\
	\\
	\\
	{\mathcal{C}}
	\arrow["{\mathcal{F}}", from=1-1, to=1-4]
	\arrow["{\mathcal{T}}"', from=1-1, to=4-1]
\end{tikzcd}\]

Then a pseudofunctor $\mathrm{Lan}_T F$ from $\ca{C}$ to $\ca{B}$ is called the left Kan extension of $\ca{F}$ along $\ca{T}$ if for every $S: \ \ca{C} \rightarrow \ca{B}$, there is an equivalence  of categories between $\mathrm{Pseudo\text{-}Nat}(F, S \circ T)$ and $\mathrm{Pseudo \text{-}Nat}(\mathrm{Lan}_T F,S)$, given by composing with a pseudo-natural transformation $\alpha$ from $F$ to $\mathrm{Lan}_T F \circ T$, moreover this equivalence is natural in $S$ (so it's really an equivalence of $\cat{Cat}$ valued presheaves).

\begin{comment}
    \begin{note*} \normalfont
    It is also possible to work with lax Kan extensions as defined in  \cite{bozapalides1976theorie},\cite{loregian2021co} (warning: Loregian only defines Lax Kan extension for strict $2$-functors between strict $2$-categories. In general the lax and pseudo-Kan extension do not agree, Although in our case the lax and pseudo-Kan extension agree since both $\cat{Top}/E$ and $[\cat{Top}/E,\cat{Groupoids}]$ are $2$-$1$ categories.

\end{note*}
\end{comment}

Now we highlight the construction of pseudo-left Kan extension in the case where $\ca{B}$ is tensored over $\mathsf{Cat}$. In this case, we claim that $\mathrm{Lan}_T F= \int^{\ca{A}}   F(A) \otimes \ca{C}(TA,-)$, here $\int$ denotes the pseudo-coend.
\begin{proof}

Let $S$ be a pseudo-functor from $\ca{C}$ to $\ca{B}$, then we have the following sequence of equivalences
    \[ \mathrm{Pseudo\text{-}Nat}(\int^A   F(A) \otimes \ca{C}(TA,-),S-)  \]
    \[ \simeq \int_{\ca{C}} \ca{B}(\int^{\ca{A}} FA\otimes \ca{C}(TA,C),SC)
    \]
    \[ \simeq \int_{\ca{C}} \int_{\ca{A}} \ca{B}( FA \otimes \ca{C}(TA,C),SC)\]

    \[ \simeq   \int_{\ca{C}}   \int_{\ca{A}} \cat{Cat}(\ca{C}(TA,C) , \ca{B}(FA ,SC)) \]

   \[ \simeq  \int_{\ca{A}}   \int_{\ca{C}} \cat{Cat}(\ca{C}(TA,C) , \ca{B}(FA ,SC))\]

\[\simeq \int_{\ca{A}} \la{Pseudo \text{-} Nat}(\ca{C}(TA,-) , \ca{B}(FA ,S-)) \]

    \[\simeq \int_{\ca{A}}  \ca{B}(FA ,STA) \  \textbf{by\ } 2\textbf{-Yoneda for \ } \ca{C}\]

\[ \simeq \mathrm{Pseudo\text{-}Nat}(F, S \circ T)\]   
\end{proof}

Now we turn back to our case. We define $\bar{\ca{F}}= \int^{ c \in \cat{Top}/E} \Hom_{[\cat{Top}/E, \cat{Groupoids}]}(y(c),-) \otimes \ca{F}(-)$. We remind the reader that the category of Grothendieck toposes is $2$-tensored and $2$-cotensored over $2$-$\cat{Cat}$ see \cite{Pitts1985} (in order to avoid size issues we have to work with a category of ``Large Grothendieck toposes'' i.e. relax the size constraint to second universe small sets in Giraud axioms, then we can notice that our construction lands inside our classic category of Grothendieck toposes, the argument for this is that every presheaf of groupoids is a $\ca{U}_1$-small $2$-colimit of representables and $\ca{F}$ preserves such colimits, a fact we are going to show).

We first want to show that this diagram:

% https://q.uiver.app/#q=WzAsMyxbMCwwLCJcXGNhdHtUb3B9L0UiXSxbMCwyLCJbXFxjYXR7VG9wfS9FLCBcXGNhdHtHcm91cG9pZHN9XSJdLFs0LDAsIlxcY2F0e1RvcG9zfSJdLFswLDEsInkiXSxbMCwyLCJcXGNhe0Z9IiwyXSxbMSwyLCJcXG1hdGhybXtMYW59X3koXFxjYXtGfSkiLDFdXQ==
\[\begin{tikzcd}
	{\cat{Top}/E} &&&& {\cat{Topos}} \\
	\\
	{[\cat{Top}/E, \cat{Groupoids}]}
	\arrow["{\ca{F}}"', from=1-1, to=1-5]
	\arrow["y", from=1-1, to=3-1]
	\arrow["{\mathrm{Lan}_y(\ca{F})}"{description}, from=3-1, to=1-5]
\end{tikzcd}\]

commutes up to isomorphism, we do this also by pseudo-coend calculus we have:  \[\bar{\ca{F}} \circ y  \simeq \int^{\cat{Top}/E}[\cat{Top}/E, \cat{Groupoids}](y(c),y-) \times \ca{F}(c)\]
\[  \simeq \int^{\cat{Top}/E} \mathsf{Top}/E(c,-) \times \ca{F}(c) \]\[\simeq \ca{F}\]

The next thing we would like to check is that $\bar{\ca{F}}$ preserves colimits. Our proof is essentially the same one as in this MathOverflow answer \cite{112206} modified to fit our $2$-categorical setting: We have that $\bar{\ca{F}} \simeq \int^{\cat{Top}/E}\ca{F}(A)\otimes [\cat{Top}/E,\la{Groupoids}](yA,-)$. Now tensoring  is a left adjoint hence it preserves colimits, and $\int^A$ is a $2$-colimit, hence $\bar{\ca{F}}$ preserves all colimits preserved by $[\cat{Top}/E,\la{Groupoids}](yA,-)$ but $[\cat{Top}/E,\la{Groupoids}](yA,-)$ is the evaluation functor at $A$ and hence it preserves all colimits (colimits are computed pointwise). And hence $\bar{\ca{F}}$ preserves all $2$-colimits (we did not state this before, but the category of toposes is $2$-cocomplete see \cite{moerdijk1988classifying}).

Now let $\ca{G}_{\bullet}$ be a topological groupoid whose colimit is $E$. Let us have a look at the following diagram:
 
 % https://q.uiver.app/#q=WzAsNCxbMCwwLCJcXERlbHRhXntvcH0iXSxbMywwLCJcXGNhdHtUb3B9L0UiXSxbNiwwLCJcXGNhdHtUb3Bvc30iXSxbMywyLCJbXFxjYXR7VG9wfS9FLFxcY2F0e0dyb3Vwb2lkc31dIl0sWzAsMSwiXFxjYXtHfV97XFxidWxsZXR9Il0sWzEsMiwie1xcY2F7Rn19Il0sWzEsMywieSJdLFszLDIsIntcXHRpbGRle1xcY2F7Rn19fSIsMV0sWzUsNywiXFxzaW1lcSIsMCx7InNob3J0ZW4iOnsic291cmNlIjoyMCwidGFyZ2V0IjoyMH0sInN0eWxlIjp7ImJvZHkiOnsibmFtZSI6Im5vbmUifSwiaGVhZCI6eyJuYW1lIjoibm9uZSJ9fX1dXQ==
\[\begin{tikzcd}
	{\Delta^{op}} &&& {\cat{Top}/E} &&& {\cat{Topos}} \\
	\\
	&&& {[\cat{Top}/E,\cat{Groupoids}]}
	\arrow["{\ca{G}_{\bullet}}", from=1-1, to=1-4]
	\arrow[""{name=0, anchor=center, inner sep=0}, "{{\ca{F}}}", from=1-4, to=1-7]
	\arrow["y", from=1-4, to=3-4]
	\arrow[""{name=1, anchor=center, inner sep=0}, "{{\bar{\ca{F}}}}"{description}, from=3-4, to=1-7]
	\arrow["\simeq", draw=none, from=0, to=1]
\end{tikzcd}\]

Here $\bar{\ca{F}}$ is the $2$-left Kan extension of $\ca{F}$ along the $2$-Yoneda embedding, and the equivalence in the middle is a pseudo-natural isomorphism. We want to show that  $\varinjlim \ca{F} =E$. We have that $\varinjlim \bar{\ca{F}} \circ y \circ G_{\bullet} \simeq \varinjlim (\ca{F} \circ \ca{G}_{\bullet}) = E$, but on the other hand we have that $\varinjlim \bar{\ca{F}} \circ y \circ G_{\bullet}$ is $\bar{\ca{F}}(\varinjlim (y \circ \ca{G}_{\bullet}))$. Let us denote by $\hat{G}$ the colimit of $(y \circ \ca{G}_{\bullet})$, this can be regarded as the ``representable'' at $\ca{G}_{\bullet}$, i.e. this is defined by sending a topological space $X \xrightarrow[f]{} E$, to the groupoid whose objects set is $\la{Hom}_{\cat{Top}/E}(X \xrightarrow{f} E, G_0  \xrightarrow{h} E)$ and whose isomorphisms set is $\la{Hom}_{\cat{Top}/E}(X \xrightarrow{f} E, G_1  \xrightarrow{h \circ s}E) \simeq \la{Hom}_{\cat{Top}/E}(X \xrightarrow{f} E, G_1  \xrightarrow{h \circ t}E)$.

We get that $\varinjlim \bar{\ca{F}} \circ y \circ \ca{G}_{\bullet} \simeq \varinjlim(\ca{F} \circ \ca{G}_{\bullet}) \simeq E$, but we know that $\varinjlim \bar{\ca{F}} \circ y \circ \ca{G}_{\bullet} \simeq \bar{\ca{F}}(\varinjlim (y \circ \ca{G}_{\bullet}))$ (since $\bar{\ca{F}}$ is colimit preserving), this allows us to deduce that $\bar{\ca{F}}(\varinjlim (y \circ \ca{G}_{\bullet}))=E$.
 
Let $p$ be the functor from $\ca{P}$ to the category $[\cat{Top}/E,\cat{Groupoids}]$, which sends $(P,I)$ to the functor that sends a topological space over $E$ to the representable at $G(P,I)$, i.e. $p= \go{y} \circ G$. Here $\go{y}$ is defined by sending a topological groupoid over $E$ $\ca{G}_{\bullet} \rightarrow E$ to $X \rightarrow E \mapsto \la{Hom}_{\cat{Top \ Groupoids}/E}(X \rightarrow E,\ca{G}_{\bullet} \rightarrow E)$. Let  us have a look at the following diagram: 

% https://q.uiver.app/#q=WzAsNSxbNiwwLCJcXGNhdHtUb3B9L0UiXSxbNiwyLCJbXFxjYXR7VG9wfS9FLFxcbWF0aHJte0dyb3Vwb2lkc31dIl0sWzksMCwiXFxtYXRocm17VG9wb3N9Il0sWzMsMiwiXFxtYXRocm17R3JvdXBvaWRzfSJdLFswLDIsIlxcY2F7UH0iXSxbMCwxLCJ5Il0sWzAsMiwiXFxjYXtGfSJdLFsxLDIsIlxcYmFye1xcY2F7Rn19IiwxXSxbMywxLCJ5Il0sWzQsMywiRyJdLFs0LDEsInAiLDAseyJvZmZzZXQiOi0yLCJjdXJ2ZSI6LTN9XSxbNiw3LCJcXHNpbWVxIiwwLHsic2hvcnRlbiI6eyJzb3VyY2UiOjIwLCJ0YXJnZXQiOjIwfSwic3R5bGUiOnsiYm9keSI6eyJuYW1lIjoibm9uZSJ9LCJoZWFkIjp7Im5hbWUiOiJub25lIn19fV1d
\[\begin{tikzcd}
	&&&&&& {\cat{Top}/E} &&& {\cat{Topos}} \\
	\\
	{\ca{P}} &&& {\mathsf{Top \ Groupoids}/E} &&& {[\cat{Top}/E,\mathsf{Groupoids}]}
	\arrow[""{name=0, anchor=center, inner sep=0}, "{\ca{F}}", from=1-7, to=1-10]
	\arrow["y", from=1-7, to=3-7]
	\arrow["G", from=3-1, to=3-4]
	\arrow["p", shift left=2, curve={height=-18pt}, from=3-1, to=3-7]
	\arrow["\go{y}", from=3-4, to=3-7]
	\arrow[""{name=1, anchor=center, inner sep=0}, "{\bar{\ca{F}}}"{description}, from=3-7, to=1-10]
	\arrow["\simeq", draw=none, from=0, to=1]
\end{tikzcd}\]

The above discussion allows us to deduce that the precomposition $\bar{\ca{F}} \circ p$ is equivalent to the constant functor at $E$, hence we may deduce that the colimit of $\bar{\ca{F}} \circ p$  is $E$, since the category $\ca{P}$ is filtered ($\ca{P}$ is a $1$-category, this is why this works). Now since $\bar{\ca{F}}$ is a colimit preserving functor, we may deduce that $\varinjlim \bar{\ca{F}} \circ p= \bar{\ca{F}}( \varinjlim p) =E$. So in order to show that $\varinjlim \ca{F}=E$, it is enough to show that $\varinjlim p = 1 $. Now since $p$ is a functor from $\ca{P}$ to $[\cat{Top}/E, \cat{Groupoids}]$, its $2$-colimit can be computed pointwise, this means that it's enough to show that for any object $X \rightarrow E$, we have that $\varinjlim_{(P,I)} \mathrm{Hom}_{\cat{Top \ Groupoids}/E} (X \rightarrow E , G(P,I) \rightarrow E) \simeq * $.

 \begin{lemma}
 Suppose that we have a functor $F : A \xrightarrow[]{} \cat{Groupoids}$, then $\varinjlim (F) = (\int_A F)^{Grp}$, where $-^{Grp}$ is the localisation of a category with respect to the class of all morphisms, and $\int$ is the Grothendieck construction of $F$ regarded as a functor to $\cat{Cat}$.
     \end{lemma}
 \begin{proof}
     Let us regard the following composition $A \rightarrow \cat{Groupoids} \rightarrow \cat{Cat}$, it is known that the oplax colimit of this diagram (\cite {gepner2017lax}, or see the nlab page on Grothendieck construction) is the Grothendieck construction of this functor, but on the other hand the pseudo-functor $-^{Grp}$ that takes a category to the localisation with respect to the class of all morphisms is the $2$-left adjoint to the inclusion of groupoids inside categories, hence $(\int F)^{Grp}$ is the oplax colimit of $F$, but since we are working with groupoids then the oplax colimit coincides with the  usual $2$-colimit (since in the $2$- category of groupoids all $2$-morphisms are invertible).
     
 \end{proof}

 \begin{note*} \normalfont
     In all the arguments, we are using the fact that $2$-left adjoints preserve $2$-colimits (and oplax colimits are instances of these), we give a sketch of the proof of this fact:

     Suppose that $\ca{J}: D \xrightarrow{} \cat{Cat}$ is a pseudofunctor (weight) and suppose that we have a functor $F$ from $D$ to $A$ and a functor $G$ from $A$ to $B$, such that this functor is a $2$-left adjoint with $2$-right adjoint $H$. And suppose that $L$ is the $2$-colimit of $F$. Then we have the following chain of pseudo-natural equivalence of $2$-categories $B(G(L),Y)  \simeq A(L,H(Y))\simeq [D,\cat{Cat}](\ca{J}, A(F-,H(Y)))  \simeq [D,\cat{Cat}](\ca{J}, B(G \circ F-,Y))$.

     A similar chain can show that a $2$-right adjoint preserves $2$-limits.
 \end{note*}

 \begin{lemma}
    Let $A$ be a filtered category, then $A^{Grp}=*$.
 \end{lemma}
 \begin{proof}
     It's easy to verify that for any filtered category, the conditions to construct the localisation of this category with respect to the class of all morphisms by a left calculus of fractions are met, and hence we may construct the category $A^{Grp}$ by a left calculus of  fractions (for a reference on calculus of fractions see \cite{gabriel35homotopy}).
     
     We wish to show that for any two objects $a,b \in A$, their images by the localisation functor satisfy $\la{Hom}^{'}(a,b)\simeq *$, by the fact that the category is filtered, we may deduce that inside $A^{Grp}$, $\la{Hom}^{'}(a,b)$ is non-empty. Now suppose that we have two cospans $a\xrightarrow{f} c \xleftarrow{g} b$ and $a\xrightarrow{f^{'}} c^{'} \xleftarrow{g^{'}} b$. Then by filteredness of $A$, we may find an object $c^{''}$ such that:

% https://q.uiver.app/#q=WzAsNSxbMCwxLCJhIl0sWzIsMCwiYyJdLFsyLDIsImNeeyd9Il0sWzQsMSwiYiJdLFsyLDEsImNeeycnfSJdLFswLDEsImYiXSxbMCwyLCJmXnsnfSIsMl0sWzMsMSwiZyIsMl0sWzMsMiwiZ157J30iXSxbMSw0LCJoXnsnJ30iXSxbMiw0LCJoXnsnfSIsMl1d
\[\begin{tikzcd}
	&& c \\
	a && {c^{''}} && b \\
	&& {c^{'}}
	\arrow["{h^{''}}", from=1-3, to=2-3]
	\arrow["f", from=2-1, to=1-3]
	\arrow["{f^{'}}"', from=2-1, to=3-3]
	\arrow["g"', from=2-5, to=1-3]
	\arrow["{g^{'}}", from=2-5, to=3-3]
	\arrow["{h^{'}}"', from=3-3, to=2-3]
\end{tikzcd}\]

Now using filteredness one more time, we may find an object $c^{'''}$ and a morphism $k$ such that $k \circ h^{''} \circ g= k \circ h^{'} \circ g^{'}$, and such that $k \circ h^{''} \circ f = k \circ h^{'} \circ f^{'}$.
 \end{proof}

So now suppose that $X \xrightarrow{f} E$ is a topological space over $E$, take the pseudo-functor $p_X$  from $\ca{P}$ to $\cat{Groupoids}$ defined by $(P,I) \mapsto \mathrm{Hom}_{\cat{Top \ Groupoids}/E} (X \rightarrow E , G(P,I))$. Our goal is to show that $\int p_X$ is filtered. So we need to check three conditions of filteredness of a category. Namely, a category $\ca{A}$ is filtered iff it satisfies the following conditions

\begin{enumerate}

    \item $A$ contains an object
    \item For any two objects $a,b \in \ca{A}$, there exists an object $c \in \ca{A}$ together with a pair of morphisms $a\xrightarrow{f}c$, and $b\xrightarrow{g}c$.
    \item For any parallel pair of morphisms $f,g: a\rightrightarrows b$, there exists a morphism $h: b \rightarrow c$, such that $h\circ g=h\circ f$.
\end{enumerate}

But before that let us show the following important result:

\begin{lemma}

\label{h-lemma}
Suppose that we have $\ca{G}_{\bullet}=G(P,I)$ a logical topological groupoid whose colimit is $E$, and suppose we have two morphisms $(h_1, \alpha_1)$ (here $\alpha_1$ is a natural isomorphism), and $(h_2, \alpha_2)$ from $X \xrightarrow{f} E$ to $G_0 \xrightarrow{g} E$ in the category $\cat{Top}/E$, then there exists a map of topological spaces $\tilde{h}$ from the topological space $X$ to the topological space $G_1$ such that $s \circ \tilde{h} =h_1$ and $t \circ \tilde{h} = h_2 $,  diagrammatically:

\[
% https://q.uiver.app/#q=WzAsNCxbMCwyLCJYIl0sWzAsNCwiRSJdLFsyLDIsIkdfMCJdLFsyLDAsIkdfMSJdLFswLDEsImYiLDJdLFswLDIsImhfMSJdLFszLDIsInMiLDIseyJsYWJlbF9wb3NpdGlvbiI6NDAsIm9mZnNldCI6Mn1dLFszLDIsInQiLDAseyJvZmZzZXQiOi0yfV0sWzIsM10sWzIsMSwiZyJdLFswLDMsIlxcZXhpc3RzXFx0aWxkZXtofSIsMl0sWzAsMiwiaF8yIiwyLHsib2Zmc2V0IjoyfV0sWzQsOSwiXFxzaW1lcSIsMCx7InNob3J0ZW4iOnsic291cmNlIjoyMCwidGFyZ2V0IjoyMH19XV0=
\begin{tikzcd}
	&& {G_1} \\
	\\
	X && {G_0} \\
	\\
	E
	\arrow["s"'{pos=0.4}, shift right=2, from=1-3, to=3-3]
	\arrow["t", shift left=2, from=1-3, to=3-3]
	\arrow["{\exists\tilde{h}}"', from=3-1, to=1-3]
	\arrow["{h_1}", from=3-1, to=3-3]
	\arrow["{h_2}"', shift right=2, from=3-1, to=3-3]
	\arrow[""{name=0, anchor=center, inner sep=0}, "f"', from=3-1, to=5-1]
	\arrow[from=3-3, to=1-3]
	\arrow[""{name=1, anchor=center, inner sep=0}, "g", from=3-3, to=5-1]
	\arrow["\simeq", shorten <=6pt, shorten >=6pt, Rightarrow, from=0, to=1]
\end{tikzcd}
\]
\end{lemma}

\begin{proof} For every point $x$ of $X$, we are going to denote, by abuse of language, by $f(x)$ the point of $E$, $x^* \circ f^*$, similarly by $g(h_1(x))$, the point $x^* \circ h_1^* \circ g^*$ of $E$, similarly we are going to denote by $\alpha_2(x)$ the isomorphism $x^* \circ \alpha_2 $.

Hence we get an isomorphism $\alpha_1(x)$ from $f(x)$ to $g (h_1(x))$, and an isomorphism $\alpha_2(x)$ between $f(x)$ and $g(h_2(x))$, this should allow us to consider the isomorphism $\alpha_2(x) \circ (\alpha_1(x)^{-1})$ between $g (h_1(x))$ and $g(h_2(x))$. This isomorphism (or to be rather precise its equivalence class) is an element of $G_1$, and hence we may define this map by sending $x$ to this element. Let us call this map
$\tilde{h}$.

\begin{comment}
    You should add a paragraph explaning the construction of logical groupoids
\end{comment}

Now we should show that such map is continuous, to do this let us take a basic open set of $G_1$, of the form $V_{\vec{i},B,\vec{j},C}$. Now notice that \[\tilde{h}^{-1}(V_{\vec{i},B,\vec{j},C})\] 

\[=\{ x \in X \   | \  e_{h_1(x)}(\vec{i}) \in B_{h_1(x)}, e_{h_2(x)}(\vec{j}) \in C_{h_2(x)},  \alpha_2(x)(\alpha_1(x)^{-1})(e_{h_1(x)}(\vec{i})) = e_{h_2(x)}(\vec{j})\} \] 

\[= \{ x \in X  \  |   \ e_{(h_1(x))}(\vec{i}) \in B_{h_1(x)}, e_{h_2(x)}(\vec{j}) \in C_{h_2(x)}, \alpha_1(x)^{-1} (e_{h_1(x)}(\vec{i})) = \alpha_2(x)^{-1}(e_{h_2(x)}(\vec{j}))\} \]

\[=\{x \in X \ | \ e_{h_1(x)}(\vec{i}) \in B_{h_1(x)} \} \bigcap \{x \in X \ | \ e_{h_2(x)}(\vec{j}) \in C_{h_2(x)} \} \bigcap \{x \in X \ | \ \alpha_1(x)^{-1} (e_{h_1(x)}(\vec{i})) = \alpha_2(x)^{-1}(e_{h_2(x)}(\vec{j})) \} \]  

The first two sets are open by continuity of $h_1$ and $h_2$ respectively. It remains to show that the set $\{x \in X \ | \ \alpha_1(x)^{-1} (e_{h_1(x)}(\vec{i})) = \alpha_2(x)^{-1}(e_{h_2(x)}(\vec{j})) \}$ is open. Notice that every $\vec{i^{'}} \in I^n$, defines a local section for the sheaf  $g^*(C)$ where $C \subseteq A^n$ is an element of the topos $E$ (reminder that an  element $M$ of $E$ can be regarded as a sheaf (local homeomorphisms) over $G_0$ ($g^*(C)$) equipped with an action of $G_1$).

Let us explain this idea. Suppose that $M \in E$, then the space $g^*(M)$ over $G_0$, is defined by $g^*(M)=\{ (p,e,a) \mid a \in M_p \} $ (more precisely we should be talking about equivalence classes of these, but that's clear), equipped with the topology with basic open sets  $V_{\vec{i}, C, f} = \{ (p, e, a) \mid (p,e) \in U_{\vec{i}, C} \text{ and } a = f(e(\vec{i})) \}$ \cite{butz1999topological}, this space has a natural projection map to $G_0$ making it étale over $G_0$. Moreover, this space has a natural action by $G_1$. Now suppose that $C \subseteq A^n$ and that $\vec{i} \in I^n$, then take the local section, which we also call by abuse of language $\vec{i}$ defined from $V_{\vec{i},C}= \{(p,e) \mid e(\vec{i}) \in C \}$ to $g^*(C)$ by $(p,e) \mapsto (p,e, e(\vec{i}))$, here $e(\vec{i})$ is short for $(e(i_1),\dots,e(i_n))$.

Now consider the sheaf  $g^*(A^n)$ on $G_0$,  and consider the local section $\vec{i}$ on $g^*(A^n)$, we can pull this section back by $h_1$ to get the local section for the sheaf $h_1^*((g^*)(A^n))$ defined on its domain  by $x \mapsto e_{h_1(x)}(\vec{i})$, also we may consider the local section $\vec{j}$ of $g^*(A^n)$, we can pull this section back by $h_2$ which hence gives us a local section for the sheaf $h_2^*(g^*(A^n))$.

Now we know that the images of local sections are open in their respective étale bundles, so the set $ \{ e_{h_1(x)}(\vec{i}) \}$ is open in $h_1^*(g^*)(A^n)$, similarly the set $ \{ e_{h_2(x)}(\vec{j}) \}$ is open in $h_2^*(g^*(A^n))$. This means that ${\alpha_1}_{A^{n}}^{-1}(\{ e_{h_1(x)}(\vec{i}) \})$ is open in $f^*(A^n)$, this set is equal to $\{\alpha_1(x)^{-1}e_{h_1(x)}(\vec{i}) \}$, similarly we have that $\{\alpha_2(x)^{-1}e_{h_2(x)}(\vec{j}) \} $ is open in $f^*(A^n)$, intersecting both these sets and projecting onto $X$ (recall that the projection map is open in the étale bundle) would allow us to deduce that $\{x \in X \ | \ \alpha_1(x)^{-1} (e_{h_1(x)}(\vec{i})) = \alpha_2(x)^{-1}(e_{h_2(x)}(\vec{j})) \}$ is  open in $X$.
\end{proof}
Moreover, $\tilde{h}$ is an internal natural transformation (isomorphism), between $h_1$ and $h_2$ regarded as internal functors between topological groupoids.

\begin{lemma}
   In the same setting as before, $\tilde{h}$ is the unique natural transformation between $h_1$ and $h_2$ in the category $\cat{Top \ Groupoid}/ E$

\label{unique lemma}
   
\end{lemma}

\begin{proof}
Suppose  that $\alpha$ is a natural transformation between $f$ and $g$ (seen as internal functors between $X$ and $\ca{G}_{\bullet}$ such that $\alpha$ is a part of the data of a $2$-morphism in the category $\cat{Top \ Groupoid }/E$.

    We may think of the category $\cat{Top \ Groupoid}/ E$ as a lax slice category as we did in the section \ref{2-Fibrations}, but this time all categorical constructions appearing in the definition should be thought as internal to the category of toposes so what we are asking is looking at the following diagram of categories internal to $\cat{Toposes}$:

 \[   % https://q.uiver.app/#q=WzAsMyxbMCwwLCJYIl0sWzIsMCwiXFxtYXRoY2Fse0d9X3tcXGJ1bGxldH0iXSxbMSwyLCJFIl0sWzAsMSwiaF8yIiwyXSxbMCwxLCJoXzEiLDAseyJvZmZzZXQiOi0zfV0sWzEsMiwiZyJdLFswLDIsImYiLDJdLFs0LDMsIlxcYWxwaGEiLDAseyJzaG9ydGVuIjp7InNvdXJjZSI6MjAsInRhcmdldCI6MjB9fV0sWzYsNSwiXFxhbHBoYV8xIiwyLHsic2hvcnRlbiI6eyJzb3VyY2UiOjIwLCJ0YXJnZXQiOjIwfX1dLFs2LDUsIlxcYWxwaGFfMiIsMCx7Im9mZnNldCI6LTMsInNob3J0ZW4iOnsic291cmNlIjoyMCwidGFyZ2V0IjoyMH19XV0=
\begin{tikzcd}
	X && {\mathcal{G}_{\bullet}} \\
	\\
	& E
	\arrow[""{name=0, anchor=center, inner sep=0}, "{h_2}"', from=1-1, to=1-3]
	\arrow[""{name=1, anchor=center, inner sep=0}, "{h_1}", shift left=3, from=1-1, to=1-3]
	\arrow[""{name=2, anchor=center, inner sep=0}, "f"', from=1-1, to=3-2]
	\arrow[""{name=3, anchor=center, inner sep=0}, "g", from=1-3, to=3-2]
	\arrow["\alpha", between={0.2}{0.8}, Rightarrow, nfold, from=1, to=0]
	\arrow["{\alpha_1}"', between={0.2}{0.8}, Rightarrow, nfold, from=2, to=3]
	\arrow["{\alpha_2}", shift left=3, between={0.2}{0.8}, Rightarrow, nfold, from=2, to=3]
\end{tikzcd}
\]

We are asking for $(g \circ \alpha) * \alpha_1 =\alpha_2$, this implies that $g \circ \alpha =\alpha_2 * \alpha_1^{-1}$. Looking at what $g$ does on points and their morphisms, we get that $\alpha =\tilde{h}$.

\end{proof}

The fact that the category $\int p_X$ is non-empty follows from \cite[remark 2.4]{butz1999topological} together with \ref{h-lemma}.

Now, let $X \xrightarrow{f} E$ be a topological space  over $E$, and suppose that we have two groupoids over $E$, which are elements of the image of the  $\ca{P}$ by the functor $G$, and two maps of groupoids over $E$ (in the pseudo-slice category $\cat{Top}/E$), $h_1$ and $h_2$ from $X \xrightarrow{f} E$ to these two groupoids respectively. Let us denote the first groupoid by $\ca{G}_{\bullet} \xrightarrow{g} E$ and the second groupoid by $\ca{G}^{'}_{\bullet} \xrightarrow{g^{'}} E$. Since the category $\ca{P}$ is filtered, we may find  a third  groupoid $\ca{G}^{''}_{\bullet}\xrightarrow{g^{''}} E$ (in the image of $G$) over $E$ such that:

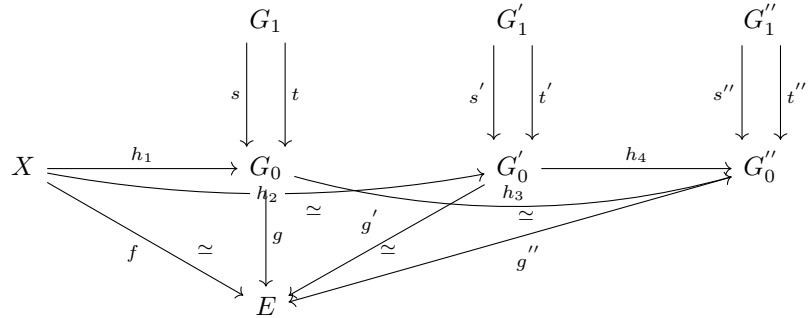
\[
\begin{tikzcd}
	& E &&& \\
	\\
	\\
	& X &&& {G''_0} \\
	\\
	&&&& {G_1''} \\
	{G_0} && {G_0'} \\
	\\
	{G_1} && {G_1'}
	\arrow["f", from=4-2, to=1-2]
	\arrow["{h_1}"{description}, from=4-2, to=7-1]
	\arrow["{h_2}", shift right=2, from=4-2, to=7-3]
	\arrow[""{name=0, anchor=center, inner sep=0}, "{g''}"{description}, from=4-5, to=1-2]
	\arrow["{s''}", shift left=2, from=6-5, to=4-5]
	\arrow["{t''}"', shift right=3, from=6-5, to=4-5]
	\arrow[""{name=1, anchor=center, inner sep=0}, "g", curve={height=-18pt}, from=7-1, to=1-2]
	\arrow["{h_3}"{description}, from=7-1, to=4-5]
	\arrow[""{name=2, anchor=center, inner sep=0}, "{g'}"{description}, curve={height=18pt}, from=7-3, to=1-2]
	\arrow["{h_4}"{description}, from=7-3, to=4-5]
	\arrow["s", shift left=2, from=9-1, to=7-1]
	\arrow["t"', shift right=3, from=9-1, to=7-1]
	\arrow["{s'}", shift left=2, from=9-3, to=7-3]
	\arrow["{t'}"', shift right=3, from=9-3, to=7-3]
	\arrow["\simeq"{description}, draw=none, from=1, to=0]
	\arrow["\simeq"', draw=none, from=1, to=4-2]
	\arrow["\simeq"', draw=none, from=2, to=0]
	\arrow["\simeq", draw=none, from=2, to=4-2]
\end{tikzcd}
\]

In this diagram, the equivalences are between:

\begin{itemize}
    \item $f$ and $g \circ h_1$
    \item $f$ and $g^{'} \circ h_2$
    \item $g$ and $ g^{''} \circ h_3$
    \item $g^{'}$ and $g^{''} \circ h_4$
\end{itemize}

Now by the result we already showed, there exists a continuous function $\bar{h}$ between $X$ and $G^{''}_{1}$, such that $s^{''} \circ \bar{h}= h_3 \circ h_1$ and  $h_4 \circ h_2= t^{''} \circ  \bar{h}$. So we have shown the second condition of filteredness of the category $\int p_X$.

Now for the third condition of filteredness of $\int p_X$, which is the condition on parallel pairs of morphisms, we are going to show something stronger which is that in the category $\int p_X$ every two parallel pairs of arrows are equal, to do suppose that we are given two parallel arrows in $\int p_X$:

\[
% https://q.uiver.app/#q=WzAsNixbMCwyLCJYIl0sWzMsMiwiR18wIl0sWzMsMCwiR18xIl0sWzEsNCwiRSJdLFs2LDIsIkdfMCciXSxbNiwwLCJHJ18xIl0sWzAsMSwiaF8xIiwyXSxbMiwxXSxbMiwxLCIiLDIseyJvZmZzZXQiOi0zfV0sWzAsMywiZiIsMl0sWzEsMywiZyIsMl0sWzEsNCwiaF8yIiwyXSxbNSw0XSxbNSw0LCIiLDAseyJvZmZzZXQiOi0zfV0sWzAsNSwiXFxhbHBoYV8yIiwxLHsibGFiZWxfcG9zaXRpb24iOjMwfV0sWzAsNSwiXFxhbHBoYV8xIiwxLHsibGFiZWxfcG9zaXRpb24iOjMwLCJvZmZzZXQiOi0zfV0sWzAsNCwiaF8zIiwyLHsiY3VydmUiOjR9XSxbNCwzLCJsIiwxXV0=
\begin{tikzcd}
	&&& {G_1} &&& {G'_1} \\
	\\
	X &&& {G_0} &&& {G_0'} \\
	\\
	& E
	\arrow[from=1-4, to=3-4]
	\arrow[shift left=3, from=1-4, to=3-4]
	\arrow[from=1-7, to=3-7]
	\arrow[shift left=3, from=1-7, to=3-7]
	\arrow["{\gamma_2}"{description, pos=0.3}, from=3-1, to=1-7]
	\arrow["{\gamma_1}"{description, pos=0.3}, shift left=3, from=3-1, to=1-7]
	\arrow["{h_1}"', from=3-1, to=3-4]
	\arrow["{h_3}"', curve={height=24pt}, from=3-1, to=3-7]
	\arrow["f"', from=3-1, to=5-2]
	\arrow["{h_2}"', from=3-4, to=3-7]
	\arrow["g"', from=3-4, to=5-2]
	\arrow["l"{description}, from=3-7, to=5-2]
\end{tikzcd}
\]

the data of a two parallel arrows in $\int p_x$ is given by

\begin{itemize}
        \item two objects of $\int p_X$, that is two morphisms $(h_1,\alpha_1)$, and $(h_3,\alpha_3)$ from $X \xrightarrow{f} E$ two topological groupoids $\ca{G}_{\bullet} \xrightarrow{g} E$ and $\ca{G}'_{\bullet} \xrightarrow{l} E$, which are in the image of $G$.

        \item a morphism of topological groupoids $h_2$ between $\ca{G}_{\bullet}$ and $\ca{G'}_{\bullet}$, which is in the image of $G$ (the Grothendieck construction requires two of these but $P$ is a poset), and two natural transformations in $\cat{Top Groupoid}/E$ between $h_3$ and $h_2 \circ h_1$

\end{itemize}

But using \ref{unique lemma} $\gamma_1$ and $\gamma_2$ must be equal, and hence every two parallel arrows in $\int p_X$ are equal.

This shows that the category $\int p_X = \int_{(I,P) \in \ca{P}}\la{Hom}(X\xrightarrow{f}E, G(I,P)) $ is filtered, and hence we may deduce that $\varinjlim(p)= * $, and hence we may deduce that $\varinjlim( \ca{F})=E$ (recall that $\ca{F}$ is the forgetful functor from the category $\cat{Top}/E$ to the category of Grothendieck toposes with enough points).

\begin{comment}
    chekc that you need enough points
\end{comment}

Now we are going to use the fact that the colimit of the forgetful functor from $\cat{Top}/E$ to the category of Grothendieck toposes is $E$, to show that there is an equivalence of categories between $\la{Geom}(E,E^{'})$ and $\la{ClovenCart}(\cat{Top}//E, \cat{Top}//E^{'})$, towards this, we clearly have a functor from $\la{Geom}(E,E^{'})$  to $\la{ClovenCart}(\cat{Top}//E, \cat{Top}//E^{'})$. First, we show it's essentially surjective.

Suppose that we have a Cloven Cartesian functor between $\cat{Top}//E$ and $\cat{Top}//E^{'}$ , we start by constructing a cocone structure on $E^{'}$ for the forgetful functor from $\cat{Top}/E$ to $\cat{Topos}$, first notice that for every object $X \xrightarrow{f} E$, we may simply consider the object   $X \xrightarrow{\ca{L}(f)} E$, Now suppose that we have a commutative diagram up to isomorphism:

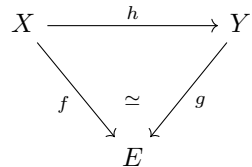
\[
% https://q.uiver.app/#q=WzAsMyxbMCwwLCJYIl0sWzIsMCwiWSJdLFsxLDIsIkUiXSxbMCwyLCJmIiwyXSxbMSwyLCJnIl0sWzAsMSwiaCJdLFszLDQsIlxcc2ltZXEiLDIseyJzaG9ydGVuIjp7InNvdXJjZSI6MjAsInRhcmdldCI6MjB9LCJzdHlsZSI6eyJib2R5Ijp7Im5hbWUiOiJub25lIn0sImhlYWQiOnsibmFtZSI6Im5vbmUifX19XV0=
\begin{tikzcd}
	X && Y \\
	\\
	& E
	\arrow["h", from=1-1, to=1-3]
	\arrow[""{name=0, anchor=center, inner sep=0}, "f"', from=1-1, to=3-2]
	\arrow[""{name=1, anchor=center, inner sep=0}, "g", from=1-3, to=3-2]
	\arrow["\simeq"', draw=none, from=0, to=1]
\end{tikzcd}
\]
  
We have $\ca{L}(f) \simeq \ca{L}(g \circ h) = \ca{L}(g) \circ h$ (using the fact that $\ca{L}$ is cloven Cartesian), also by the fact that $\ca{L}$ is Cartesian we get that this family of isomorphisms is compatible with compositions and hence we got a $2$-cocone for the forgetful functor from $\cat{Top}/E$ to $\cat{Topos}$. and hence, by the fact that $E$ is the $2$-colimit, there exists a geometric morphism $q$ from $E$ to $E^{'}$ such that $\ca{L}(f) \simeq q \circ f$ via a natural transformation $\alpha_f$, such that this natural transformation satisfies the following condition: Suppose that we have a morphism $\alpha$ in the category $\cat{Top}/E $, then the following diagram commutes:

\begin{equation*}% https://q.uiver.app/#q=WzAsOSxbMCwwLCJZIl0sWzIsNCwiRSJdLFs0LDAsIlgiXSxbNiw0LCJFXnsnfSJdLFs4LDAsIlkiXSxbOCw0LCJFIl0sWzEyLDQsIkVeeyd9Il0sWzcsMiwiPSJdLFsxMiwwLCJYIl0sWzAsMSwiZiIsMl0sWzAsMiwiaCJdLFsyLDEsImciXSxbMSwzLCJxIiwyXSxbMiwzLCJcXGNhe0x9KGcpIl0sWzUsNiwicSIsMl0sWzQsNiwiXFxjYXtMfShmKSIsMl0sWzQsNSwiZiIsMl0sWzgsNiwiXFxjYXtMfShnKSJdLFs0LDgsImgiXSxbOSwxMSwiXFxhbHBoYSIsMix7InNob3J0ZW4iOnsic291cmNlIjoyMCwidGFyZ2V0IjoyMH19XSxbMTEsMTMsIlxcYWxwaGFfZyIsMix7InNob3J0ZW4iOnsic291cmNlIjoyMCwidGFyZ2V0IjoyMH19XSxbMTYsMTUsIlxcYWxwaGFfe2Z9IiwwLHsic2hvcnRlbiI6eyJzb3VyY2UiOjIwLCJ0YXJnZXQiOjIwfX1dLFsxNSwxNywiXFxjYXtMfShcXGFscGhhKSIsMCx7InNob3J0ZW4iOnsic291cmNlIjoyMCwidGFyZ2V0IjoyMH19XV0=
\begin{tikzcd}[ampersand replacement=\&]
	Y \&\&\&\& X \&\&\&\& Y \&\&\&\& X \\
	\\
	\&\&\&\&\&\&\& {=} \\
	\\
	\&\& E \&\&\&\& {E^{'}} \&\& E \&\&\&\& {E^{'}}
	\arrow["h", from=1-1, to=1-5]
	\arrow[""{name=0, anchor=center, inner sep=0}, "f"', from=1-1, to=5-3]
	\arrow[""{name=1, anchor=center, inner sep=0}, "g", from=1-5, to=5-3]
	\arrow[""{name=2, anchor=center, inner sep=0}, "{\ca{L}(g)}", from=1-5, to=5-7]
	\arrow["h", from=1-9, to=1-13]
	\arrow[""{name=3, anchor=center, inner sep=0}, "f"', from=1-9, to=5-9]
	\arrow[""{name=4, anchor=center, inner sep=0}, "{\ca{L}(f)}"', from=1-9, to=5-13]
	\arrow[""{name=5, anchor=center, inner sep=0}, "{\ca{L}(g)}", from=1-13, to=5-13]
	\arrow["q"', from=5-3, to=5-7]
	\arrow["q"', from=5-9, to=5-13]
	\arrow["\alpha"', shorten <=13pt, shorten >=13pt, Rightarrow, from=0, to=1]
	\arrow["{\alpha_g}"', shorten <=13pt, shorten >=13pt, Rightarrow, from=1, to=2]
	\arrow["{\alpha_{f}}", shorten <=13pt, shorten >=13pt, Rightarrow, from=3, to=4]
	\arrow["{\ca{L}(\alpha)}", shorten <=13pt, shorten >=13pt, Rightarrow, from=4, to=5]
\end{tikzcd}
\label{weaK naturality} \end{equation*}

This is not quite what we are looking for, in order that the family $(\alpha_f)_{f \in \cat{Top}//E}$ (or equivalently $(\alpha_f)_{f \in \cat{Top}/E}$, the two categories have the same family of objects), to define a natural transformation between the cloven fibration $\ca{L}$ and the cloven fibration $q \circ -$, we need the commutativity of the diagram above for any morphism in $\cat{Top}//E$ i.e. we want to replace $\alpha$ with any natural transformation and we want commutativity of the same diagram, the fact that $\ca{L}$ is cloven allows us to restrict to the case where $X=Y$ and $h=\la{id}$. Now, suppose we have the following diagram:

\[
% https://q.uiver.app/#q=WzAsMixbMCwwLCJYIl0sWzAsMywiRSJdLFswLDEsImYiLDIseyJjdXJ2ZSI6Mn1dLFswLDEsImciLDAseyJjdXJ2ZSI6LTJ9XSxbMiwzLCJcXGFscGhhIiwyLHsic2hvcnRlbiI6eyJzb3VyY2UiOjIwLCJ0YXJnZXQiOjIwfX1dXQ==
\begin{tikzcd}
	X \\
	\\
	\\
	E
	\arrow[""{name=0, anchor=center, inner sep=0}, "f"', curve={height=12pt}, from=1-1, to=4-1]
	\arrow[""{name=1, anchor=center, inner sep=0}, "g", curve={height=-12pt}, from=1-1, to=4-1]
	\arrow["\alpha"', shorten <=5pt, shorten >=5pt, Rightarrow, from=0, to=1]
\end{tikzcd}
\]

Here $\alpha$ is just a natural transformation. We employ the following trick, a pair of parallel geometric morphisms together with a natural transformation between them can be identified with a geometric morphism from $S \times X$ to $E$, where $S$ is the Sierpinski space (to understand this notice that open sets in $S \times X$ can be identified with pairs of open sets of $X$, with the first included inside the second). Before continuing let us make sure that $\ca{L}$ respects this identification, suppose that $\tilde{\alpha}$ is the geometric morphism from $S \times X$ to $E$, then we may recover $\alpha$ as follows:

\[
% https://q.uiver.app/#q=WzAsMyxbMiwwLCJTXFx0aW1lcyBYIl0sWzIsMywiRSJdLFswLDAsIlMiXSxbMCwxLCJcXEFscGhhIl0sWzIsMCwidCIsMix7Im9mZnNldCI6NH1dLFsyLDAsInMiLDAseyJvZmZzZXQiOi00fV0sWzUsNCwiIiwwLHsic2hvcnRlbiI6eyJzb3VyY2UiOjIwLCJ0YXJnZXQiOjIwfX1dXQ==
\begin{tikzcd}
	X && {S\times X} \\
	\\
	\\
	&& E
	\arrow[""{name=0, anchor=center, inner sep=0}, "t"', shift right=4, from=1-1, to=1-3]
	\arrow[""{name=1, anchor=center, inner sep=0}, "s", shift left=4, from=1-1, to=1-3]
	\arrow["\tilde{\alpha}", from=1-3, to=4-3]
	\arrow[shorten <=2pt, shorten >=2pt, Rightarrow, from=1, to=0]
\end{tikzcd}
\]

The first thing one needs to notice is the fact that $\ca{L}$ respects the equivalence above i.e. we have the following equality : 
\[
% https://q.uiver.app/#q=WzAsMyxbMiwwLCJTXFx0aW1lcyBYIl0sWzIsMywiRV57J30iXSxbMCwwLCJTIl0sWzAsMSwiXFxjYXtMfShcXGJhcntcXGFscGhhfSkgPVxcYmFye1xcY2F7TH0oXFxhbHBoYSl9Il0sWzIsMCwidCIsMix7Im9mZnNldCI6NH1dLFsyLDAsInMiLDAseyJvZmZzZXQiOi00fV0sWzUsNCwiIiwwLHsic2hvcnRlbiI6eyJzb3VyY2UiOjIwLCJ0YXJnZXQiOjIwfX1dXQ==
\begin{tikzcd}[ampersand replacement=\&]
	X \&\& {S\times X} \\
	\\
	\\
	\&\& {E^{'}}
	\arrow[""{name=0, anchor=center, inner sep=0}, "t"', shift right=4, from=1-1, to=1-3]
	\arrow[""{name=1, anchor=center, inner sep=0}, "s", shift left=4, from=1-1, to=1-3]
	\arrow["{\ca{L}(\bar{\alpha}) =\overline{\ca{L}(\alpha)}}", from=1-3, to=4-3]
	\arrow[shorten <=2pt, shorten >=2pt, Rightarrow, from=1, to=0]
\end{tikzcd}
\]

This equality results from the fact that $\ca{L}$ is a cloven Cartesian functor between $\cat{Top}//E$ and $\cat{Top}//E^{'}$ considered as $2$-categories $2$-fibred over the $2$-category  $\cat{Top}$.

Now we turn to our goal which is showing the following:

\[
% https://q.uiver.app/#q=WzAsOCxbMCw0LCJFIl0sWzIsMCwiWCJdLFs0LDQsIkVeeyd9Il0sWzYsMCwiWCJdLFs2LDQsIkUiXSxbMTAsNCwiRV57J30iXSxbNSwyLCI9Il0sWzEwLDBdLFsxLDAsImciLDAseyJjdXJ2ZSI6LTN9XSxbMCwyLCJxIiwyXSxbMSwyLCJcXGNhe0x9KGcpIl0sWzQsNSwicSIsMl0sWzMsNSwiXFxjYXtMfShmKSIsMl0sWzMsNCwiZiIsMl0sWzEsMCwiZiIsMix7ImN1cnZlIjo1fV0sWzMsNSwiXFxjYXtMfShnKSIsMCx7ImN1cnZlIjotNX1dLFsxMywxMiwiXFxhbHBoYV97Zn0iLDAseyJzaG9ydGVuIjp7InNvdXJjZSI6MjAsInRhcmdldCI6MjB9fV0sWzE0LDgsIlxcYWxwaGEiLDIseyJzaG9ydGVuIjp7InNvdXJjZSI6MjAsInRhcmdldCI6MjB9fV0sWzgsMiwiXFxhbHBoYV9nIiwyLHsic2hvcnRlbiI6eyJzb3VyY2UiOjIwLCJ0YXJnZXQiOjIwfX1dLFsxMiwxNSwiXFxjYXtMfShcXGFscGhhKSIsMCx7InNob3J0ZW4iOnsic291cmNlIjoyMCwidGFyZ2V0IjoyMH19XV0=
\begin{tikzcd}[ampersand replacement=\&]
	\&\& X \&\&\&\& X \&\&\&\& {} \\
	\\
	\&\&\&\&\& {=} \\
	\\
	E \&\&\&\& {E^{'}} \&\& E \&\&\&\& {E^{'}}
	\arrow[""{name=0, anchor=center, inner sep=0}, "g", curve={height=-18pt}, from=1-3, to=5-1]
	\arrow[""{name=1, anchor=center, inner sep=0}, "f"', curve={height=30pt}, from=1-3, to=5-1]
	\arrow["{\ca{L}(g)}", from=1-3, to=5-5]
	\arrow[""{name=2, anchor=center, inner sep=0}, "f"', from=1-7, to=5-7]
	\arrow[""{name=3, anchor=center, inner sep=0}, "{\ca{L}(f)}"', from=1-7, to=5-11]
	\arrow[""{name=4, anchor=center, inner sep=0}, "{\ca{L}(g)}", curve={height=-30pt}, from=1-7, to=5-11]
	\arrow["q"', from=5-1, to=5-5]
	\arrow["q"', from=5-7, to=5-11]
	\arrow["\alpha"', shorten <=9pt, shorten >=9pt, Rightarrow, from=1, to=0]
	\arrow["{\alpha_g}"', shorten <=15pt, shorten >=15pt, Rightarrow, from=0, to=5-5]
	\arrow["{\alpha_{f}}", shorten <=13pt, shorten >=13pt, Rightarrow, from=2, to=3]
	\arrow["{\ca{L}(\alpha)}", shorten <=5pt, shorten >=5pt, Rightarrow, from=3, to=4]
\end{tikzcd}
\]

In what follows let $\circ$ denote the horizontal composition of natural transformations and let $*$ denote the vertical composition, our goal is to show that $\alpha_g * (q \circ \alpha) = \ca{L}(\alpha) * \alpha_f$. Let us look at the following diagram:

\[
% https://q.uiver.app/#q=WzAsNCxbMCwwLCJYIl0sWzMsMCwiUyBcXHRpbWVzIFgiXSxbMywzLCJFIl0sWzYsMywiRV57J30iXSxbMCwxLCJzIiwwLHsib2Zmc2V0IjotM31dLFswLDEsInQiLDIseyJvZmZzZXQiOjN9XSxbMSwyLCJcXGJhcntcXGFscGhhfSJdLFsxLDMsIlxcY2F7TH0oXFxiYXJ7XFxhbHBoYX0pPVxcb3ZlcmxpbmV7XFxjYXtMfShcXGFscGhhKX0iXSxbMiwzLCJxIiwxXSxbNCw1LCIiLDAseyJzaG9ydGVuIjp7InNvdXJjZSI6MjAsInRhcmdldCI6MjB9fV0sWzYsNywiXFxhbHBoYV97XFxiYXJ7XFxhbHBoYX19IiwwLHsic2hvcnRlbiI6eyJzb3VyY2UiOjIwLCJ0YXJnZXQiOjIwfX1dXQ==
\begin{tikzcd}[ampersand replacement=\&]
	X \&\&\& {S \times X} \\
	\\
	\\
	\&\&\& E \&\&\& {E^{'}}
	\arrow[""{name=0, anchor=center, inner sep=0}, "s", shift left=3, from=1-1, to=1-4]
	\arrow[""{name=1, anchor=center, inner sep=0}, "t"', shift right=3, from=1-1, to=1-4]
	\arrow[""{name=2, anchor=center, inner sep=0}, "{\bar{\alpha}}", from=1-4, to=4-4]
	\arrow[""{name=3, anchor=center, inner sep=0}, "{\ca{L}(\bar{\alpha})=\overline{\ca{L}(\alpha)}}", from=1-4, to=4-7]
	\arrow["q"{description}, from=4-4, to=4-7]
	\arrow[shorten <=2pt, shorten >=2pt, Rightarrow, from=0, to=1]
	\arrow["{\alpha_{\bar{\alpha}}}", shorten <=10pt, shorten >=10pt, Rightarrow, from=2, to=3]
\end{tikzcd}
\]

 Our first claim is that $\alpha_{\bar{\alpha}} \circ s =\alpha_f$ and similarly $\alpha_{\bar{\alpha}} \circ t =\alpha_g$. This follows from applying \ref{weaK naturality} in the case where the natural transformation $\alpha$ is respectively $\la{id}_s$ and $\la{id}_t$. Let us call $u$ the natural transformation between  $s$ and $t$, then we can write the following:

 \[
 \adjustbox{max width = \textwidth}{% https://q.uiver.app/#q=WzAsMTksWzMsMCwiWCJdLFs2LDAsIlMgXFx0aW1lcyBYIl0sWzksMCwiRV57J30iXSxbMTAsMCwiPSJdLFswLDIsIlgiXSxbMywyLCJTIFxcdGltZXMgWCJdLFs2LDIsIkVeeyd9Il0sWzcsMiwiKiJdLFs4LDIsIlgiXSxbMTEsMiwiUyBcXHRpbWVzIFgiXSxbMTQsMiwiRV57J30iXSxbMTUsMiwiPSJdLFs3LDQsIioiXSxbMyw0LCJTIFxcdGltZXMgWCJdLFswLDQsIlgiXSxbNiw0LCJFXnsnfSJdLFs4LDQsIlgiXSxbMTEsNCwiUyBcXHRpbWVzIFgiXSxbMTQsNCwiRV57J30iXSxbMCwxLCJ0IiwyLHsib2Zmc2V0IjozfV0sWzAsMSwicyIsMCx7Im9mZnNldCI6LTN9XSxbMSwyLCJxIFxcY2lyYyBcXGJhcntcXGFscGhhfSIsMCx7Im9mZnNldCI6LTN9XSxbMSwyLCJcXGNhe0x9KFxcYmFye1xcYWxwaGF9KSIsMix7Im9mZnNldCI6M31dLFs0LDUsInMiLDAseyJvZmZzZXQiOi0zfV0sWzQsNSwidCIsMix7Im9mZnNldCI6M31dLFs1LDYsIiBcXGNhe0x9KFxcYmFye1xcYWxwaGF9KSJdLFs4LDksInMiXSxbOSwxMCwicSBcXGNpcmMgXFxiYXJ7XFxhbHBoYX0iLDAseyJvZmZzZXQiOi0zfV0sWzksMTAsIlxcY2F7TH0oXFxiYXJ7XFxhbHBoYX0pIiwyLHsib2Zmc2V0IjozfV0sWzE0LDEzLCJ0Il0sWzEzLDE1LCJxIFxcY2lyYyBcXGJhcntcXGFscGhhfSIsMCx7Im9mZnNldCI6LTN9XSxbMTMsMTUsIlxcY2F7TH0oXFxiYXJ7XFxhbHBoYX0pIiwyLHsib2Zmc2V0IjozfV0sWzE2LDE3LCJzIiwwLHsib2Zmc2V0IjotM31dLFsxNiwxNywidCIsMix7Im9mZnNldCI6M31dLFsxNywxOCwicSBcXGNpcmMgXFxiYXJ7XFxhbHBoYX0iXSxbMjAsMTksInUiLDIseyJzaG9ydGVuIjp7InNvdXJjZSI6MjAsInRhcmdldCI6MjB9fV0sWzIxLDIyLCJcXGFscGhhX3tcXGJhcntcXGFscGhhfX0iLDIseyJzaG9ydGVuIjp7InNvdXJjZSI6MjAsInRhcmdldCI6MjB9fV0sWzI3LDI4LCJcXGFscGhhX3tcXGJhcntcXGFscGhhfX0iLDAseyJzaG9ydGVuIjp7InNvdXJjZSI6MjAsInRhcmdldCI6MjB9fV0sWzMwLDMxLCJcXGFscGhhX3tcXGJhcntcXGFscGhhfX0iLDAseyJzaG9ydGVuIjp7InNvdXJjZSI6MjAsInRhcmdldCI6MjB9fV0sWzIzLDI0LCJ1IiwyLHsic2hvcnRlbiI6eyJzb3VyY2UiOjIwLCJ0YXJnZXQiOjIwfX1dLFszMiwzMywidSIsMCx7InNob3J0ZW4iOnsic291cmNlIjoyMCwidGFyZ2V0IjoyMH19XV0=
\begin{tikzcd}
	&&& X &&& {S \times X} &&& {E^{'}} & {=} \\
	\\
	X &&& {S \times X} &&& {E^{'}} & {*} & X &&& {S \times X} &&& {E^{'}} & {=} \\
	\\
	X &&& {S \times X} &&& {E^{'}} & {*} & X &&& {S \times X} &&& {E^{'}}
	\arrow[""{name=0, anchor=center, inner sep=0}, "t"', shift right=3, from=1-4, to=1-7]
	\arrow[""{name=1, anchor=center, inner sep=0}, "s", shift left=3, from=1-4, to=1-7]
	\arrow[""{name=2, anchor=center, inner sep=0}, "{q \circ \bar{\alpha}}", shift left=3, from=1-7, to=1-10]
	\arrow[""{name=3, anchor=center, inner sep=0}, "{\ca{L}(\bar{\alpha})}"', shift right=3, from=1-7, to=1-10]
	\arrow[""{name=4, anchor=center, inner sep=0}, "s", shift left=3, from=3-1, to=3-4]
	\arrow[""{name=5, anchor=center, inner sep=0}, "t"', shift right=3, from=3-1, to=3-4]
	\arrow["{ \ca{L}(\bar{\alpha})}", from=3-4, to=3-7]
	\arrow["s", from=3-9, to=3-12]
	\arrow[""{name=6, anchor=center, inner sep=0}, "{q \circ \bar{\alpha}}", shift left=3, from=3-12, to=3-15]
	\arrow[""{name=7, anchor=center, inner sep=0}, "{\ca{L}(\bar{\alpha})}"', shift right=3, from=3-12, to=3-15]
	\arrow["t", from=5-1, to=5-4]
	\arrow[""{name=8, anchor=center, inner sep=0}, "{q \circ \bar{\alpha}}", shift left=3, from=5-4, to=5-7]
	\arrow[""{name=9, anchor=center, inner sep=0}, "{\ca{L}(\bar{\alpha})}"', shift right=3, from=5-4, to=5-7]
	\arrow[""{name=10, anchor=center, inner sep=0}, "s", shift left=3, from=5-9, to=5-12]
	\arrow[""{name=11, anchor=center, inner sep=0}, "t"', shift right=3, from=5-9, to=5-12]
	\arrow["{q \circ \bar{\alpha}}", from=5-12, to=5-15]
	\arrow["u"', shorten <=2pt, shorten >=2pt, Rightarrow, from=1, to=0]
	\arrow["{\alpha_{\bar{\alpha}}}"', shorten <=2pt, shorten >=2pt, Rightarrow, from=2, to=3]
	\arrow["u"', shorten <=2pt, shorten >=2pt, Rightarrow, from=4, to=5]
	\arrow["{\alpha_{\bar{\alpha}}}", shorten <=2pt, shorten >=2pt, Rightarrow, from=6, to=7]
	\arrow["{\alpha_{\bar{\alpha}}}", shorten <=2pt, shorten >=2pt, Rightarrow, from=8, to=9]
	\arrow["u", shorten <=2pt, shorten >=2pt, Rightarrow, from=10, to=11]
\end{tikzcd}
 }
 \]

 This implies that $ \ca{L}(\alpha) * \alpha_f = (\alpha_g) * (q \circ \alpha) $.

This entire argument shows that the functor from $\la{ClovenCart}(\cat{Top}//E, \cat{Top}//E^{'})$ to the category $\la{Geom}(E,E^{'})$ is essentially surjective. The fact that this functor is fully faithful follows from the fact that $E$ is the colimit of the forgetful functor from $\cat{Top}/E$ to the category $\cat{Topos}$.

 \begin{note*} \normalfont
     The reader may object that our proof that $\mathrm{Geom}(E,E^{'}) \simeq \mathrm{Lult}(M_{E}, M_{E^{'}})$ is unconventional. A more ``natural way'' of thinking is to consider the functor that sends a geometric morphism $f$ in  $\mathrm{Geom}(E,E^{'})$ to the functor that  composes a point of $E$ (a geometric morphism from $\cat{Set}$ to $E$) by $f$, and show that it's an equivalence. Tracing the backward direction of  the equivalence we did show, reveals that the construction above is exactly the reverse equivalence. Stating the theorem this way is closer to the form originally stated by Makkai \cite{makkai88strong,makkai1987stone}.
 \end{note*}

 \begin{note*} \normalfont
     By replacing  the topos $E^{'}$ with the classifying topos of the theory of objects $\ca{S}[\bb{O}]$, we get that for every topos with enough points, there is an equivalence of categories between $\mathrm{Lult}(M_E,\cat{Set})$ and $E$.

     Since we have shown that the notions of left ultrafunctors agree for both generalised ultracategories and for actual ultracategories which are pseudo-algebras for the monad $T$ (since those only can be considered as generalised ultracategories in a canonical way), then our result provides an alternative proof of Lurie's result.
 \end{note*}

\bibliography{bib}

\section*{Appendix A: The underlying category of a generalised ultracategory}
\label{underlying}
\addcontentsline{toc}{section}{\nameref{underlying}}
In this appendix, we show that the underlying category of a generalised ultracategory is indeed a category:

To do so we start by verifying the right unit axiom (of categories namely that for any morphism $h$, we have $\la{id} \circ h =h$) which expresses the fact that the following composition is the identity:
\[
\adjustbox{scale =0.89}{

\begin{tikzcd}[row sep=3em, every label/.append style={font=\small}]
	{\mathrm{Hom}(A,\int_*Ad*)} \\
	{\mathrm{Hom}(A,\int_*Ad*)\times 1} \\
	{\mathrm{Hom}(A,\int_*Ad*)\times \mathrm{Hom}(A,\int_*A d*)} \\
	{\mathrm{Hom}(A,\int_*Ad*)\times \int_*\mathrm{Hom}(A,\int_*A d*)d*} \\
	{\mathrm{Hom}(A, \int_{\coprod_{*}*}Ad*)} \\
	{\mathrm{Hom}(A,\int_*Ad*)}
	\arrow["\simeq", from=1-1, to=2-1]
	\arrow["{\mathrm{id} \times \kappa}", from=2-1, to=3-1]
	\arrow["{\mathrm{id} \times \epsilon^{-1}}", from=3-1, to=4-1]
	\arrow["\beta", from=4-1, to=5-1]
	\arrow["{\Xi_{f_*}}", from=5-1, to=6-1]
\end{tikzcd}

}
\]
\noindent To do so let us have a look at the following diagram:

\[
\adjustbox{scale =0.89}{
% https://q.uiver.app/#q=WzAsNyxbMCwwLCJcXG1hdGhybXtIb219KEEsXFxpbnRfKkFkKikiXSxbMCw3LCJcXG1hdGhybXtIb219KEEsXFxpbnRfKkFkKilcXHRpbWVzIFxcbWF0aHJte0hvbX0oQSxcXGludF8qQSBkKikiXSxbMSw3LCIgXFxIb20oQSxcXGludF8qQWQqKSBcXHRpbWVzIFxcaW50XyoxIGQqIl0sWzAsMTAsIlxcbWF0aHJte0hvbX0oQSxcXGludF8qQWQqKVxcdGltZXMgXFxpbnRfKlxcbWF0aHJte0hvbX0oQSxcXGludF8qQSBkKilkKiJdLFswLDEzLCJcXG1hdGhybXtIb219KEEsIFxcaW50X3tcXGNvcHJvZF97Kn0qfUFkKikiXSxbMCwxNiwiXFxtYXRocm17SG9tfShBLFxcaW50XypBZCopIl0sWzAsNCwiXFxtYXRocm17SG9tfShBLFxcaW50XypBZCopXFx0aW1lcyAxIl0sWzEsMywiXFxtYXRocm17aWR9IFxcdGltZXMgXFxlcHNpbG9uXnstMX0iXSxbMiwzLCJcXGxhe2lkfSBcXHRpbWVzIFxcaW50XypcXGthcHBhIGQqIl0sWzMsNCwiXFxiZXRhIl0sWzQsNSwiXFxYaV97Zl8qfSJdLFs2LDEsIlxcbWF0aHJte2lkfSBcXHRpbWVzIFxca2FwcGEiXSxbNiwyLCJcXGxhe2lkfSBcXHRpbWVzIFxcZXBzaWxvbl57LTF9Il0sWzAsNiwiXFxzaW1lcSJdXQ==
\begin{tikzcd}
	{\mathrm{Hom}(A,\int_*Ad*)} & \\
	\\
	\\
	\\
	{\mathrm{Hom}(A,\int_*Ad*)\times 1} \\
	\\
	\\
	{\mathrm{Hom}(A,\int_*Ad*)\times \mathrm{Hom}(A,\int_*A d*)} & { \Hom(A,\int_*Ad*) \times \int_*1 d*} \\
	\\
	\\
	{\mathrm{Hom}(A,\int_*Ad*)\times \int_*\mathrm{Hom}(A,\int_*A d*)d*} \\
	\\
	\\
	{\mathrm{Hom}(A, \int_{\coprod_{*}*}Ad*)} \\
	\\
	\\
	{\mathrm{Hom}(A,\int_*Ad*)}
	\arrow["\simeq", from=1-1, to=5-1]
	\arrow["{\mathrm{id} \times \kappa}", from=5-1, to=8-1]
	\arrow["{\la{id} \times \epsilon^{-1}}", from=5-1, to=8-2]
	\arrow["{\mathrm{id} \times \epsilon^{-1}}", from=8-1, to=11-1]
	\arrow["{\la{id} \times \int_*\kappa d*}", from=8-2, to=11-1]
	\arrow["\beta", from=11-1, to=14-1]
	\arrow["{\Xi_{f_*}}", from=14-1, to=17-1]
\end{tikzcd}}\]
\noindent The outermost composition is equal to the identity by the right unit axiom of generalised ultracategories, while the square commutes because both compositions are equal to  $\mathrm{id} \times \kappa$.

\noindent Now to verify left unit axiom (of categories namely that for any morphism $h$, we have $h \circ \la{id} =h$), we want the following composition to be equal to the identity :

\[
\begin{tikzcd}[row sep=3em, every label/.append style={font=\small}]
	{\Hom(A,\int_*Ad*)} \\
	{1 \times\mathrm{Hom}(A,\int_*Ad*)} \\
	{\mathrm{Hom}(A,\int_*Ad*) \times \mathrm{Hom}(A,\int_*Ad*)} \\
	{\mathrm{Hom}(A,\int_*Ad*) \times \int_*\mathrm{Hom}(A,\int_*Ad*)d*} \\
	{\Hom(A,\int_{\coprod_**}Ad*)} \\
	{\Hom(A,\int_{\coprod_**}Ad*)}
	\arrow["\simeq", from=1-1, to=2-1]
	\arrow["{\kappa \times \la{id}}", from=2-1, to=3-1]
	\arrow["{\la{id} \times \epsilon^{-1}}", from=3-1, to=4-1]
	\arrow["\beta", from=4-1, to=5-1]
	\arrow["{\Xi_{g_*}=\Xi_{f_*}}", from=5-1, to=6-1]
\end{tikzcd}
\]

To do so let us use the following diagram: 

% https://q.uiver.app/#q=WzAsOCxbMiwzLCJcXGludF8qXFxIb20oQSxcXGludF8qQWQqKWQqIl0sWzIsNiwiMSBcXHRpbWVzXFxpbnRfKlxcSG9tKEEsXFxpbnRfKkFkKilkKiJdLFswLDAsIlxcSG9tKEEsXFxpbnRfKkFkKikiXSxbMCwzLCIxIFxcdGltZXNcXG1hdGhybXtIb219KEEsXFxpbnRfKkFkKikiXSxbMCw2LCJcXG1hdGhybXtIb219KEEsXFxpbnRfKkFkKikgXFx0aW1lcyBcXG1hdGhybXtIb219KEEsXFxpbnRfKkFkKikiXSxbMCw5LCJcXG1hdGhybXtIb219KEEsXFxpbnRfKkFkKikgXFx0aW1lcyBcXGludF8qXFxtYXRocm17SG9tfShBLFxcaW50XypBZCopZCoiXSxbMCwxMywiXFxIb20oQSxcXGludF97XFxjb3Byb2RfKip9QWQqKSJdLFswLDE2LCJcXEhvbShBLFxcaW50X3sqfUFkKikiXSxbMCwxLCJcXHNpbWVxIl0sWzEsNSwiXFxrYXBwYSBcXHRpbWVzIFxcbGF7aWR9Il0sWzIsMCwiXFxlcHNpbG9uXnstMX0iXSxbMiwzLCJcXHNpbWVxIl0sWzMsMSwiXFxsYXtpZH0gXFx0aW1lcyBcXGVwc2lsb25eey0xfSJdLFszLDQsIlxca2FwcGEgXFx0aW1lcyBcXGxhe2lkfSJdLFs0LDUsIlxcbGF7aWR9IFxcdGltZXMgXFxlcHNpbG9uXnstMX0iXSxbNSw2LCJcXGJldGEiXSxbNiw3LCJcXFhpX3tnXyp9PVxcWGlfe2ZfKn0iXV0=
\[\begin{tikzcd}
	{\Hom(A,\int_*Ad*)} && \\
	\\
	\\
	{1 \times\mathrm{Hom}(A,\int_*Ad*)} && {\int_*\Hom(A,\int_*Ad*)d*} \\
	\\
	\\
	{\mathrm{Hom}(A,\int_*Ad*) \times \mathrm{Hom}(A,\int_*Ad*)} && {1 \times\int_*\Hom(A,\int_*Ad*)d*} \\
	\\
	\\
	{\mathrm{Hom}(A,\int_*Ad*) \times \int_*\mathrm{Hom}(A,\int_*Ad*)d*} \\
	\\
	\\
	\\
	{\Hom(A,\int_{\coprod_**}Ad*)} \\
	\\
	\\
	{\Hom(A,\int_{*}Ad*)}
	\arrow["\simeq", from=1-1, to=4-1]
	\arrow["{\epsilon^{-1}}", from=1-1, to=4-3]
	\arrow["{\kappa \times \la{id}}", from=4-1, to=7-1]
	\arrow["{\la{id} \times \epsilon^{-1}}", from=4-1, to=7-3]
	\arrow["\simeq", from=4-3, to=7-3]
	\arrow["{\la{id} \times \epsilon^{-1}}", from=7-1, to=10-1]
	\arrow["{\kappa \times \la{id}}", from=7-3, to=10-1]
	\arrow["\beta", from=10-1, to=14-1]
	\arrow["{\Xi_{g_*}=\Xi_{f_*}}", from=14-1, to=17-1]
\end{tikzcd}\]

The outermost composition is equal to the identity by the left unit axiom of generalised ultracategories, the left square commutes by naturality of the isomorphism between a set $B$ and $1 \times B$. The right square commutes because both compose to $\kappa \times \epsilon^{-1}$.

Finally, we want to show associativity of the composition, which is not straightforward. In the underlying category structure, this translates to showing the commutativity of the following diagram:

\adjustbox{max width = \textwidth}
{
\begin{tikzcd}
	{\mathrm{Hom}(A,\int_* Bd *)\times \mathrm{Hom}(B,\int_*Cd*)\times \int_*\Hom(C,\int_*Dd*)d*} &&&& {\mathrm{Hom}(A,\int_* B  d*) \times \mathrm{Hom}(B,\int_* C d* )\times \Hom(C,\int_*Dd*)} \\
	\\
	\\
	\\
	\\
	{\Hom(A,\int_* B d*)\times \Hom(B,\int_{\coprod_**}Dd*)} &&&& {(\mathrm{Hom}(A,\int_* Bd *)\times \int_*\mathrm{Hom}(B,\int_*Cd*)d*)\times \Hom(C,\int_*Dd*)} \\
	\\
	\\
	\\
	{\Hom(A,\int_*Bd*)\times \Hom(B,\int_*Dd*)} &&&& {\mathrm{Hom}(A,\int_{\coprod_**}Cd*)  \times \Hom(C,\int_*D)} \\
	\\
	\\
	\\
	\\
	&&&& {\mathrm{Hom}(A,\int_{*}Cd*)  \times \int_*\Hom(C,\int_*Dd*)d*} \\
	\\
	\\
	\\
	\\
	{\Hom(A,\int_*Bd*)\times \int_*\Hom(B,\int_*Dd*)d*} &&&& {\Hom(A,\int_{\coprod_**}Dd*)}
	\arrow["{\mathrm{id} \times \beta}", from=1-1, to=6-1]
	\arrow["{\mathrm{id}\times \mathrm{id} \times \epsilon_{*,*}^{-1}}", from=1-5, to=1-1]
	\arrow["{\mathrm{id} \times \epsilon_{*,*}^{-1} \times \mathrm{id}}"', from=1-5, to=6-5]
	\arrow["{\mathrm{id} \times \Xi_{f_*}}", from=6-1, to=10-1]
	\arrow["{\beta \times \mathrm{id} }"', from=6-5, to=10-5]
	\arrow["{\mathrm{id} \times \epsilon^{-1} }", from=10-1, to=20-1]
	\arrow["{\Xi_{f_*} \times \epsilon^{-1}}"', from=10-5, to=15-5]
	\arrow["\beta"', from=15-5, to=20-5]
	\arrow["\beta", from=20-1, to=20-5]
\end{tikzcd}}

\noindent Now  we already know that the following diagram commutes:

\[
\adjustbox{max width =\textwidth}{

\begin{tikzcd}
	{\mathrm{Hom}(A,\int_* B  d*) \times \mathrm{Hom}(B,\int_* C d* )\times \Hom(C,\int_*Dd*)} &&&& {(\mathrm{Hom}(A,\int_* Bd *)\times \int_*\mathrm{Hom}(B,\int_*Cd*)d*\times \Hom(C,\int_*Dd*)} \\
	\\
	\\
	{\mathrm{Hom}(A,\int_* Bd *)\times \mathrm{Hom}(B,\int_*Cd*)\times \int_*\Hom(C,\int_*Dd*)d*} &&&& {\mathrm{Hom}(A,\int_* Bd *)\times \int_*\mathrm{Hom}(B,\int_*Cd*)d*\times \int_*\Hom(C,\int_*Dd*)d*}
	\arrow["{\mathrm{id} \times \epsilon_{*,*}^{-1} \times \mathrm{id}}"', from=1-1, to=1-5]
	\arrow["{\mathrm{id}\times \mathrm{id} \times \epsilon_{*,*}^{-1}}", from=1-1, to=4-1]
	\arrow["{\mathrm{id} \times \mathrm{id} \times \epsilon^{-1}_{*,*}}"', from=1-5, to=4-5]
	\arrow["{\mathrm{id} \times \epsilon^{-1}_{*,*} \times \mathrm{id}}", from=4-1, to=4-5]
\end{tikzcd}}
\]

\noindent So this allows us to simplify the task to showing that the following diagram is commutative:

\[
\adjustbox{ max width = 50 em, max height =100 em}{

\begin{tikzcd}
	{\mathrm{Hom}(A,\int_* Bd *)\times \int_*\mathrm{Hom}(B,\int_*Cd*)d*\times \int_*\Hom(C,\int_*D)d*} &&&& {\mathrm{Hom}(A,\int_{\coprod_**}Cd*)  \times \Hom(C,\int_*D d*)} &&&&& {\mathrm{Hom}(A,\int_{*}Cd*)  \times \int_*\Hom(C,\int_*D)d*} &&&&& {\Hom(A,\int_{\coprod_**}Dd*)} \\
	\\
	\\
	\\
	{\Hom(A,\int_* B d*)\times \Hom(B,\int_{\coprod_**}Dd*)} &&&& {\Hom(A,\int_*Bd*)\times \Hom(B,\int_*Dd*)} &&&&&&&&&& {\Hom(A,\int_*Bd*)\times \int_*\Hom(B,\int_*Dd*)d*}
	\arrow["{  \beta \times \epsilon}"', from=1-1, to=1-5]
	\arrow["{\mathrm{id} \times \epsilon^{-1}}", from=1-1, to=5-1]
	\arrow["{\Xi_{f_*} \times \epsilon^{-1}}"', from=1-5, to=1-10]
	\arrow["\beta"', from=1-10, to=1-15]
	\arrow["{\mathrm{id} \times \Xi_{f_*}}", from=5-1, to=5-5]
	\arrow["{\mathrm{id} \times \epsilon^{-1} }", from=5-5, to=5-15]
	\arrow["\beta", from=5-15, to=1-15]
\end{tikzcd}}
\]

\noindent Now to show this commutativity, let us have a look at the following diagram:

\[
\adjustbox{ max width = 50 em, max height =70 em}{

\begin{tikzcd}
	&&&&&& {\mathrm{Hom}(A,\int_{\coprod_**}Cd*)  \times \Hom(C,\int_*Dd*)} && {\mathrm{Hom}(A,\int_{*}Cd*)  \times \int_*\Hom(C,\int_*D)d*} && {\mathrm{Hom}(A,\int_{\coprod_{*}*}Dd*)} &&& {} \\
	\\
	&&&&&& {\Hom(A,\int_{\coprod_**}Cd*) \times \int_{\coprod_{*}*} \Hom(C,\int_{*}Dd*)d*} && {\Hom(A,\int_{\coprod_{\coprod_{*}*}*}Dd*)} \\
	\\
	&&& {} \\
	&&&&&& {\Hom(A,\int_{\coprod_**}Cd*) \times   \int_*\int_*\Hom(C,\int_*D)d*d*} &&&& {\Hom(A, \int_{\coprod_{*}\coprod_{*}*}Dd*)} \\
	\\
	&&& {\mathrm{Hom}(A,\int_* Bd *)\times \int_*\mathrm{Hom}(B,\int_*Cd*)d*\times \int_*\Hom(C,\int_*D)d*} &&& {\mathrm{Hom}(A,\int_* Bd *)\times \int_*\mathrm{Hom}(B,\int_*Cd*)d*\times \int_*\int_*\Hom(C,\int_*D)d*d*} \\
	\\
	\\
	\\
	\\
	&&& {\Hom(A,\int_* B d*)\times \Hom(B,\int_{\coprod_**}Dd*)} &&& {\mathrm{Hom}(A,\int_* Bd *)\times \int_*\mathrm{Hom}(B,\int_*Cd*)\times \int_* \Hom(C,\int_*D)d*d*} &&&& {\Hom(A,\int_*B d*) \times \int_* \Hom(B,\int_{\coprod_**}D d*)} \\
	\\
	\\
	\\
	\\
	{} \\
	\\
	&&&&&& {} \\
	\\
	{} &&& {\Hom(A,\int_*Bd*)\times \Hom(B,\int_*Dd*)} &&&&&&& {\Hom(A,\int_*Bd*)\times \int_*\Hom(B,\int_*Dd*)d*}
	\arrow[""{name=0, anchor=center, inner sep=0}, "{\la{Id} \times \epsilon^{-1}}", from=1-7, to=1-9]
	\arrow["5", draw=none, from=1-7, to=3-9]
	\arrow["\beta", from=1-9, to=1-11]
	\arrow["{\la{Id} \times\epsilon}", from=3-7, to=1-7]
	\arrow["\beta", from=3-7, to=3-9]
	\arrow["{\Xi_{}}", from=3-9, to=1-11]
	\arrow["{\la{Id} \times a^{-1}}", from=6-7, to=3-7]
	\arrow["\Xi", from=6-11, to=1-11]
	\arrow[""{name=1, anchor=center, inner sep=0}, "{\beta  \times \epsilon^{-1}}", from=8-4, to=1-7]
	\arrow["{\la{Id} \times \la{Id} \times \epsilon^{-1}}", from=8-4, to=8-7]
	\arrow["{(\mathrm{Id} \times \epsilon^{-1})\circ (\mathrm{Id} \times \beta)}", from=8-4, to=13-4]
	\arrow["{\la{Id} \times \beta}", from=8-7, to=6-7]
	\arrow["{\mathrm{commutativity \  of \  products \ with \ directed \  colimits}}", from=8-7, to=13-7]
	\arrow[""{name=2, anchor=center, inner sep=0}, "{\la{Id} \times \epsilon^{-1}}"', curve={height=60pt}, from=13-4, to=13-11]
	\arrow["{\la{Id} \times \Xi_{f_*}}"', from=13-4, to=22-4]
	\arrow[""{name=3, anchor=center, inner sep=0}, "{\la{Id} \times \beta}", from=13-7, to=13-11]
	\arrow["\beta", from=13-11, to=6-11]
	\arrow[""{name=4, anchor=center, inner sep=0}, "{\la{Id} \times \epsilon^{-1}}", from=22-4, to=22-11]
	\arrow["\beta"{description, pos=0.6}, shift right=5, curve={height=80pt}, from=22-11, to=1-11]
	\arrow["4", draw=none, from=0, to=3]
	\arrow["1"', draw=none, from=1, to=8-7]
	\arrow["2"', draw=none, from=8-4, to=2]
	\arrow["3"', draw=none, from=2, to=4]
\end{tikzcd}
}
\]

\noindent First let us have a look at diagram $1$:

\[
\adjustbox{scale =0.7}{
\begin{tikzcd}
	&&&&& {\mathrm{Hom}(A,\int_{\coprod_**}Cd*)  \times \Hom(C,\int_*Dd*)} \\
	\\
	&&&&& {\Hom(A,\int_{\coprod_**}Cd*) \times \int_{\coprod_{*}*} \Hom(C,\int_*Dd*)d*} \\
	{\mathrm{Hom}(A,\int_{\coprod_**}Cd*)  \times \int_* \Hom(C,\int_*Dd*)d*} \\
	&&&&& {\Hom(A,\int_{\coprod_**}Cd*) \times  \int_*\int_*\Hom(C,\int_*Dd*)d*d*} \\
	\\
	{\mathrm{Hom}(A,\int_* Bd *)\times \int_*\mathrm{Hom}(B,\int_*Cd*)d*\times \int_*\Hom(C,\int_*Dd*)d*} &&&&& {\mathrm{Hom}(A,\int_* Bd *)\times \int_*\mathrm{Hom}(B,\int_*Cd*)d*\times \int_*\int_*\Hom(C,\int_*Dd*)d*d*}
	\arrow["{\la{id} \times \epsilon^{-1}}", from=3-6, to=1-6]
	\arrow["{ \la{id} \times \epsilon^{-1}}"', from=4-1, to=1-6]
	\arrow[""{name=0, anchor=center, inner sep=0}, "{\la{id} \times \epsilon^{-1} \circ \epsilon^{-1}}"{description, pos=0.4}, curve={height=70pt}, from=5-6, to=1-6]
	\arrow[""{name=1, anchor=center, inner sep=0}, "{\la{id} \times a^{-1}}", from=5-6, to=3-6]
	\arrow["{\beta \times \la{id}}"', from=7-1, to=4-1]
	\arrow["{\la{id} \times \la{id} \times \epsilon^{-1}}", from=7-1, to=7-6]
	\arrow["{\beta \times \la{id}}", from=7-6, to=5-6]
	\arrow["1", draw=none, from=1, to=0]
\end{tikzcd}}
\]

\noindent The outer most  diagram commutes by naturality of $\epsilon$. Here one should also note that the diagram commutes by definition, while triangle $1$ (in the diagram $1$) commutes by construction of the colax associator \cite{hamad2025ultracategories} and the first of Lurie's axioms, and hence diagram $1$ commutes.

\noindent Now let us regard the  diagram $2$:

\[
\adjustbox{scale =0.8}{
% https://q.uiver.app/#q=WzAsNixbMCwwLCIgXFxpbnRfKlxcbWF0aHJte0hvbX0oQixcXGludF8qQ2QqKWQqXFx0aW1lcyBcXGludF8qXFxIb20oQyxcXGludF8qRGQqKWQqIl0sWzAsNCwiIFxcSG9tKEIsXFxpbnRfe1xcY29wcm9kXyoqfURkKikiXSxbNSwwLCIgXFxpbnRfKlxcbWF0aHJte0hvbX0oQixcXGludF8qQ2QqKWQqXFx0aW1lcyBcXGludF8qXFxpbnRfKlxcSG9tKEMsXFxpbnRfKkQpZCpkKiJdLFs1LDQsIiBcXGludF8qIFxcSG9tKEIsXFxpbnRfe1xcY29wcm9kXyoqfUQgZCopIl0sWzAsMiwiXFxtYXRocm17SG9tfShCLFxcaW50XypDZCopIFxcdGltZXMgXFxpbnRfKlxcSG9tKEMsXFxpbnRfKkRkKilkKiJdLFs1LDIsIlxcaW50XyooXFxtYXRocm17SG9tfShCLFxcaW50XypDZCopXFx0aW1lcyBcXGludF8qIFxcSG9tKEMsXFxpbnRfKkRkKikpZCoiXSxbMCwyLCJcXG1hdGhybXtpZH0gXFx0aW1lcyBcXGVwc2lsb25eey0xfSAiLDJdLFs0LDEsIlxcYmV0YSJdLFswLDQsIiBcXGVwc2lsb25cXCBcXHRpbWVzIFxcbWF0aHJte2lkfSAiXSxbMiw1LCJcXHRleHR7Y29tbXV0YXRpdml0eSAgXFwgb2YgIFxcIHByb2R1Y3RzIFxcIHdpdGggXFwgZGlyZWN0ZWQgXFwgY29saW1pdHN9Il0sWzQsNSwiXFxlcHNpbG9uXnstMX0iLDFdLFs1LDMsIlxcaW50XypcXGJldGEgZCoiXSxbMSwzLCJcXGludF97Kn1cXGVwc2lsb25eey0xfSJdLFs2LDEwLCIxIiwyLHsic2hvcnRlbiI6eyJzb3VyY2UiOjIwLCJ0YXJnZXQiOjIwfSwic3R5bGUiOnsiYm9keSI6eyJuYW1lIjoibm9uZSJ9LCJoZWFkIjp7Im5hbWUiOiJub25lIn19fV0sWzEwLDEyLCIyIiwyLHsic2hvcnRlbiI6eyJzb3VyY2UiOjIwLCJ0YXJnZXQiOjIwfSwic3R5bGUiOnsiYm9keSI6eyJuYW1lIjoibm9uZSJ9LCJoZWFkIjp7Im5hbWUiOiJub25lIn19fV1d
\begin{tikzcd}
	{ \int_*\mathrm{Hom}(B,\int_*Cd*)d*\times \int_*\Hom(C,\int_*Dd*)d*} &&&&& { \int_*\mathrm{Hom}(B,\int_*Cd*)d*\times \int_*\int_*\Hom(C,\int_*D)d*d*} \\
	\\
	{\mathrm{Hom}(B,\int_*Cd*) \times \int_*\Hom(C,\int_*Dd*)d*} &&&&& {\int_*(\mathrm{Hom}(B,\int_*Cd*)\times \int_* \Hom(C,\int_*Dd*))d*} \\
	\\
	{ \Hom(B,\int_{\coprod_**}Dd*)} &&&&& { \int_* \Hom(B,\int_{\coprod_**}D d*)}
	\arrow[""{name=0, anchor=center, inner sep=0}, "{\mathrm{id} \times \epsilon^{-1} }"', from=1-1, to=1-6]
	\arrow["{ \epsilon\ \times \mathrm{id} }", from=1-1, to=3-1]
	\arrow["{\text{commutativity  \ of  \ products \ with \ directed \ colimits}}", from=1-6, to=3-6]
	\arrow[""{name=1, anchor=center, inner sep=0}, "{\epsilon^{-1}}"{description}, from=3-1, to=3-6]
	\arrow["\beta", from=3-1, to=5-1]
	\arrow["{\int_*\beta d*}", from=3-6, to=5-6]
	\arrow[""{name=2, anchor=center, inner sep=0}, "{\int_{*}\epsilon^{-1}}", from=5-1, to=5-6]
	\arrow["1"', draw=none, from=0, to=1]
	\arrow["2"', draw=none, from=1, to=2]
\end{tikzcd}
}
\]

\noindent Starting with square $2$ we can see that its commutative by naturality of $\epsilon$. The commutativity of square $1$ can be done by diagram chasing, or categorically by noticing that the map that we called ``commutativity of products with directed colimits'' is part of the left ultrastructure of the following functor $\cat{Set} \times \cat{Set} \xrightarrow{\times} \cat{Set}$ (Here the ultraproduct in $\cat{Set} \times \cat{Set}$ is defined coordinate-wise).

Now let us regard diagram $3$:

\[
\adjustbox{scale =0.7}{

% https://q.uiver.app/#q=WzAsNixbMCwwLCJcXEhvbShBLFxcaW50XypCIGQqKSBcXHRpbWVzIFxcaW50XyogXFxIb20oQixcXGludF97XFxjb3Byb2RfKip9RCBkKilkKiJdLFswLDQsIlxcSG9tKEEsXFxpbnRfKiBCIGQqKVxcdGltZXMgXFxIb20oQixcXGludF97XFxjb3Byb2RfKip9RGQqKSJdLFswLDcsIlxcSG9tKEEsXFxpbnRfKkJkKilcXHRpbWVzIFxcSG9tKEIsXFxpbnRfKkRkKikiXSxbNiw3LCJcXEhvbShBLFxcaW50XypCZCopXFx0aW1lcyBcXGludF8qXFxIb20oQixcXGludF8qRGQqKWQqIl0sWzExLDcsIlxcSG9tKEEsXFxpbnRfe1xcY29wcm9kXyoqfURkKikiXSxbMTEsMCwiXFxIb20oQSwgXFxpbnRfe1xcY29wcm9kX3sqfVxcY29wcm9kX3sqfSp9RGQqKSJdLFswLDMsIlxcbGF7aWR9IFxcdGltZXMgXFxpbnRfKiBcXGthcHBhIGQqIiwxXSxbMSwwLCJcXG1hdGhybXtpZH0gXFx0aW1lcyBcXGVwc2lsb25eey0xfSJdLFsxLDIsIlxcbWF0aHJte2lkfSBcXHRpbWVzIFxca2FwcGEiLDJdLFsyLDMsIlxcbWF0aHJte2lkfSBcXHRpbWVzIFxcZXBzaWxvbl57LTF9IiwyXSxbMyw0LCJcXGJldGEiLDJdLFs1LDQsIlxca2FwcGEiXSxbMCw1LCJcXGJldGEiXV0=
\begin{tikzcd}
	{\Hom(A,\int_*B d*) \times \int_* \Hom(B,\int_{\coprod_**}D d*)d*} &&&&&&&&&&& {\Hom(A, \int_{\coprod_{*}\coprod_{*}*}Dd*)} \\
	\\
	\\
	\\
	{\Hom(A,\int_* B d*)\times \Hom(B,\int_{\coprod_**}Dd*)} \\
	\\
	\\
	{\Hom(A,\int_*Bd*)\times \Hom(B,\int_*Dd*)} &&&&&& {\Hom(A,\int_*Bd*)\times \int_*\Hom(B,\int_*Dd*)d*} &&&&& {\Hom(A,\int_{\coprod_**}Dd*)}
	\arrow["\beta", from=1-1, to=1-12]
	\arrow["{\la{id} \times \int_* \kappa d*}"{description}, from=1-1, to=8-7]
	\arrow["\kappa", from=1-12, to=8-12]
	\arrow["{\mathrm{id} \times \epsilon^{-1}}", from=5-1, to=1-1]
	\arrow["{\mathrm{id} \times \kappa}"', from=5-1, to=8-1]
	\arrow["{\mathrm{id} \times \epsilon^{-1}}"', from=8-1, to=8-7]
	\arrow["\beta"', from=8-7, to=8-12]
\end{tikzcd}
}
\]
\noindent The leftmost diagram commutes by  naturality of $\la{id} \times \epsilon^{-1}$, and the rightmost diagram commutes by the first diagram of axiom  $6$, and hence diagram $3$ commutes.

\noindent Diagram $4$ commutes by the composition axiom of generalised ultracategories.

\noindent Now let us regard diagram $5$:

\[
\adjustbox{max width=\textwidth}{

\begin{tikzcd}
	{\mathrm{Hom}(A,\int_{\coprod_**}Cd*)  \times \Hom(C,\int_*D)} &&& {\mathrm{Hom}(A,\int_{*}Cd*)  \times \int_*\Hom(C,\int_*D)d*} &&& {\Hom(A,\int_{\coprod_**}D*)} \\
	\\
	{\Hom(A,\int_{\coprod_**}Cd*) \times \int_{\coprod_{*}*} \Hom(C,\int_*Dd*)} &&&&&& {\Hom(A,\int_{\coprod_{\coprod_{*}*}*}Dd*)}
	\arrow["{\mathrm{id} \times \epsilon^{-1}_{*,*}}", from=1-1, to=1-4]
	\arrow["\beta"{description}, from=1-4, to=1-7]
	\arrow["{\mathrm{id} \times \epsilon_{\coprod_{*}*,*}}"{description}, from=3-1, to=1-1]
	\arrow["{\mathrm{id} \times (\epsilon^{-1}_{*,*} \circ \epsilon_{\coprod_{*}*,*})}"{description}, from=3-1, to=1-4]
	\arrow["\beta", from=3-1, to=3-7]
	\arrow["\kappa", from=3-7, to=1-7]
\end{tikzcd}
}
\]
\noindent Notice that the left diagram commutes by definition while the right one commutes by the second diagram of axiom $(6)$.
 
Thus, we can deduce that the underlying category of an ultracategory is really a category.

\section*{Appendix B: Proof that strong ultracategories ($T$-pseudo-algebras) are generalised ultracategories }
\label{not very crucial proof}
\addcontentsline{toc}{section}{\nameref{not very crucial proof}}

Suppose that we have a full subcategory $\ca{C}$ of an ultracategory $\ca{U}$, where $\ca{U}$ is an algebra for the monad $T$. We define the generalised ultrastructure as in \ref{not very crucial theorem}. Now we check that our definition satisfies the axioms of generalised ultracategories.

Axioms $1, 2, 3$  are properties of the ultraproduct diagonal map as defined in \cite{lurie2018ultracategories}. Now for axiom $4$:

We want to show that the following composition:

\[
% https://q.uiver.app/#q=WzAsNSxbMCwxLCJcXGludF97Kn1cXG1hdGhybXtIb219KEEsXFxpbnRfSSBNX2kgZFxcbXUpZCoiXSxbMCwwLCJcXG1hdGhybXtIb219KEEsXFxpbnRfSSBNX2kgZFxcbXUpIl0sWzAsMiwiXFxtYXRoYmZ7MX0gXFx0aW1lcyBcXGludF97Kn1cXG1hdGhybXtIb219KEEsXFxpbnRfaSBNX2kgZFxcbXUgKWQqIl0sWzAsNCwiXFxtYXRocm17SG9tfShBLCBcXGludF9cXGFzdCBBIGQgXFxhc3QgKSBcXHRpbWVzIFxcaW50X3tcXGFzdH1cXG1hdGhybXtIb219KEEsXFxpbnRfSSBNX2kgZFxcbXUgKWQqIl0sWzAsNywiXFxtYXRocm17SG9tfShBLFxcaW50X3tcXGNvcHJvZF8qSX0gTV9pIGRcXG11ICkiXSxbMSwwLCJcXGVwc2lsb25eey0xfV97KiwqfSJdLFswLDIsIlxcc2ltZXEiXSxbMiwzLCJcXGthcHBhIFxcdGltZXMgXFxtYXRocm17aWR9Il0sWzMsNCwiXFxiZXRhIl1d
\begin{tikzcd}
	{\mathrm{Hom}(A,\int_I M_i d\mu)} \\
	{\int_{*}\mathrm{Hom}(A,\int_I M_i d\mu)d*} \\
	{\mathbf{1} \times \int_{*}\mathrm{Hom}(A,\int_i M_i d\mu )d*} \\
	\\
	{\mathrm{Hom}(A, \int_\ast A d \ast ) \times \int_{\ast}\mathrm{Hom}(A,\int_I M_i d\mu )d*} \\
	\\
	\\
	{\mathrm{Hom}(A,\int_{\coprod_*I} M_i d\mu )}
	\arrow["{\epsilon^{-1}_{*,*}}", from=1-1, to=2-1]
	\arrow["\simeq", from=2-1, to=3-1]
	\arrow["{\kappa \times \mathrm{id}}", from=3-1, to=5-1]
	\arrow["\beta", from=5-1, to=8-1]
\end{tikzcd}
\]
is the inverse of  $\Xi_{f_{I},\mu}$, where $f_I$ is the natural isomorphism between $I$ and $\coprod_{*}I$. To do so, let $f \in \mathrm{Hom}(A,\int_I M_i d\mu)$, its image by the series of equivalences is the following: $f \mapsto (f) \mapsto (1, (f)) \mapsto (\epsilon_{*,*}, (f)) \mapsto (\epsilon_{*,*},  (\tilde{f})) \mapsto \tilde{f} \circ \epsilon_{*,*}^{-1} \in \mathrm{Hom}(A, \int_{\coprod_{*}I}M_i d\mu) $
Here $\tilde{f}$ is the map from $\int_* A d*$ to $\int_* \int_I M_i d\mu d*$ corresponding to $f$ by the functor $\int_* d*$ (in other words $\tilde{f} =\int_*fd*$).

So the first axiom can be deduced from the fact that the following diagram commutes by naturality of $\epsilon_{*,*}$:

% https://q.uiver.app/#q=WzAsNCxbMCwwLCJBIl0sWzAsMywiXFxpbnRfKkFkKiJdLFs0LDMsIlxcaW50XypcXGludF9JTV9pZFxcbXUgZCoiXSxbNCwwLCJcXGludF9JTV9pIGRcXG11Il0sWzAsMSwiXFxlcHNpbG9uX3sqLCp9Il0sWzEsMiwiXFx0aWxkZXtmfSJdLFswLDMsImYiXSxbMywyLCJcXERlbHRhX3tmX0ksKn1eey0xfSA9XFxlcHNpbG9uX3sqLCp9Il1d
\[\begin{tikzcd}[ampersand replacement=\&]
	A \&\&\&\& {\int_IM_i d\mu} \\
	\\
	\\
	{\int_*Ad*} \&\&\&\& {\int_*\int_IM_id\mu d*}
	\arrow["f", from=1-1, to=1-5]
	\arrow["{\epsilon_{*,*}}", from=1-1, to=4-1]
	\arrow["{\Delta_{f_I,*}^{-1} =\epsilon_{*,*}}", from=1-5, to=4-5]
	\arrow["{\tilde{f}=\int_*fd*}", from=4-1, to=4-5]
\end{tikzcd}\]

Now for axiom $5$ we want to show that the following composition: 

\[
% https://q.uiver.app/#q=WzAsNSxbMCwzLCJcXG1hdGhybXtIb219KEEsXFxpbnRfSSBNX2kgZFxcbXUpXFx0aW1lcyBcXGludF97SX1cXG1hdGhiZnsxfWRcXG11Il0sWzAsNSwiXFxtYXRocm17SG9tfShBLFxcaW50X0kgTV9pIGRcXG11IClcXHRpbWVzIFxcaW50X0lcXG1hdGhybXtIb219KE1fe2l9ICwgXFxpbnRfeyp9TV9pZCopZFxcbXUiXSxbMCwyLCJcXG1hdGhybXtIb219KEEsXFxpbnRfSSBNX2kgZFxcbXUpIFxcdGltZXMgMSJdLFswLDAsIlxcbWF0aHJte0hvbX0oQSxcXGludF9JIE1faSBkXFxtdSkiXSxbMCw3LCJcXG1hdGhybXtIb219KEEsXFxpbnRfe1xcY29wcm9kX3tJfSp9TV9pZFxcbXUpIl0sWzIsMCwiXFxtYXRocm17aWR9IFxcdGltZXMgXFxkZWx0YV97XFxtdX0iXSxbMywyLCJcXHNpbWVxIl0sWzAsMSwiXFxtYXRocm17aWR9IFxcdGltZXMgXFxpbnRfSSBcXGthcHBhIGRcXG11Il0sWzEsNCwiXFxiZXRhIl1d
\begin{tikzcd}
	{\mathrm{Hom}(A,\int_I M_i d\mu)} \\
	\\
	{\mathrm{Hom}(A,\int_I M_i d\mu) \times 1} \\
	{\mathrm{Hom}(A,\int_I M_i d\mu)\times \int_{I}\mathbf{1}d\mu} \\
	\\
	{\mathrm{Hom}(A,\int_I M_i d\mu )\times \int_I\mathrm{Hom}(M_{i} , \int_{*}M_id*)d\mu} \\
	\\
	{\mathrm{Hom}(A,\int_{\coprod_{I}*}M_id\mu)}
	\arrow["\simeq", from=1-1, to=3-1]
	\arrow["{\mathrm{id} \times \delta_{\mu}}", from=3-1, to=4-1]
	\arrow["{\mathrm{id} \times \int_I \kappa d\mu}", from=4-1, to=6-1]
	\arrow["\beta", from=6-1, to=8-1]
\end{tikzcd}
\]
is the inverse of $\Xi_{g_{I},\mu}$ where $g_{I}$ is the natural isomorphism between $I$ and $\coprod_{I}*$.

To do so again, let $f \in \mathrm{Hom}(A, \int_I M_id\mu) $ then the image of $f$ by the series of morphism is $f \mapsto (f , 1) \mapsto (f, (1)) \mapsto (f, (\epsilon_{*,*})) \mapsto (f , ((\epsilon_{*,*}))) \mapsto \int _I \epsilon^{-1}_{*,*}d\mu \circ f$, but on the other hand $\kappa_{g_I}$ is the composition with $\Delta_{g_I,\mu}$, so the fact that the composition above is the inverse of $\Xi_{g_{I},\mu}$ can be deduced from the commutativity of the following diagram:
% https://q.uiver.app/#q=WzAsNSxbMCwwLCJcXGludF9JIE1faSBkXFxtdSJdLFswLDMsIlxcaW50X3tcXGNvcHJvZF9JKn1NX2kgZCBcXHRpbGRle1xcbXV9Il0sWzMsNV0sWzMsMCwiXFxpbnRfe0l9XFxpbnRfKk1faSBkKiBkXFxtdSJdLFszLDMsIlxcaW50X0kgXFxpbnRfe1xcY29wcm9kX3tJfSp9TV9pIGQgXFxkZWx0YV9pIGRcXG11Il0sWzAsMSwiXFxEZWx0YV97Z157LTF9X0ksXFxtdX0iLDJdLFswLDMsIlxcaW50X0kgXFxlcHNpbG9uX3sqLCp9IGRcXG11ID1cXERlbHRhX3tcXG11LCBcXGRlbHRhX3tcXGJ1bGxldH19Il0sWzMsNCwiXFxpbnRfSSBcXERlbHRhX3tpIFxcbWFwc3RvICosIFxcZGVsdGFfaSB9IGRcXG11Il0sWzEsNCwiXFxEZWx0YV97XFxtdSxcXGRlbHRhX3tcXGJ1bGxldH19Il1d
\[\begin{tikzcd}
	{\int_I M_i d\mu} &&& {\int_{I}\int_*M_i d* d\mu} \\
	\\
	\\
	{\int_{\coprod_I*}M_i d \tilde{\mu}} &&& {\int_I \int_{\coprod_{I}*}M_i d \delta_i d\mu} \\
	\\
	&&& {}
	\arrow["{\int_I \epsilon_{*,*} d\mu =\Delta_{\mu, \delta_{\bullet}}}", from=1-1, to=1-4]
	\arrow["{\Delta_{g^{-1}_I,\mu}}"', from=1-1, to=4-1]
	\arrow["{\int_I \Delta_{i \mapsto *, \delta_i } d\mu}", from=1-4, to=4-4]
	\arrow["{\Delta_{\mu,\delta_{\bullet}}}", from=4-1, to=4-4]
\end{tikzcd}\]

First, we should note that the fact that $\int_I \epsilon_{*,*} d\mu =\Delta_{\mu, \delta_{\bullet}}$ follows from Corollary 1.3.6 in Lurie \cite{lurie2018ultracategories}.

Now for the fact that the diagram above is commutative we refer the reader to section $6.1$ of our other paper \cite{hamad2025ultracategories}.

Now for axiom $6$, which ensures compatibility between $\beta$ and $\Xi$. Starting with the following diagram: 
\[
\adjustbox{max width = \textwidth}{
% https://q.uiver.app/#q=WzAsNixbMCwwLCJcXEhvbShBLFxcaW50X0lNX2lkXFxtdSkgXFx0aW1lcyBcXGludF9JIFxcSG9tKE1fe2l9LFxcaW50X3tYX2l9Tl97KGkseCl9ZGZfaVxcbGFtYmRhX2kpIGRcXG11Il0sWzQsMCwiXFxIb20oQSxcXGludF97XFxjb3Byb2Rfe0l9IFhfaX1OX3soaSx4KX1kXFxpbnRcXGlvdGFfaWZfaVxcbGFtYmRhX2lkXFxtdSkiXSxbMCwzLCJcXEhvbShBLFxcaW50X0lNX2kgZFxcbXUpIFxcdGltZXMgXFxpbnRfSVxcSG9tKE1fe2l9LFxcaW50X3tLX2l9Tl97KGksZl9pKHgpKX1kXFxsYW1iZGFfaSlkXFxtdSJdLFs0LDMsIlxcSG9tKEEsXFxpbnRfe1xcY29wcm9kX3tJfSBLX2l9Tl97KGksZl9pKHgpKX1kXFxpbnRcXGlvdGFfaVxcbGFtYmRhX2lkXFxtdSkiXSxbMCwxXSxbMSwxXSxbMCwxLCJcXGJldGEiXSxbMiwzLCJcXGJldGEiLDJdLFsxLDMsIlxcWGlfe1xcYmFye2Z9fSJdLFswLDIsIlxcbGF7aWR9IFxcdGltZXMgXFxpbnRcXFhpX3tmX2ksXFxsYW1iZGFfaX0iLDJdXQ==
\begin{tikzcd}
	{\Hom(A,\int_IM_id\mu) \times \int_I \Hom(M_{i},\int_{X_i}N_{(i,x)}df_i\lambda_i) d\mu} &&&& {\Hom(A,\int_{\coprod_{I} X_i}N_{(i,x)}d\int\iota_if_i\lambda_id\mu)} \\
	{} & {} \\
	\\
	{\Hom(A,\int_IM_i d\mu) \times \int_I\Hom(M_{i},\int_{K_i}N_{(i,f_i(x))}d\lambda_i)d\mu} &&&& {\Hom(A,\int_{\coprod_{I} K_i}N_{(i,f_i(x))}d\int\iota_i\lambda_id\mu)}
	\arrow["\beta", from=1-1, to=1-5]
	\arrow["{\la{id} \times \int\Xi_{f_i,\lambda_i}}"', from=1-1, to=4-1]
	\arrow["{\Xi_{\bar{f}}}", from=1-5, to=4-5]
	\arrow["\beta"', from=4-1, to=4-5]
\end{tikzcd}
}
\]

One can see easily that the task reduces to showing that the following diagram is commutative:

% https://q.uiver.app/#q=WzAsNCxbMCwwLCJcXGludF9JIFxcaW50X3tYX2l9Tl97KGkseCl9ZGZfaVxcbGFtYmRhX2kgZFxcbXUiXSxbNCwwLCJcXGludF97XFxjb3Byb2RfSSBYX2l9IE5feyhpLHgpfWQgXFxpbnRfe0l9Zl9pIFxcbGFtYmRhX2kgZFxcbXUiXSxbNCw0LCJcXGludF97XFxjb3Byb2RfSSBLX2l9IE5feyhpLHgpfWQgXFxpbnRfe0l9IFxcbGFtYmRhX2kgZFxcbXUiXSxbMCw0LCJcXGludF9JIFxcaW50X3tLX2l9Tl97KGkseCl9ZFxcbGFtYmRhX2kgZFxcbXUiXSxbMSwwLCJhIiwyXSxbMSwyLCJcXERlbHRhX3tcXGJhcntmfSwgXFxpbnRfSSBcXGxhbWJkYV9pIGRcXG11fSJdLFswLDMsIlxcaW50X0kgXFxEZWx0YV97Zl9pLFxcbGFtYmRhX2l9IGRcXG11IiwyXSxbMiwzLCJhIl1d
\[\begin{tikzcd}[ampersand replacement=\&]
	{\int_I \int_{X_i}N_{(i,x)}df_i\lambda_i d\mu} \&\&\&\& {\int_{\coprod_I X_i} N_{(i,x)}d \int_{I}f_i \lambda_i d\mu} \\
	\\
	\\
	\\
	{\int_I \int_{K_i}N_{(i,x)}d\lambda_i d\mu} \&\&\&\& {\int_{\coprod_I K_i} N_{(i,x)}d \int_{I} \lambda_i d\mu}
	\arrow["{\int_I \Delta_{f_i,\lambda_i} d\mu}"', from=1-1, to=5-1]
	\arrow["a"', from=1-5, to=1-1]
	\arrow["{\Delta_{\bar{f}, \int_I \lambda_i d\mu}}", from=1-5, to=5-5]
	\arrow["a", from=5-5, to=5-1]
\end{tikzcd}\]

To show this we refer the reader to the section on the naturality of the colax associator in \cite{hamad2025ultracategories}.

Now we wish to show the following diagram is commutative:

\[
\adjustbox{scale=0.58}{

% https://q.uiver.app/#q=WzAsNSxbMCwwLCJcXEhvbShBLCBcXGludF9JTV9pIGRmXFxtdSkgXFx0aW1lcyBcXGludF9JIFxcSG9tKE1fe2l9LFxcaW50X3tYX2l9Tl97KGksayl9ZFxcbGFtYmRhX3tpfSlkZlxcbXUiXSxbMCwzLCJcXEhvbShBLCBcXGludF9JTV9pIGRmXFxtdSkgXFx0aW1lcyBcXGludF9KIFxcSG9tKE1fe2Yoail9LFxcaW50X3tYX3tmKGopfX1OX3soZihqKSxrKX1kXFxsYW1iZGFfe2Yoail9ZFxcbXUiXSxbMywzLCJcXEhvbShBLCBcXGludF9KTV97ZihqKX0gZFxcbXUpIFxcdGltZXMgXFxpbnRfSiBcXEhvbShNX3tmKGopfSxcXGludF97WF97ZihqKX19Tl97KGYoaiksayl9ZFxcbGFtYmRhX3tmKGopfWRcXG11Il0sWzYsMywiXFxIb20oQSxcXGludF97XFxjb3Byb2Rfe0p9WF97ZihqKX19Tl97KGYoaikseCl9ZFxcaW50X3tKfVxcaW90YV9qXFxsYW1iZGFfe1xcbGFtYmRhX3tmKGopfX1kXFxtdSkiXSxbNiwwLCJcXEhvbShBLFxcaW50X3tcXGNvcHJvZF97SX1YX2l9Tl97KGkseCl9ZFxcaW50X3tJfVxcaW90YV9pXFxsYW1iZGFfaWRmXFxtdSkiXSxbMCwxLCJcXGxhe2lkfSBcXHRpbWVzIFxcRGVsdGFfe2YsXFxtdX0iXSxbMSwyLCJcXFhpX3tmfSBcXHRpbWVzIFxcbGF7aWR9Il0sWzIsMywiXFxiZXRhIl0sWzQsMywiXFxYaV97XFx0aWxkZXtmfX0iLDJdLFswLDQsIlxcYmV0YSJdXQ==
\begin{tikzcd}
	{\Hom(A, \int_IM_i df\mu) \times \int_I \Hom(M_{i},\int_{X_i}N_{(i,k)}d\lambda_{i})df\mu} &&&&&& {\Hom(A,\int_{\coprod_{I}X_i}N_{(i,x)}d\int_{I}\iota_i\lambda_idf\mu)} \\
	\\
	\\
	{\Hom(A, \int_IM_i df\mu) \times \int_J \Hom(M_{f(j)},\int_{X_{f(j)}}N_{(f(j),k)}d\lambda_{f(j)})d\mu} &&& {\Hom(A, \int_JM_{f(j)} d\mu) \times \int_J \Hom(M_{f(j)},\int_{X_{f(j)}}N_{(f(j),k)}d\lambda_{f(j)}d\mu} &&& {\Hom(A,\int_{\coprod_{J}X_{f(j)}}N_{(f(j),x)}d\int_{J}\iota_j\lambda_{\lambda_{f(j)}}d\mu)}
	\arrow["\beta", from=1-1, to=1-7]
	\arrow["{\la{id} \times \Delta_{f,\mu}}", from=1-1, to=4-1]
	\arrow["{\Xi_{\tilde{f}}}"', from=1-7, to=4-7]
	\arrow["{\Xi_{f} \times \la{id}}", from=4-1, to=4-4]
	\arrow["\beta", from=4-4, to=4-7]
\end{tikzcd}
}
\]
The commutativity of this diagram comes down to the commutativity of the following diagram:

% https://q.uiver.app/#q=WzAsNCxbMCwwLCJcXGludF97SX1cXGludF97WF9pfU5feyhpLHgpfSBkIFxcbGFtYmRhX2lkZlxcbXUiXSxbMCwzLCJcXGludF9KXFxpbnRfe1hfe2Yoail9fU5feyhmKGopLHgpfWQgXFxsYW1iZGFfe2Yoail9ZFxcbXUiXSxbNSwzLCJcXGludF97XFxjb3Byb2Rfe0p9WF97ZihqKX19Tl97ZihqKSx4fWRcXGludF97Sn1cXGlvdGFfaiBcXGxhbWJkYV97ZihqKX1kXFxtdSJdLFs1LDAsIlxcaW50X3tcXGNvcHJvZF97SX1YX3tpfX1OX3tpLHh9ZFxcaW50X3tJfVxcaW90YV9pIFxcbGFtYmRhX3tpfWRmXFxtdSJdLFswLDEsIlxcRGVsdGFfe2YsXFxtdX0iXSxbMiwxLCJhIiwyXSxbMywwLCJhIl0sWzMsMiwiXFxEZWx0YV97XFxiYXJ7Zn0sIFxcaW50X0kgXFxsYW1iZGFfaSBkXFxtdX0iLDJdXQ==
\[\begin{tikzcd}[ampersand replacement=\&]
	{\int_{I}\int_{X_i}N_{(i,x)} d \lambda_idf\mu} \&\&\&\&\& {\int_{\coprod_{I}X_{i}}N_{i,x}d\int_{I}\iota_i \lambda_{i}df\mu} \\
	\\
	\\
	{\int_J\int_{X_{f(j)}}N_{(f(j),x)}d \lambda_{f(j)}d\mu} \&\&\&\&\& {\int_{\coprod_{J}X_{f(j)}}N_{f(j),x}d\int_{J}\iota_j \lambda_{f(j)}d\mu}
	\arrow["{\Delta_{f,\mu}}", from=1-1, to=4-1]
	\arrow["a", from=1-6, to=1-1]
	\arrow["{\Delta_{\tilde{f}, \int_I \lambda_i d\mu}}"', from=1-6, to=4-6]
	\arrow["a"', from=4-6, to=4-1]
\end{tikzcd}\]

Which again can be deduced from the naturality of the colax associator $a$.

Finally we need to check the commutativity of the following diagram, in the context of axiom $7$:

\[
\adjustbox{scale=0.68}{
\begin{tikzcd}
	& {} \\
	\\
	& {} \\
	{\mathrm{Hom}(A,\int_IM_id\mu) \times \int_{I}\mathrm{Hom}(M_{i},\int_{X_{i}}N_{(i,x)}d\lambda_{i}) \times\int_{X_i}\mathrm{Hom}(N_{(i,x)}, \int_{T_{(i,x)}}L_{(i,x,t)}d\omega_{(i,x)} d\mu} & {} & {\mathrm{Hom}(A,\int_IM_id\mu)\times \int_I\mathrm{Hom}(M_{i}, \int_{\coprod T_{(i,x)}}L_{(i,x,t)}d \int_{X_i}\iota_x \omega_{(i,x)}d\lambda_i)d\mu } \\
	\\
	\\
	{\mathrm{Hom}(A,\int_IM_id\mu) \times \int_I\mathrm{Hom}(M_{i},\int_{X_i}N_{(i,x)}d\lambda_{i}) \times\int_I \int_{X_i}\mathrm{Hom}(N_{(i,x)}, \int_{T_{(i,x)}}L_{(i,x,t)}d\omega_{(i,x)}d\lambda_{i} d\mu} \\
	&& {\mathrm{Hom}(A, \int_{\coprod_{I}\coprod_{X_i}T_{(i,x)}}L_{(i,x,t)}d\int_{}\omega_{(i,x)}d\int_I\iota_i\lambda_id\mu)} \\
	\\
	{ \mathrm{Hom}(A,\int_{\coprod_{i}X_{i}}N_{(i,x)} d\int_{I}\iota_i\lambda_id\mu)\times \int_I \int_{X_i}\mathrm{Hom}(N_{(i,x)}, \int_{T_{(i,x)}}L_{(i,x,t)}d\omega_{(i,x)}d\lambda_i d\mu} \\
	\\
	{\mathrm{Hom}(A,\int_{\coprod_{I}X_{i}}N_{i,x} dh\int_{I}\iota_i\lambda_id\mu)\times\int_{\coprod_{I}X_i}\mathrm{Hom}(N_{(i,x)},\int_{T_{(i,x)}}L_{(i,x,t)}d\omega_{(i,x)})d\int_{I}\iota_i\lambda_id\mu} & {} & {\mathrm{Hom}(A, \int_{\coprod_{\coprod_{I}X_i}T_{(i,x)}}L_{(i,x,t)}d\int_{}\omega_{(i,x)}d\int_I\iota_i\lambda_id\mu)}
	\arrow["{\mathrm{id} \times \int\beta}"', from=4-1, to=4-3]
	\arrow["\beta"', from=4-3, to=8-3]
	\arrow["{\mathrm{commutativity \ of \  products \  with \ directed \  colimits }}"', from=7-1, to=4-1]
	\arrow["{\beta \times \mathrm{id}}", from=7-1, to=10-1]
	\arrow["{\Xi_{\theta}}"', from=8-3, to=12-3]
	\arrow["{\mathrm{id} \times  a^{-1}}", from=10-1, to=12-1]
	\arrow["\beta", from=12-1, to=12-3]
\end{tikzcd}
}
\]

Suppose that we have $(f, (g_i), h_{(i,x)}) \in \mathrm{Hom}(A,\int_IM_id\mu) \times \int_{I}(\mathrm{Hom}(M_{i},\int_{X_{i}}N_{(i,x)}d\lambda_{i})\times$ \\ $\int_{X_i}\mathrm{Hom}(N_{i,x}, \int_{T_{i,x}}L_{(i,x,t)}d\omega_{(i,x)})d\lambda_i) d\mu $, its image by the direction right down is

$$\Delta^{-1} \circ  a^{-1} \circ \int_{I} a^{-1} \circ (\int_{X_i}(((h_{(i,x)}) d\lambda_i) \circ (g_i)) d\mu   \circ f$$

$$= \Delta^{-1} \circ  a^{-1} \circ \int_I a^{-1} d\mu \circ (\int_I \int_{X_i}((h_{(i,x)}) d\lambda_i) d\mu \circ \int_I g_i d\mu   \circ f $$
While the image of $(f, (g_i), h_{(i,x)})$ by the direction down right is 

$$a^{-1} \circ (\int_{\coprod_{I}X_i} (h_{(i,x)}) d\int_{I} \iota_i \lambda_i d\mu)  \circ  a^{-1} \circ \int_I g_i d\mu \circ f $$

$$ = a^{-1} \circ a^{-1} \circ (\int_I \int_{X_i}((h_{(i,x)}) d\lambda_i) d\mu \circ \int_I g_i d \mu \circ f \ \  \mathbf{by \ naturality \ of \ } a $$

Here, we are omitting to write the subscripts of $a$ and $\Delta$ to keep the notation concise.

So it's enough to show that $$\Delta^{-1} \circ  a^{-1} \circ \int_I a^{-1} d\mu= a^{-1} \circ a^{-1} $$

Setting  things up, and writing the definitions of $a$ and the ultraproduct diagonal map $\Delta$ we get the following diagram:

\[
\adjustbox{max width = \textwidth}{
\begin{tikzcd}[ampersand replacement=\&]
	 \\
	{\int_{\coprod_{\coprod_{I}X_i}T_{(i,x)}}L d \rho} \&\&\& {\int_{\coprod_{I}X_i} \int_{\coprod_{\coprod_{I}X_i}T_{(i,x)}} Ld \iota_{(i,x)} \omega_{(i,x)} d \int_I \iota_i \lambda_i d\mu} \\
	\\
	\\
	\\
	\\
	{\int_{\coprod_{I}\coprod_{X_i}T_{(i,x)}}L d \rho} \&\&\&\&\&\&\&\&\& {\int_{\coprod_I X_i} \int_{T_{(i,x)}}L  d\omega_{x}d \int_I \iota_i\lambda_i d \mu} \\
	\\
	\&\&\&\&\&\&\&\&\& {\int_{I} \int_{\coprod_{X_i}T_{(i,x)}} \int_{T_{(i,x)}} Ld \omega_x d \iota_i \lambda_i d \mu} \\
	{\int_{I} \int_{\coprod_{I}\coprod_{X_i}T_{(i,x)}}Ld \iota_i \int_{X_i} \iota_{x}\omega_{(i,x)}d \lambda_{i}d\mu} \\
	\\
	{\int_I \int_{\coprod_{X_i}T_{(i,x)}}L d \int_{X_i} \iota_x \omega_{(i,x)} d \lambda_i d\mu} \&\&\& {\int_I \int_{X_i} \int_{\coprod_{X_i}T_{(i,x)}}Ld \iota_{x} \omega_{(i,x)} d \lambda_i d \mu} \&\&\&\&\&\& {\int_I \int_{X_i} \int_{T_{(i,x)}} L d\omega_{(i,x)} d \lambda_i d\mu}
	\arrow["\Delta", from=2-1, to=2-4]
	\arrow["\Delta"', from=2-1, to=7-1]
	\arrow["{{{\int_{\coprod_{I}X_i} \Delta_{\lambda_i,\iota_{(i,x)}}}}}"{description}, from=2-4, to=7-10]
	\arrow["{{{\Delta_{\mu,\iota_{\bullet}\int_{X_{\bullet}\iota_xd\lambda_{\bullet}}}}}}", from=7-1, to=10-1]
	\arrow["{{{\Delta_{\mu,\iota\lambda_{\bullet}}}}}"', from=7-10, to=9-10]
	\arrow["{{{\int_I \Delta_{\lambda_i,\iota_i}d\mu}}}"', from=9-10, to=12-10]
	\arrow["{{{\int_I \Delta_{\int_X \iota_x d \lambda_i,\iota_i}d\mu}}}"', from=10-1, to=12-1]
	\arrow["{{{\int_I \Delta_{\lambda_i,\int_X \iota_x \omega_{\bullet} }d\mu}}}"', from=12-1, to=12-4]
	\arrow["{{{\int_I \int_{X} \Delta_{\omega_{(i,x)},\iota_x}d\lambda_id\mu}}}"', from=12-4, to=12-10]
\end{tikzcd}
}
\]

Here, $\rho$ is equal to $\int_I   \iota_i (\int_{X_i} \iota_x \omega_{(i,x)} d \lambda_i) d \mu=\int_I    (\int_{\coprod_{I} X_i}\iota_i\iota_x \omega_{(i,x)} d \iota_i \lambda_i) d\mu=\int_{\coprod_{I}X_i}\iota_i \iota_x \omega_{(i,x)}d \int_I  \iota_i \lambda_i d\mu$.

Now to show that this diagram is commutative, we can look at the following diagram:

\adjustbox{max width = \textwidth}{
\begin{tikzcd}[ampersand replacement=\&]
	{} \\
	{\int_{\coprod_{\coprod_{I}X_i}T_{(i,x)}}L d \rho} \&\&\& {\int_{\coprod_{I}X_i} \int_{\coprod_{\coprod_{I}X_i}T_{(i,x)}} Ld \iota_{(i,x)} \omega_{(i,x)} d \int_I \iota_i \lambda_i d\mu} \\
	\\
	\\
	\\
	\\
	{\int_{\coprod_{I}\coprod_{X_i}T_{(i,x)}}L d \rho} \&\&\& {\int_{\coprod_{I}X_i} \int_{\coprod_{I}\coprod_{X_i}T_{(i,x)}} Ld \iota_{(i,x)} \omega_{(i,x)} d \int_I \iota_i \lambda_i d\mu} \&\&\&\&\&\& {\int_{\coprod_I X_i} \int_{T_{(i,x)}}L  d\omega_{x}d \int_I \iota_i\lambda_i d \mu} \\
	\\
	\&\&\& {\int_I(\int_{\coprod_{I}X_i}(\int_{\coprod_{I}\coprod_{X_i}T_{(i,x)}} Ld \iota_i\iota_x\omega_{(i,x)}) d  \iota_i \lambda_i) d \mu} \&\&\&\&\&\& {\int_{I} \int_{\coprod_{X_i}T_{(i,x)}} \int_{T_{(i,x)}} Ld \omega_x d \iota_i \lambda_i d \mu} \\
	{\int_{I} \int_{\coprod_{I}\coprod_{X_i}T_{(i,x)}}Ld \iota_i \int_{X_i} \iota_{x}\omega_{(i,x)}d \lambda_{i}d\mu} \&\&\& {\int_I\int_{X_i} \int_{\coprod_{I}\coprod_{X_i}T_{(i,x)}} Ld \iota_i \iota_x\omega_{(i,x)}d\lambda_i d \mu} \\
	\\
	{\int_I \int_{\coprod_{X_i}T_{(i,x)}}L d \int_{X_i} \iota_x \omega_{(i,x)} d \lambda_i d\mu} \&\&\& {\int_I \int_{X_i} \int_{\coprod_{X_i}T_{(i,x)}}Ld \iota_{x} \omega_{(i,x)} d \lambda_i d \mu} \&\&\&\&\&\& {\int_I \int_{X_i} \int_{T_{(i,x)}} L d\omega_{(i,x)} d \lambda_i d\mu}
	\arrow["{{{\Delta_{\int_I\iota_i \lambda_i d\mu,\iota \omega_{\bullet}}}}}", from=2-1, to=2-4]
	\arrow["\Delta"', from=2-1, to=7-1]
	\arrow["7", draw=none, from=2-1, to=7-4]
	\arrow["{\int_{\coprod_{I}X_i}{\Delta_{}}d \int_I \iota_i \lambda_i d\mu}"', from=2-4, to=7-4]
	\arrow[""{name=0, anchor=center, inner sep=0}, "{{{\int_{\coprod_{I}X_i} \Delta_{\lambda_i,\iota_{(i,x)}}}}}"{description}, from=2-4, to=7-10]
	\arrow["{{{\Delta_{\int_I\iota_i \lambda_i d\mu,\iota \omega_{\bullet}}}}}", from=7-1, to=7-4]
	\arrow["1", draw=none, from=7-1, to=9-4]
	\arrow["{{{\Delta_{\mu,\iota_{\bullet}\int_{X_{\bullet}\iota_xd\lambda_{\bullet}}}}}}", from=7-1, to=10-1]
	\arrow["{{{\int_{\coprod_{I}X_i} \Delta_{\lambda_i,\iota_{(i,x)}}}}}"', from=7-4, to=7-10]
	\arrow["{{{\Delta_{\mu,\iota\lambda_{\bullet}}}}}", from=7-4, to=9-4]
	\arrow["3", draw=none, from=7-4, to=9-10]
	\arrow["{{{\Delta_{\mu,\iota\lambda_{\bullet}}}}}"', from=7-10, to=9-10]
	\arrow["{{{\int_{I}\int_{\coprod_{X_i}T_{(i,x)}}\Delta_{\omega_{(i,x)},\iota_i \iota_x} d \int_I\iota_i \lambda_i d \mu}}}"', from=9-4, to=9-10]
	\arrow["{{{\int_I \Delta_{\lambda_i,\iota_i}d \mu}}}", from=9-4, to=10-4]
	\arrow[""{name=1, anchor=center, inner sep=0}, "{{{\int_I \Delta_{\lambda_i,\iota_i}d\mu}}}"', from=9-10, to=12-10]
	\arrow["{{{\int_I\Delta_{\iota_i\lambda_i,\iota\omega_{\bullet}}}}}", from=10-1, to=9-4]
	\arrow[""{name=2, anchor=center, inner sep=0}, "{{{\int_I\Delta_{\lambda_i,\iota\omega_{\bullet}}}}}", from=10-1, to=10-4]
	\arrow["{{{\int_I \Delta_{\int_X \iota_x d \lambda_i,\iota_i}d\mu}}}"', from=10-1, to=12-1]
	\arrow["5", draw=none, from=10-1, to=12-4]
	\arrow["{{{\int_I \int_X \Delta_{\iota_x\omega_{(x,i)},\iota_i}}}}", from=10-4, to=12-4]
	\arrow["{{{\int_{I}\int_X\Delta_{\omega_{(i,x)},\iota_{(i,x)}}d\lambda_id\mu}}}", from=10-4, to=12-10]
	\arrow["{{{\int_I \Delta_{\lambda_i,\int_X \iota_x \omega_{\bullet} }d\mu}}}"', from=12-1, to=12-4]
	\arrow[""{name=3, anchor=center, inner sep=0}, "{{{\int_I \int_{X} \Delta_{\omega_{(i,x)},\iota_x}d\lambda_id\mu}}}"', from=12-4, to=12-10]
	\arrow["8"', draw=none, from=7-4, to=0]
	\arrow["2"{description}, draw=none, from=9-4, to=2]
	\arrow["4"{description}, draw=none, from=9-4, to=1]
	\arrow["6"{marking, allow upside down}, draw=none, from=10-4, to=3]
\end{tikzcd}
}

We have shown that diagrams $1$ through $6$ are commutative in \cite{hamad2025ultracategories} (see pages $18\dash19$ of version $4$ of the arXiv version, this is the commutativity of  the large diagram $2$ which is shown using large diagram $3$), so the only thing remaining is to check the commutativity of diagrams $7$ and $8$. Now, it's easy to see that diagram $8$ commutes by Lurie \cite[1.3.7]{lurie2018ultracategories}.

So the only thing remaining is checking diagram $7$  commutes (before doing this, it is very natural to expect such diagram to commute; it's just telling us that the categorical Fubini transform is ``invariant'' under isomorphism of sets with the ultraproduct diagonal maps being the translation maps). An easy way to see why this diagram must commute is to notice that categorical Fubini transform is a composition between the colax associator and the algebra functor, and both these things satisfy the required naturality, in other words  suppose that we have a morphism $f$ from $I$ to $I^{'}$, we wish to show that this diagram commutes:

% https://q.uiver.app/#q=WzAsNCxbMCwwLCJcXGludF9JIE1faSBkXFxpbnRfU2YgXFxudV9TIGRcXG11ID0gXFxpbnRfe0l9TV97aX1kZiAgXFxpbnRfe1N9XFxudV9zIGRcXG11Il0sWzQsMCwiXFxpbnRfU1xcaW50X3tJfU1fe2l9ZCBmXFxudV9zIGRcXG11Il0sWzAsMywiXFxpbnRfe0leeyd9fSBNX3tpXnsnfX0gZFxcaW50X1MgXFxudV9TIGRcXG11Il0sWzQsMywiXFxpbnRfU1xcaW50X3tJXnsnfX1NX3tpXnsnfX0gZFxcbnVfcyAgZCBcXG11Il0sWzAsMSwiXFxEZWx0YV97fSJdLFswLDIsIlxcRGVsdGFfe2YsIFxcaW50X1NcXG51X3NkXFxtdX0iLDJdLFsyLDMsIlxcRGVsdGEiLDJdLFsxLDMsIlxcaW50X3tTfSBcXERlbHRhX3t9IGRcXG11Il1d
\[\begin{tikzcd}
	{\int_I M_i d\int_Sf \nu_S d\mu = \int_{I}M_{i}df  \int_{S}\nu_s d\mu} &&&& {\int_S\int_{I}M_{i}d f\nu_s d\mu} \\
	\\
	\\
	{\int_{I^{'}} M_{i^{'}} d\int_S \nu_S d\mu} &&&& {\int_S\int_{I^{'}}M_{i^{'}} d\nu_s  d \mu}
	\arrow["{\Delta_{}}", from=1-1, to=1-5]
	\arrow["{\Delta_{f, \int_S\nu_sd\mu}}"', from=1-1, to=4-1]
	\arrow["{\int_{S} \Delta_{} d\mu}", from=1-5, to=4-5]
	\arrow["\Delta"', from=4-1, to=4-5]
\end{tikzcd}\]

Now we can look at the following diagram:

\[
\adjustbox{max width =\textwidth}{
% https://q.uiver.app/#q=WzAsNixbMCwwLCJcXGludF9JIE1faSBkXFxpbnRfU2YgXFxudV9TIGRcXG11ID0gXFxpbnRfe0l9TV97aX1kZiAgXFxpbnRfe1N9XFxudV9zIGRcXG11Il0sWzAsMywiXFxpbnRfe0leeyd9fSBNX3tmKGkpfSBkXFxpbnRfUyBcXG51X1MgZFxcbXUiXSxbNiwwLCJcXGludF9TXFxpbnRfe0l9TV97aX1kIGZcXG51X3MgZFxcbXUiXSxbNiwzLCJcXGludF9TXFxpbnRfe0leeyd9fU1fe2leeyd9fSBkXFxudV9zICBkIFxcbXUiXSxbMywwLCJcXGludF97UyBcXHRpbWVzIEl9TV57J31faWQgZl57J30gXFxpbnRfUyBcXG51X3NkXFxtdT0gXFxpbnRfe1MgXFx0aW1lcyBJfU1eeyd9X2kgZCBcXGludF9TIFxcaW90YV9zIFxcbnVfcyBkZlxcbXUiXSxbMywzLCJcXGludF97UyBcXHRpbWVzIEl7J319IGQgTV57J31fe2YoaSl9ZCBcXGludF9TXFxpb3RhX3MgXFxudV9zIGRcXG11Il0sWzAsMSwiXFxEZWx0YV97ZiwgXFxpbnRfU1xcbnVfc2RcXG11fSIsMl0sWzIsMywiXFxpbnRfe1N9IFxcRGVsdGFfe30gZFxcbXUiXSxbMCw0LCJtIl0sWzEsNSwibSIsMl0sWzUsMywiYSIsMl0sWzQsMiwiYSJdLFs0LDUsIlxcRGVsdGFfe2Zeeyd9LFxcaW50X1MgXFxpb3RhX3MgXFxudV9zIGRcXG11fSIsMV1d
\begin{tikzcd}
	{\int_I M_i d\int_Sf \nu_S d\mu = \int_{I}M_{i}df  \int_{S}\nu_s d\mu} &&& {\int_{S \times I}M^{'}_id f^{'} \int_S \nu_sd\mu= \int_{S \times I}M^{'}_i d \int_S \iota_s \nu_s df\mu} &&& {\int_S\int_{I}M_{i}d f\nu_s d\mu} \\
	\\
	\\
	{\int_{I^{'}} M_{f(i)} d\int_S \nu_S d\mu} &&& {\int_{S \times I'} d M^{'}_{f(i)}d \int_S\iota_s \nu_s d\mu} &&& {\int_S\int_{I^{'}}M_{i^{'}} d\nu_s  d \mu}
	\arrow["m", from=1-1, to=1-4]
	\arrow["{\Delta_{f, \int_S\nu_sd\mu}}"', from=1-1, to=4-1]
	\arrow["a", from=1-4, to=1-7]
	\arrow["{\Delta_{f^{'},\int_S \iota_s \nu_s d\mu}}"{description}, from=1-4, to=4-4]
	\arrow["{\int_{S} \Delta_{} d\mu}", from=1-7, to=4-7]
	\arrow["m"', from=4-1, to=4-4]
	\arrow["a"', from=4-4, to=4-7]
\end{tikzcd}
}
\]

Here $f^{'}$ is defined by $f^{'}(s,i)=(s,f(i))$, and $M^{'}_{(s,i)}=M_{i}$ for every $(s,i)$. Now the left diagram commutes by functoriality of $m$ (the algebra functor), more precisely we have that the following diagram is commutative in the category $TA$:

% https://q.uiver.app/#q=WzAsNCxbMCwwLCIoKE1faSksSSwgZlxcaW50X1MgXFxudV9zIGRcXG11KSJdLFswLDIsIigoTV97ZihpKX0pLEleeyd9LCBcXGludF9TIFxcbnVfcyBkXFxtdSkiXSxbMywwLCIoKE1eeyd9X3soaSxzKX0pLEkgXFx0aW1lcyBTLCBmXnsnfVxcaW50X1MgXFxpb3RhX1NcXG51X3MgZFxcbXUpIl0sWzMsMiwiKChNXnsnfV97ZihpKSxzKX0sSV57J30gXFx0aW1lcyBTLCBcXGludF9TIFxcaW90YV9TXFxudV9zIGRcXG11KSJdLFswLDFdLFswLDJdLFsyLDNdLFsxLDNdXQ==
\[\begin{tikzcd}
	{((M_i),I, f\int_S \nu_s d\mu)} &&& {((M^{'}_{(i,s)}),I \times S, f^{'}\int_S \iota_S\nu_s d\mu)} \\
	\\
	{((M_{f(i)}),I^{'}, \int_S \nu_s d\mu)} &&& {((M^{'}_{(f(i),s)}),I^{'} \times S, \int_S \iota_S\nu_s d\mu)}
	\arrow[from=1-1, to=1-4]
	\arrow[from=1-1, to=3-1]
	\arrow[from=1-4, to=3-4]
	\arrow[from=3-1, to=3-4]
\end{tikzcd}\]

Here diagram above the arrows are the obvious ones between the objects (see how morphisms are defined in the category $TA$ \cite{hamad2025ultracategories})

So its image by the algebra functor commutes. The right diagram commutes by naturality of $a$. Hence diagram $7$ is commutative. Note that it's also possible to show commutativity of diagram $7$ just using Lurie's axiom (no need to use our results in \cite{hamad2025ultracategories}), we leave giving this alternative proof to the reader.

\section*{Appendix C: Equivalence of different notions of left ultrafunctors}

\label{Equaivalence of different notions of left ultrafunctors}

\addcontentsline{toc}{section}{\nameref{Equaivalence of different notions of left ultrafunctors}}
Suppose that we have two ultracategories $\ca{B}$ and $\ca{B}^{'}$, whose ultrastructure comes from \ref{not very crucial theorem}, we want to show that the original definition of left ultrafunctors as presented in \cite{lurie2018ultracategories} is equivalent to the new definition of left ultrafunctor that comes from the generalised ultrastructure.

We have already shown that there is an equivalence of categories between left ultrafunctors and lax morphisms of $T$-colax algebras, so it's enough to check the equivalence of the two notions of left ultrafunctors in the case of generalised ultracategories that inherit their generalised ultrastructures from being pseudo-algebras for the monad $T$.

    Suppose that  we have a left ultrafunctor $F$, we may define the maps $\zeta$, by $\zeta(f) = \sigma_{\mu} \circ F(f)$.

    We want to show that this data defines a left ultrafunctor between generalised ultracategories, in other  words, we want to show that the three diagrams of that definition are commutative

    Starting with the first diagram, let us add the morphism $\sigma_{\mu}$ to that diagram:

\[
\adjustbox{scale=0.56}{
    % https://q.uiver.app/#q=WzAsOCxbMCwwLCJGKEEpIl0sWzMsMCwiRihcXGludF9JTV9pZFxcbXUpIl0sWzcsMCwiXFxpbnRfSSBGKE1faSlkXFxtdSJdLFsxMCwwLCJcXGludF97SX0gRihcXGludF97WF9pfU5feyhpLHgpfWRcXGxhbWJkYV9pKWRcXG11Il0sWzE0LDAsIihcXGludF9JXFxpbnRfe1hfaX0oRihOX3soaSx4KX0pKSBkXFxsYW1iZGFfaSBkXFxtdSAiXSxbMTQsMywiXFxpbnRfe1xcY29wcm9kX3tJfVhfaSB9RihOX3soaSx4KX1kIFxcaW50X0kgXFxpb3RhX2kgXFxsYW1iZGFfaSBkXFxtdSkiXSxbMywzLCJGKCBcXGludF9JXFxpbnRfe1hfaX1OX3tpLHh9IGRcXGxhbWJkYV9pKWRcXG11KSJdLFs5LDMsIkYoXFxpbnRfe1xcY29wcm9kX3tJfVhfaX1OX3soaSx4KX1kXFxpbnRfe0l9XFxpb3RhX2lcXGxhbWJkYV9pZFxcbXUpIl0sWzAsMSwiRihmKSJdLFsxLDIsIlxcc2lnbWFfe1xcbXV9Il0sWzIsMywiXFxpbnRfSSBGKGdfaSkgZCBcXG11Il0sWzMsNCwiXFxpbnRfSSAgXFxzaWdtYV97XFxsYW1iZGFfaX1kXFxtdSJdLFsxLDYsIlxcaW50X0kgZ19pIGRcXG11IiwyXSxbNiw3LCJGKGFeey0xfSkiLDJdLFs0LDUsImFeey0xfSJdLFs3LDUsIlxcc2lnbWFfe1xcaW50X3tJfVxcaW90YV9pXFxsYW1iZGFfaWRcXG11fSIsMl0sWzYsMywiXFxzaWdtYV97XFxtdX0iLDFdXQ==
\begin{tikzcd}
	{F(A)} &&& {F(\int_IM_id\mu)} &&&& {\int_I F(M_i)d\mu} &&& {\int_{I} F(\int_{X_i}N_{(i,x)}d\lambda_i)d\mu} &&&& {(\int_I\int_{X_i}(F(N_{(i,x)})) d\lambda_i d\mu } \\
	\\
	\\
	&&& {F( \int_I\int_{X_i}N_{i,x} d\lambda_i)d\mu)} &&&&&& {F(\int_{\coprod_{I}X_i}N_{(i,x)}d\int_{I}\iota_i\lambda_id\mu)} &&&&& {\int_{\coprod_{I}X_i }F(N_{(i,x)}d \int_I \iota_i \lambda_i d\mu)}
	\arrow["{F(f)}", from=1-1, to=1-4]
	\arrow["{\sigma_{\mu}}", from=1-4, to=1-8]
	\arrow["{\int_I g_i d\mu}"', from=1-4, to=4-4]
	\arrow["{\int_I F(g_i) d \mu}", from=1-8, to=1-11]
	\arrow["{\int_I  \sigma_{\lambda_i}d\mu}", from=1-11, to=1-15]
	\arrow["{a^{-1}}", from=1-15, to=4-15]
	\arrow["{\sigma_{\mu}}"{description}, from=4-4, to=1-11]
	\arrow["{F(a^{-1})}"', from=4-4, to=4-10]
	\arrow["{\sigma_{\int_{I}\iota_i\lambda_id\mu}}"', from=4-10, to=4-15]
\end{tikzcd}
}
\]

The left diagram is a naturality diagram, while the right one is part of the definition of left ultrafunctors.

Now for the second and third diagrams of the definition of  left ultrafunctors between generalised ultracategories, their commutativity follows directly from the fact that the families $\sigma_{\mu}$ constitute the data of lax-morphisms of $T$-colax algebras.

Now on the other hand suppose that we have a left ultrafunctor between the generalised ultrastructures, we recover the usual left ultrastructure by defining $\sigma_{\mu}=\zeta(\la{id}_{\int_I M_i d\mu})$. We have already shown that such a left ultrafunctor defines a functor between the underlying categories.

Before we continue, we remind that the functor structure we obtained was between the underlying categories of the generalised ultracategories. So if we want to obtain functors between the initial category we must compose with the epsilon maps (those maps providing the isomorphism of categories) i.e. we have $F(f) = \epsilon_{*,*} \circ \mathbf{F}(f \circ \epsilon^{-1}_{*,*}) $. Here $F$ is functor structure between the initial categories and $\mathbf{F}(\epsilon_{*,*}^{-1} \circ f) = \zeta( \epsilon_{*,*}^{-1} \circ f )$ is the functor structure between the underlying ultracategories.

Now we need to verify that this indeed gives a lax morphism of  pseudo-algebras which is equivalent to show that the following diagram is commutative:

\[
\adjustbox{max width =\textwidth}{
% https://q.uiver.app/#q=WzAsNSxbMTEsMCwiKFxcaW50X0lcXGludF97WF9pfShGKE5feyhpLHgpfSkpIGRcXGxhbWJkYV9pIGRcXG11ICJdLFsxMSwzLCJcXGludF97XFxjb3Byb2Rfe0l9WF9pIH1GKE5feyhpLHgpfWQgXFxpbnRfSSBcXGlvdGFfaSBcXGxhbWJkYV9pIGRcXG11KSJdLFswLDMsIkYoIFxcaW50X0lcXGludF97WF9pfU5fe2kseH0gZFxcbGFtYmRhX2kpZFxcbXUpIl0sWzYsMywiRihcXGludF97XFxjb3Byb2Rfe0l9WF9pfU5feyhpLHgpfWRcXGludF97SX1cXGlvdGFfaVxcbGFtYmRhX2lkXFxtdSkiXSxbMCwwLCJcXGludF97SX0gRihcXGludF97WF9pfU5feyhpLHgpfWRcXGxhbWJkYV9pKWRcXG11Il0sWzIsMywiRihhXnstMX0pIiwyXSxbMCwxLCJhXnstMX0iXSxbMywxLCJcXHNpZ21hX3tcXGludF97SX1cXGlvdGFfaVxcbGFtYmRhX2lkXFxtdX0iLDJdLFsyLDQsIlxcc2lnbWFfe1xcbXV9IiwxXSxbNCwwLCJcXGludF9JICBcXHNpZ21hX3tcXGxhbWJkYV9pfWRcXG11Il1d
\begin{tikzcd}
	{\int_{I} F(\int_{X_i}N_{(i,x)}d\lambda_i)d\mu} &&&&&&&&&&& {(\int_I\int_{X_i}(F(N_{(i,x)})) d\lambda_i d\mu } \\
	\\
	\\
	{F( \int_I\int_{X_i}N_{i,x} d\lambda_id\mu)} &&&&&& {F(\int_{\coprod_{I}X_i}N_{(i,x)}d\int_{I}\iota_i\lambda_id\mu)} &&&&& {\int_{\coprod_{I}X_i }F(N_{(i,x)}d \int_I \iota_i \lambda_i d\mu)}
	\arrow["{\int_I  \sigma_{\lambda_i}d\mu}", from=1-1, to=1-12]
	\arrow["{a^{-1}}", from=1-12, to=4-12]
	\arrow["{\sigma_{\mu}}"{description}, from=4-1, to=1-1]
	\arrow["{F(a^{-1})}"', from=4-1, to=4-7]
	\arrow["{\sigma_{\int_{I}\iota_i\lambda_id\mu}}"', from=4-7, to=4-12]
\end{tikzcd}
}
\]

But this is exactly the diagram we obtain in the first axiom of definition of generalised left ultrafunctors when we replace the maps $(g_i)$ and $f$ by the identity maps.

Th next thing we need to check is that these two processes are in fact inverses, starting with a left ultrafunctor  between ultracategories we have that  $\sigma_{\mu} = \sigma_{\mu} \circ F(\la{id})$. On the other hand suppose that we have a left ultrafunctor between generalised ultracategories, and $f \in \la{Hom}(A,\int_I M_id\mu)$, We need to show that   $\zeta(f)= \zeta(\la{id}) \circ F(f) = \zeta(\la{id}) \circ \zeta(f \circ \epsilon_{*,*} ) \circ \epsilon_{*,*}^{-1}$, or diagrammatically:

% https://q.uiver.app/#q=WzAsNCxbMCwwLCJGKEEpIl0sWzMsMCwiXFxpbnRfIEkgRihNX2kpIGRcXG11Il0sWzAsMywiXFxpbnRfKiBGKFxcaW50X0lNX2lkXFxtdSkgZCoiXSxbMywzLCJGKFxcaW50X0lNX2kgZFxcbXUpIl0sWzAsMSwiXFx6ZXRhKGYpIl0sWzIsMywiXFxlcHNpbG9uXnstMX0iXSxbMywxLCJcXHpldGEoXFxtYXRocm17aWR9KSIsMl0sWzAsMiwiXFx6ZXRhKFxcZXBzaWxvbiBcXGNpcmMgZikiLDJdXQ==
\[\begin{tikzcd}
	{F(A)} &&& {\int_ I F(M_i) d\mu} \\
	\\
	\\
	{\int_* F(\int_IM_id\mu) d*} &&& {F(\int_IM_i d\mu)}
	\arrow["{\zeta(f)}", from=1-1, to=1-4]
	\arrow["{\zeta(\epsilon \circ f)}"', from=1-1, to=4-1]
	\arrow["{\epsilon^{-1}}", from=4-1, to=4-4]
	\arrow["{\zeta(\mathrm{id})}"', from=4-4, to=1-4]
\end{tikzcd}\]

Now notice that using the first axiom of the definition, we get the commutativity of the following diagram:

\[\begin{tikzcd}
	{F(A)} &&& {\int_*F(\int_IM_id\mu)d*} &&& {\int_*(\int_IF(M_i)d\mu)d*} \\
	\\
	\\
	{}
	\arrow["{\zeta(\epsilon^{-1} \circ f)}", from=1-1, to=1-4]
	\arrow["{a \circ\zeta(a^{-1}\circ \  \epsilon^{-1} \circ f )}"', curve={height=30pt}, from=1-1, to=1-7]
	\arrow["{\int_*\zeta(\mathrm{id})d*}", from=1-4, to=1-7]
\end{tikzcd}\]

So inserting two  additional nodes to the second to last diagram, we get the following diagram:

% https://q.uiver.app/#q=WzAsNixbMCwwLCJGKEEpIl0sWzMsMCwiXFxpbnRfIEkgRihNX2kpIGRcXG11Il0sWzAsMywiXFxpbnRfKiBGKFxcaW50X0lNX2lkXFxtdSkgZCoiXSxbMywzLCJGKFxcaW50X0lNX2kgZFxcbXUpIl0sWzEsMiwiXFxpbnRfKiAoXFxpbnRfSUYoTV9pKWRcXG11KSBkKiJdLFsxLDEsIlxcaW50X3tcXGNvcHJvZF97Kn1JfUYoTV9pKSBkXFxtdSJdLFswLDEsIlxcemV0YShmKSJdLFsyLDMsIlxcZXBzaWxvbiJdLFszLDEsIlxcemV0YShcXG1hdGhybXtpZH0pIiwyXSxbMCwyLCJcXHpldGEoXFxlcHNpbG9uXnstMX0gXFxjaXJjIGYpIiwyXSxbMiw0LCJcXGludF8qIFxcemV0YShcXG1hdGhybXtpZH0pZCoiLDFdLFs0LDEsIlxcZXBzaWxvbiJdLFswLDUsIlxcemV0YShhXnstMX1cXGNpcmMgXFwgIFxcZXBzaWxvbl57LTF9IFxcY2lyYyBmICkiLDFdLFs1LDQsImEiLDFdLFs5LDQsIjIiLDIseyJzaG9ydGVuIjp7InNvdXJjZSI6MjB9LCJzdHlsZSI6eyJib2R5Ijp7Im5hbWUiOiJub25lIn0sImhlYWQiOnsibmFtZSI6Im5vbmUifX19XSxbNCw4LCIzIiwyLHsic2hvcnRlbiI6eyJ0YXJnZXQiOjIwfSwic3R5bGUiOnsiYm9keSI6eyJuYW1lIjoibm9uZSJ9LCJoZWFkIjp7Im5hbWUiOiJub25lIn19fV0sWzYsMTEsIjEiLDEseyJzaG9ydGVuIjp7InNvdXJjZSI6MjB9LCJzdHlsZSI6eyJib2R5Ijp7Im5hbWUiOiJub25lIn0sImhlYWQiOnsibmFtZSI6Im5vbmUifX19XV0=
\[\begin{tikzcd}
	{F(A)} &&& {\int_ I F(M_i) d\mu} \\
	& {\int_{\coprod_{*}I}F(M_i) d\mu} \\
	& {\int_* (\int_IF(M_i)d\mu) d*} \\
	{\int_* F(\int_IM_id\mu) d*} &&& {F(\int_IM_i d\mu)}
	\arrow[""{name=0, anchor=center, inner sep=0}, "{\zeta(f)}", from=1-1, to=1-4]
	\arrow["{\zeta(a^{-1}\circ \  \epsilon^{-1} \circ f )}"{description}, from=1-1, to=2-2]
	\arrow[""{name=1, anchor=center, inner sep=0}, "{\zeta(\epsilon^{-1} \circ f)}"', from=1-1, to=4-1]
	\arrow["a"{description}, from=2-2, to=3-2]
	\arrow[""{name=2, anchor=center, inner sep=0}, "\epsilon", from=3-2, to=1-4]
	\arrow["{\int_* \zeta(\mathrm{id})d*}"{description}, from=4-1, to=3-2]
	\arrow["\epsilon", from=4-1, to=4-4]
	\arrow[""{name=3, anchor=center, inner sep=0}, "{\zeta(\mathrm{id})}"', from=4-4, to=1-4]
	\arrow["1"{description}, draw=none, from=0, to=2]
	\arrow["2"', draw=none, from=1, to=3-2]
	\arrow["3"', draw=none, from=3-2, to=3]
\end{tikzcd}\]

Square $2$ commutes, while Square $3$ is commutative by naturality of $\epsilon$. Now Square  $1$ commutes by axiom $2$ of our definition of left ultrafunctors between generalised ultracategories, and in order to use it we must note that for any ultracategory we have that $\int_{\coprod_{*}I} B_i d\mu \xrightarrow{\epsilon \circ a} \int_I B_i \mu  $, is the inverse of the map $\Delta_{\mu, f_{I}}$, where $f_{I}$ is the isomorphism between $I$ and $\coprod_{*}I$.

So we have shown that the two notions of left ultrafunctor agree. Now we show that the two notions of natural transformation of left ultrafunctors also agree. Suppose that we have $\alpha$ a natural transformation of left ultrafunctors in the first sense between left ultrafunctors $F$ and $G$, then we can define the family of  maps $(\alpha^1)$ and $(\alpha^2)$ to be just representables.

Now, we need to check the commutativity of the five diagrams in the definition of natural transformation of left ultrafunctors, the commutativity of diagrams $1$ and $2$ are just naturality properties of the ultraproduct diagonal map $\Delta_{f,\mu}$ in the second category.

Now for diagrams $3$ and $4$, since the maps $\beta$ are defined by a series of compositions, then the question of commutativity can be reduced to a usual naturality square, and hence diagrams $3,4$ commute.

Diagram $5$ corresponds exactly to the definition of natural transformation of left ultrafunctors (Lurie's definition \cite[Definition 1.4.2]{lurie2018ultracategories}).

Now suppose that we have $\alpha= (\alpha^1,\alpha^2)$ a natural transformation of left ultrafunctors in the generalised sense. We know that on the level of the underlying category of the generalised ultracategory  $\alpha$ is just an ordinary natural transformation, but it's good to have description on the level of the initial category structure. We claim that for every object $A$, there exists a map which we denote by $\alpha_A$ from $F(A)$ to $G(A)$ such that for any $f$ $\alpha_2(f) = f \circ \alpha_A$ and such that for any $g$ $\alpha_1(g) =\alpha_A \circ g$. The construction of $\alpha_A$ is easy, we define $\alpha_A =\epsilon^{-1} \circ \alpha_1(\zeta(\kappa_A)) =  \epsilon^{-1} \circ \alpha_2(\zeta^{'}(\kappa_A))$, this is just transposing the natural transformation from the functor structures on the underlying category of the generalised ultracategories associated to the ultrastructure to the functor structure between the initial category structures, so the family of maps $\alpha_A$ defines a natural transformation.
Now to check the condition we use the composition conditions (either diagram $3$ or $4$) and change of base condition (either diagram $1$ or $2$ would work).

It remains to check the defining condition of natural transformations of left ultrafunctors, which turns out to be exactly diagram $5$ of our definition in the case of $\kappa_{\int_I M_i d\mu} \in \la{Hom}(\int_I M_i d\mu, \int_I M_i d\mu)$.

\label{equivalence}

\section*{Appendix D: Definition of topology on the space $I_{\mu, (\lambda_i)}$}
\label{topology}
We want to show that the convergence relation we used in defining the topological space $I_{\mu, (\lambda_i)}$, really defines a topology towards this we show the definition satisfies the conditions $UQ1$ and $UQ4$ of \ref{topology theorem}.

In this section we are going to denote by $\forall_{\omega}$ the ultrafiltered universal quantifier, the statement $\forall_{\omega}j \  \phi(j)$, means that the set $\{j \mid \phi(j)\} \in \omega$. This quantifier has the nice property of being the negation of itself: $\neg(\forall_{\omega}j \ \phi(j)) \equiv \forall_{\omega}j \  \neg(\phi(j))$. 

Condition $UQ1$ of \ref{topology theorem} is automatically satisfied.

We want to show that our definition of topology satisfies the condition $UQ4$, in order to do this, let $J$ be a set, $\omega$ be an ultrafilter on a set $J$, and suppose that we have a map of sets $m$ from $J$ to the space $I_{\mu ,(\lambda_i)}$, and a family of ultrafilters $(\nu_j)$ such that $\nu_j$ converges to $m(j)$, we want to show that $\int_J \nu_j d\omega$ converges to $p_{\mu}$. Before that let us recall the definition of the space $I_{\mu,\lambda_i}$: As a set this is $ \coprod_{i \in I}X_i \bigsqcup \{p_\mu\} \bigsqcup \coprod_{i \in I} \{ p_{\lambda_i}\} $. We denoted by $\mu'$ and $\lambda_i'$ the pushforward of $\mu$, and $\lambda_i$ by the inclusion map of $I$  and $X_i$ respectively in $I_{\mu}$.

We have defined the convergence relations as follows:

\begin{itemize}
    \item Principal ultrafilters converge (at least) to their respective points.

    \item The ultrafilter $\lambda_i^{'}$  converges to $p_{\lambda_i}$.
    
    \item The ultrafilter $\mu^{'}$  converges to $p_{\mu}$.
    
    \item If $\mu$ is principal say $\mu = \delta_i$, we want $\lambda_i'$ not only to converge to $p_{\lambda_i}$, but also to $p_{\mu}$.

    \item If $\lambda_i$ is principal say at $x$, we want the principal ultrafilter $\delta_x$ to converge to $p_{\lambda_i}$ , moreover if $\mu$ is also principal at $i$, then we also want $\delta_x$ to converge to $p_{\mu}$.
    
    \item the ultrafilter $\int_{I} \lambda^{'}_i d \mu$ converges to $p_\mu$.
\end{itemize}

There are  six cases to consider:

\paragraph{ Case $1$: $m  \omega = \delta_{p_{\mu}}$}

If $m\omega = \delta_{p_{\mu}}$, then $m\omega$ converges to $p_{\mu}$ and $\forall_{\omega} m(j) \in \{p_{\mu}\}$, since every $\nu_{j}$ converges to $m(j)$, that means that we have $5$ different possibilities \begin{enumerate}
    \item $\forall_{\omega}j \ \nu_j = \delta_{p_{\mu}}$
    \item $\forallw j \  \nu_j = \mu^{'}$
    \item $\forallw j \ \nu_j = \int_I \iota_i \lambda_i d\mu$.
    \item $\mu$ is a principal ultrafilter at $i$, and $\forall_{\omega}j$ $\nu_j$ is principal at $p_{\lambda_i}$

    \item $\mu$ is principal at $i$ and $\lambda_i$ is principal at $x \in X_i$ and $\forall_{\omega}j$ $\nu_j$ is principal at $x$.

\end{enumerate}

In all these $\int_J \nu_j d\omega$ is the ultraproduct (in the sense of ultraproduct of ultrafilters) of a constant family which converges to $p_{\mu}$, which is the unique point for which the ultrafilter $\delta_{p_{\mu}}$ converges.

\paragraph{Case $2$: $m\omega= \delta_{p_{\lambda_i}}$}
We have the following cases: \begin{enumerate}
    \item $\forallw j \  \nu_j =\delta_{ p_{\lambda_j}}$
    \item $\forallw j \ \nu_j = \lambda^{'}_j$
\end{enumerate}

This case is similar to the one before with the exception that the $p\omega$ may converge to two points, but the resulting ultrafilter (which is also an ultraproduct of a constant family) will converge to the same family of points.

\paragraph{Case $3$: $m \omega =\delta_x$}

In this case, the only possibility is that $\forallw j \  \nu_j =\delta_x$ and hence $\int_J \nu_j d\omega =\delta_x$, hence they converge to the same set of points.

\paragraph{Case $4$: $m \omega = \lambda^{'}_i$ and $\lambda_i$ is non-principal}

This can only happen in the case where every $\nu_j$ is principal (since all the converging ultrafilters on $X_i$ are principal), in this case we get that $\int_{J} \nu_j d\omega = \lambda^{'}_i$ and hence converges to the same ultrafilters
\paragraph{ Case $5$: $m \omega = \mu^{'} $}

In this case we have two possibilities:

\begin{enumerate}
    \item $\forallw j  \ \nu_j$ is principal at  some $p_{\lambda_i}$
    \item $\forallw j \  \nu_j$ is not principal at some $p_{\lambda_i}$

\end{enumerate}

In the first case $\forallw j \nu_j$ is principal at some $p_{\lambda_i}$, such that the set $\forall_{\mu^{'}}i \delta_{p_{\lambda_{i}}} = \nu_j$ for some $j$ , so we get that $\int_J \nu_j d\omega= \mu^{'}$.

In the second case $\forallw j \nu_j$ is equal to $\lambda^{'}_i$ such that  $\forall_{\mu^{'}}i \lambda_i = \nu_j$, which in turn implies that $\int_J \nu_j d\omega =\int_I \lambda_i d\mu$, which converges to $p_{\mu}$, the same point which $\mu^{'}$ converges to.

This is the most important part of this proof, since this reveals the compositional nature of the space $I_{\mu,\lambda_i}$.

\paragraph{ Case $6$: $m \omega = \int_I \lambda^{'}_i d\mu $ and $\mu$ is non-principal}

Let's first say that in the notation above, we are identifying $I$ with its image by the map $i \mapsto \lambda_i$, in this case we have the following $\forallw j$ there exists $i$ and $x \in X_i$ such that $\nu_j = \delta_{x}$, such that $\forall_{\mu} i  $ $\forall_{\lambda_i} x \exists j$  $\nu_j =\delta_{x}$, this would imply that $\int_J \nu_j d\omega = \int_I \lambda^{'}_i d \mu=m\omega$, and hence they converge to the same set of points.

\end{document}